\documentclass{article}

\PassOptionsToPackage{numbers, compress}{natbib}

\usepackage[utf8]{inputenc} %
\usepackage[T1]{fontenc}    %
\usepackage[hidelinks]{hyperref}       %
\usepackage{url}            %
\usepackage{booktabs}       %
\usepackage{amsfonts}       %
\usepackage{nicefrac}       %
\usepackage{microtype}      %
\usepackage{cite}

\usepackage{amsmath, amsthm, amssymb}
\usepackage{mathtools}
\usepackage[parfill]{parskip}
\usepackage{soul}
\usepackage{multirow}
\usepackage{caption}
\usepackage{subcaption}

\usepackage{epsfig,fullpage,psfrag,url,verbatim}
\usepackage{graphics}
\usepackage{graphicx}
\usepackage{epstopdf}

\usepackage{amssymb}
\usepackage{amsthm}
\usepackage{url}
\usepackage[ruled,linesnumbered]{algorithm2e}
\usepackage{enumitem}
\usepackage{booktabs}
\usepackage{mathtools}
\usepackage{fullpage}
\usepackage{siunitx}

\newcommand{\R}{{\mathbb R}}

\newcommand{\N}{{\mathbb N}}

\newcommand{\eye}{\mbox{\bf I}}

\newcommand{\dist}{\mathop{\bf dist{}}}
\newcommand{\diam}{\mathop{\bf diam{}}}
\newcommand{\argmin}{\mathop{\rm argmin}}

\newcommand{\argmax}{\mathop{\rm argmax}}

\newcommand{\zBall}[2]{W_{#1}(#2)}

\newcommand{\ind}[1]{_{#1}}

\newcommand{\tz}{\hat{z}}
\newcommand{\ty}{\hat{y}}
\newcommand{\tx}{\hat{x}}
\newcommand{\hz}{\tilde{z}}
\newcommand{\hy}{\tilde{y}}
\newcommand{\hx}{\tilde{x}}
\newcommand{\bz}{\bar{z}}
\newcommand{\bx}{\bar{x}}
\newcommand{\by}{\bar{y}}
\newcommand{\pran}[1]{\left(#1\right)}
\newcommand{\bracket}[1]{\left\{ #1 \right\}}
\newcommand{\defeq}{\vcentcolon=}

\newcommand{\Span}[1]{\text{\emph{Span}}\left\{#1\right\}}

\usepackage{nicefrac}
\newcommand{\vectorr}[2]{\begin{pmatrix}  #1 \\ #2 \end{pmatrix}}
\newcommand{\twomatrix}[4]{\begin{pmatrix}  #1 & #2\\#3  & #4 \end{pmatrix}}

\usepackage{physics}

\newtheorem{theorem}{Theorem}
\newtheorem*{theorem*}{Theorem}
\newtheorem{lemma}{Lemma}
\newtheorem{fact}{Fact}
\newtheorem{definition}{Definition}

\newtheorem{remark}{Remark}
\newtheorem{proposition}{Proposition}

\newtheorem{corollary}{Corollary}

\newtheorem{property}{Property}

\usepackage{todonotes}

\SetKwFunction{GenericRestartScheme}{GenericRestartScheme}

\SetKwFunction{GenericFixedFrequencyRestartScheme}{GenericFixedFrequencyRestartScheme}

\SetKwFunction{GenericIterativeAlgorithm}{GenericIterativeAlgorithm}

\newcommand{\callGenericIterativeAlgorithm}[1]{\hyperref[function:GenericIterativeAlgorithm]{\GenericIterativeAlgorithm{#1}}}

\SetKwFunction{OneStepOfAlgorithm}{OneStepOfAlgorithm}
\SetKwFunction{InitializeAlgorithm}{InitializeAlgorithm}

\SetKwFunction{OneStepOfPDHG}{OneStepOfPDHG}
\SetKwFunction{InitializePDHG}{InitializePDHG}
\SetKwFunction{AdaptiveRestartPDHG}{AdaptiveRestartPDHG}

\SetKwFunction{OneStepOfAGD}{OneStepOfAGD}
\SetKwFunction{InitializeAGD}{InitializeAGD}
\SetKwFunction{AdaptiveRestartAGD}{AdaptiveRestartAGD}

\SetKwFunction{OneStepOfExtragradient}{OneStepOfExtragradient}
\SetKwFunction{InitializeExtragradient}{InitializeExtragradient}
\SetKwFunction{AdaptiveRestartExtragradient}{AdaptiveRestartExtragradient}
\newcommand{\PrimalDualStep}{\text{PrimalDualStep}}

\SetKwProg{Fn}{Function}{:}{}

\newcommand{\InitialTau}[1]{\tau_{\text{init}}}

\newcommand{\IterateSet}[0]{S}

\DeclarePairedDelimiter{\ceil}{\lceil}{\rceil}

\newcommand{\zind}[1]{\bar{z}^{#1}}

\newcommand{\len}[1]{\tau^{#1}}
\newcommand{\zStar}{z^{\star}}
\newcommand{\ZStar}{Z^{\star}}
\newcommand{\XStar}{X^{\star}}
\newcommand{\YStar}{Y^{\star}}
\newcommand{\xStar}{x^{\star}}
\newcommand{\yStar}{y^{\star}}
\newcommand{\fStar}{f^{\star}}
\newcommand{\tStar}{t^{\star}}
\newcommand{\tHatStar}{\hat{t}^{\star}}
\newcommand{\ReducePotentialBy}[0]{\beta}
\newcommand{\PEV}{\Gamma}
\newcommand{\Pev}{\gamma}
\newcommand{\sv}{\sigma}

\newcommand{\ii}{\mathbf{i}}

\newcommand{\Mat}[0]{A}
\newcommand{\U}[0]{U}
\newcommand{\V}[0]{V}

\newcommand{\Lip}[0]{L}
\newcommand{\Lag}[0]{\mathcal{L}}
\newcommand{\T}{\top}
\newcommand{\CT}{\dagger}

\newcommand{\sigmaMin}[1]{{\sigma_{\min}^+}{(#1)}}
\newcommand{\sigmaMax}[1]{{\sigma_{\max}}{(#1)}}

\newcommand{\blue}[1]{{#1}}

\usepackage{accents}

\begin{document}
\title{Faster First-Order Primal-Dual Methods for Linear Programming using Restarts and Sharpness}
\author{
David Applegate\thanks{Google Research (email: dapplegate@google.com)}
,
Oliver Hinder\thanks{Google Research, University of Pittsburgh (email: ohinder@pitt.edu)}
,
Haihao Lu\thanks{University of Chicago Booth School of Business (email: haihao.lu@chicagobooth.edu)}
,
Miles Lubin\thanks{Google Research}
}
\date{}
\maketitle

\begin{abstract}

First-order primal-dual methods are appealing for their low memory overhead, fast iterations, and effective parallelization. However, they are often slow at finding high accuracy solutions, which creates a barrier to their use in traditional linear programming (LP) applications. This paper exploits the sharpness of primal-dual formulations of LP instances to achieve linear convergence using restarts in a general setting that applies to ADMM (alternating direction method of multipliers), PDHG (primal-dual hybrid gradient method) and EGM (extragradient method). 
In the special case of PDHG, without restarts we show an iteration count lower bound of $\Omega(\kappa^2 \log(1/\epsilon))$, while with restarts we show an iteration count upper bound of $O(\kappa \log(1/\epsilon))$, where $\kappa$ is a condition number and $\epsilon$ is the desired accuracy. Moreover, the upper bound is optimal for a wide class of primal-dual methods, and applies to the strictly more general class of sharp primal-dual problems. We develop an adaptive restart scheme and verify that restarts significantly improve the ability of PDHG\blue{, EGM, and ADMM} to find high accuracy solutions to LP problems.
\end{abstract}

\vspace{0.2cm}

\section{Introduction}

Linear programming (LP) is a fundamental tool in operations research, with applications in transportation, scheduling, inventory control, revenue management \cite{charnes1954stepping, hanssmann1960linear, bowman1956production,manne1960linear,anderson2000hotel,liu2008choice}. 
An LP problem with $n$ variables and $m$ constraints in standard form is \begin{align*}
\begin{split}
    \min_{x\in\R^n} &~ c^{\T} x\\
    \text{subject to:} &~ \Mat x= b \\
    &~ x \ge 0 \ ,
\end{split}
\end{align*}
where $\R$ denotes the set of real numbers.

Currently, most practitioners solve LP problems with either simplex~\cite{dantzig1998linear} or interior-point methods~\cite{karmarkar1984new}. There are excellent reasons for this, namely that both methods typically find highly accurate solutions in a reasonable amount of time, and reliable implementations are widely available. While the two methods are quite different, both exactly solve linear system subproblems using factorizations (i.e., matrix inversion). These factorizations can be extremely fast, even for large problems. However, there are some drawbacks of relying on methods that use factorizations. Factorization speed heavily depends on the sparsity pattern of the linear systems, and in some examples, can be very slow. Moreover, for large enough problems, factorizations will run out of memory, even when the original problem data can fit into memory. Finally, factorizations are challenging to parallelize or distribute across multiple machines.

For this reason, there is a recent interest in developing methods for LP that forgo factorizations and instead use matrix-vector multiplies~\cite{ipm25gondzio2012,pmlr-v119-basu20a,applegate2021practical,li2020asymptotically,lin2021admm,basu2020eclipse}. Methods built on matrix-vector multiplication have many computationally appealing aspects. For example, they are readily parallelized or distributed across multiple machines. Moreover, their memory footprint is almost the same as the original problem data, and each matrix-vector multiply is reliably fast. Because of these fundamental differences, Nesterov~\cite{Nesterov2014huge} categorizes optimization methods that use factorizations as capable of handling \textit{medium-scale} problems, and those that instead use matrix-vector multiplies as capable of handling \textit{large-scale} problems.

A major reason why methods based on matrix-vector multiplies remain unsuitable replacements for factorization-based methods for LP is that they often struggle to obtain moderate to high accuracy solutions in a reasonable time-frame~\cite{ipm25gondzio2012}. This paper attempts to tackle this issue in an important setting: first-order primal-dual methods for LP. Primal-dual methods solve:
\begin{equation}\label{eq:poi-primal-dual}
    \min_{x\in X}\max_{y\in Y} \Lag(x,y) \ ,
\end{equation}
where $\Lag(x,y)$ is differentiable, convex in $x$ and concave in $y$, and $X,Y$ are closed convex sets. For notational convenience, we denote $z=(x,y)$ as the primal-dual variables and $Z=X\times Y$ as the feasible region.  Furthermore, denote 
$$F(z)= \begin{pmatrix}
\grad_x \Lag(x,y) \\
-\grad_y \Lag(x,y)
\end{pmatrix}$$
as the first-order derivatives of $\Lag(x,y)$. \blue{Throughout the paper, we assume the feasible region of \eqref{eq:poi-primal-dual} is non-empty and that \eqref{eq:poi-primal-dual} has a bounded optimal value, which are standard assumptions in the convex optimization literature~\cite{nesterov2013introductory}\footnote{\blue{A robust algorithm for LP would need to detect violations of this assumption, i.e., when the problem is infeasible or unbounded. 
We refer readers to  \cite{applegate2021infeasibility} to understand the behavior of primal-dual methods when applied to infeasible or unbounded LP problems.
}.}}.   Well-known primal-dual methods include proximal point method (PPM)~\cite{rockafellar1976monotone}, extragradient method (EGM)~\cite{korpelevich1976extragradient,tseng1995linear}, primal-dual hybrid gradient (PDHG)~\cite{chambolle2011first}, and alternating direction method of multipliers (ADMM)~\cite{douglas1956numerical}.
We can solve LP problems with primal-dual methods by writing them in the form:
\begin{align}\label{eq:LP-Lag}
    \min_{x\ge 0}\max_{y \in \R^{m}}  \Lag(x,y)= c^{\T} x + y^{\T} b-y^{\T} \Mat x \ .
\end{align}
To obtain high accuracy solutions for \eqref{eq:LP-Lag}, our key idea is to use restarts and sharpness. A restart feeds the output of an algorithm back in as input; a restart scheme will restart an algorithm whenever certain criteria are met. Sharpness presents a lower bound on the objective function growth in terms of the distance to the optimal solution set.
Restart schemes and sharpness conditions are well-studied in the context of unconstrained minimization \cite{o2015adaptive,roulet2020sharpness,yang2018rsg}.
In this context they can improve the theoretical and practical convergence of minimization algorithms to high accuracy solutions.
We analyze restarts in a different context: primal-dual methods for solving sharp primal-dual problems, such as those derived from LP, obtaining similar improvements.

\paragraph{Contributions} The contributions of the paper can be summarized as follow:
\begin{itemize}
\item We introduce a sharpness condition of primal-dual problems, and show LP problems are sharp.
\item We propose restarted algorithms for solving sharp primal-dual problems, and present their linear convergence rates. Their linear convergence rate matches the lower bound for a general class of primal-dual methods. 
\item We analyze primal-dual hybrid gradient without restarts and show that it is slower than its restarted variant on bilinear problems. 
\item We present an adaptive restart scheme that requires no hyperparameter searching over restart lengths. Numerical experiments showcase the effectiveness of our scheme.
\end{itemize}

\blue{While not a contribution of this paper, a subsequent computational study~\cite{applegate2021practical} evaluates a method named PDLP, which builds on top of the theory in the present paper by implementing a restarted version of PDHG with various additional enhancements. The study shows that adaptive restarts contribute significantly to PDLP's ability to solve a wide variety of benchmark LP problems to high levels of accuracy (see Section~C.2 in \cite{applegate2021practical}). \cite{applegate2021practical} includes an open source Julia prototype of PDLP\footnote{\url{https://github.com/google-research/FirstOrderLp.jl}}. More recently, a C++ version of PDLP was released within the open source OR-Tools library\footnote{\url{https://developers.google.com/optimization}}.}

There are two concepts we introduce to facilitate our implementation and analysis: the normalized duality gap and sharpness. 

\paragraph{Normalized duality gap.} The traditional metric for measuring solution quality for primal-dual problems is the primal-dual gap \cite{nemirovski2004prox,chambolle2011first} defined as 
\begin{flalign}\label{eq:primal-dual-gap}
\max_{\tz \in Z}\{ \Lag(x, \ty) - \Lag(\tx, y)\}
\end{flalign}
where for brevity we denote $(\tx, \ty) = \tz$. A solution to \eqref{eq:poi-primal-dual} has zero primal-dual gap. However, when the feasible region is unbounded, such as in LP \eqref{eq:LP-Lag}, this quantity can be infinite.
For this reason, for unbounded problems, we cannot provide convergence guarantees on this quantity.
This makes it a poor progress metric.
To overcome this issue we introduce the normalized duality gap:
\begin{subequations}\label{eq:rhorz}
\begin{flalign}\label{eq:rhorz-pos}
\rho_{r}(z) := \frac{\max_{\tz \in \zBall{r}{z}}\{ \Lag(x, \ty) - \Lag(\tx, y)\} }{r} \ ,
\end{flalign}
where $\zBall{r}{z} :=\{\tz \in Z \mid \|z-\tz\|\le r\}$ is the ball centered at $z$ with radius $r \in (0, \infty)$ intersected with the set $Z$, and $\| \cdot \|$ is a carefully selected semi-norm on $Z$. If $r = 0$ then for completeness we define\footnote{
\blue{The RHS of \eqref{eq:rhor-zero} is well-defined because lim sup always exists \cite[Section~5.3]{thomson2008elementary}. 
Later we show (Proposition~\ref{prop:basic-properties}) that $\rho_r(z)$ is monotonically
non-increasing in $r \in (0,\infty)$ for fixed $z$ which means that $\rho_{0}(z) = \limsup_{r \rightarrow 0^{+}} \rho_r(z) = \lim_{r \rightarrow 0^{+}} \rho_{r}(z) < \infty$.}
}:
\begin{flalign}\label{eq:rhor-zero}
\rho_0(z) \blue{\defeq} \limsup_{r \rightarrow 0^{+}} \rho_r(z).
\end{flalign}
\end{subequations}
Similar to primal-dual gap, we show the normalized duality gap attains zero if and only if it is evaluated at a solution to \eqref{eq:poi-primal-dual}. Unlike the primal-dual gap, if $\| \cdot \|$ is a norm, then the normalized duality gap $\rho_{r}(z)$ is always finite at any point $z \in Z$ with $r \in (0, \infty)$.

The normalized duality gap plays two roles in our work. The first is that it is used as a potential function in our restart scheme.
In particular, our adaptive restarts are triggered by constant factor reductions in the normalized duality gap.
Secondly, it forms the foundations of our analysis for restart schemes and is directly involved in our definition of sharpness for primal-dual problems.

\paragraph{Sharpness for primal-dual problems.} Sharpness, first introduced by Polyak~\cite{polyak1979sharp}, has been an important concept and a useful analytical tool in convex minimization. For a given function $f$, we say $f$ is $\alpha$-sharp if it holds for any $z\in Z$ that
\begin{equation}\label{eq:stand_sharp}
    f(z)-\fStar\ge \alpha \dist(z, \ZStar) \ ,
\end{equation}
where $\ZStar$ is the set of minima of the function $f$, $\fStar$ is the optimal function value, and $\dist(z, Z) := \inf_{\hat{z} \in Z} \| z - \hat{z} \|$.
When $f$ is Lipschitz continuous, convex and sharp, it is well-known that restarted sub-gradient descent obtains linear convergence. In contrast, without sharpness it obtains a slow sublinear rate~\cite{yang2018rsg}. In this work, we generalize sharpness from minimization problems to primal-dual problems \eqref{eq:poi-primal-dual}. Informally, we say a primal-dual problem is sharp if and only if the normalized duality gap $\rho_r(z)$ is sharp as per \eqref{eq:stand_sharp} with $f=\rho_r$. This condition is satisfied by LP problems, and the sharpness constant is proportional to the inverse of the Hoffman constant \cite{hoffman1952approximate} of the matrix that forms the KKT conditions.

\paragraph{Notation}
 Let $\N$ denote the set of natural numbers (starting from one).
 Let $\log(\cdot)$ and $\exp(\cdot)$ refer to the natural log and exponential function respectively.
 Let $\sigma_{\min}( \cdot )$, $\sigmaMin{\cdot}$, $\sigmaMax{\cdot}$ denote the minimum singular value, minimum nonzero singular value and maximum singular value of a matrix.
Let $\|z\|_2=\sqrt{\sum_j z_j^2}$ be the $\ell_2$ norm of $z$. For a bounded set $S \subseteq Z$,  $\diam( S )$ denotes the diameter of $S$, i.e., $\diam( S ) := \sup_{s, s' \in S} \| s - s' \|$.

The remainder of this section discusses related literature. In Section \ref{sec:unified-pda}, we present a unified viewpoint of the sublinear rate of classic primal-dual algorithms. In Section \ref{sec:sharpness}, we discuss the basic properties and examples of sharp primal-dual problems. In Section \ref{sec:restarts}, we present two restart schemes: fixed frequency restarts and adaptive restarts, and prove the linear convergence of both algorithms.
In Section~\ref{sec:tightness}, we consider the special case of bilinear problems. In this setting, we show restarts obtain optimal convergence rates and that standard PDHG is provably slower than restarted PDHG for solving bilinear problems. 
In Section \ref{sec:efficent_comp}, we present efficient algorithms to compute the normalized duality gap. In Section \ref{sec:experimental-results}, we present numerical experiments that showcase the effectiveness of our proposed restart schemes.

\subsection{Related Literature}

\textbf{Convex-convave primal-dual problems.} 
The study of convex-concave primal-dual problems has a long history.
Convex-concave primal-dual problems are a special case of monotone variational inequalities.
In his seminal work~\cite{rockafellar1976monotone}, Rockafellar studied PPM for solving monotone variational inequalities. Later on, Tseng~\cite{tseng1995linear} shows that both PPM and EGM have a linear convergence rate for solving variational inequalities when certain complicated conditions are satisfied, and these conditions are satisfied for solving strongly-convex-strongly-concave primal-dual problems. In 2004, Nemirovski~\cite{nemirovski2004prox} proposed the Mirror Prox algorithm (a special selection of the prox function recovers EGM), which first showed that EGM has $O(\frac{1}{\varepsilon})$ sub-linear convergence rate for solving convex-concave primal-dual problems over a bounded and compact set.

There are several works that study the special case of \eqref{eq:poi-primal-dual} when the primal-dual function has bilinear interaction terms, i.e., $\Lag(x,y)=f(x)+x^{\T} B y - g(y)$ where $f(\cdot)$ and $g(\cdot)$ are both convex functions. Algorithms for solving these bilinear interaction problems include Douglas-Rachford splitting (a special case is Alternating Direction Method of Multipliers (ADMM))~\cite{douglas1956numerical,eckstein1992douglas} and Primal-Dual Hybrid Gradient Method (PDHG)~\cite{chambolle2011first}.

The linear convergence of primal-dual methods is widely studied. Daskalakis et al. \cite{daskalakis2018training} studies the Optimistic Gradient Descent Ascent (OGDA), designed for training GANs, and shows that OGDA converges linearly for bilinear problems $\Lag(x,y)=x^{\T} A y$ with a full rank matrix $A$. Mokhtari et al. \cite{mokhtari2020unified} shows that OGDA and  EGM both approximate PPM (indeed, the fact that EGM is an approximation to PPM was first shown in Nemirovski's earlier work \cite{nemirovski2004prox}). They further show that these three algorithms have a linear convergence rate when $\Lag(x,y)$ is strongly-convex strongly-concave or when $\Lag(x,y)=x^{\T} A y$ is bilinear with a square and full rank matrix. Lu \cite{lu2020s} presents an ODE framework to analyze the dynamics of unconstrained primal-dual algorithms, yielding tight conditions under which different algorithms exhibit linear convergence, including bilinear problems with nonsingular $A$. Eckstein and Bertsekas \cite{eckstein1990alternating} use a variant of ADMM to solve LP problems and show the linear convergence of their method. More recently, \cite{liang2018local,lewis2018partial} show that many primal-dual algorithms, including PDHG and ADMM, have eventual linear convergence when applied to LP problems satisfying a non-degeneracy condition.
Very recently, for PDHG on LP problems,  linear convergence is established without requiring a non-degeneracy assumption \cite{alacaoglu2019convergence}.

However, not all linear convergence rates are equal. For example, it is well known that both gradient descent and accelerated gradient descent (AGD) obtain linear convergence on smooth, strongly convex functions. However, the gradient descent bound is $O(L / \theta \log(1/\epsilon))$ and accelerated gradient descent's bound is $O(\sqrt{L / \theta} \log(1/\epsilon))$ where $L$ is the smoothness and $\theta$ is strong convexity parameter \cite{nesterov2013introductory}.
Therefore, accelerated gradient descent has a better dependence on the condition number (i.e., the ratio between $\theta$ and $L$) by a square root factor.
This significant improvement is seen in practice too. 
A similar phenomena occurs for PDHG where we show for bilinear problems, the dependency on the condition number improves by a square root with introduction of restarts.
We also show that, in practice, restarts improve the performance of PDHG.
A more detailed discussion of this issue and related literature appears in Section~\ref{sec:tightness}.

\textbf{Sharpness conditions.}
The concept of sharpness of a minimization problem was first proposed by Polyak~\cite{polyak1979sharp} in 1970s. The original work by Polyak assumes the optimal solution set is a singleton. This work was then generalized by Burke and Ferris \cite{burke1993weak} to include the possibility of multiple optima. The early work on sharp minimization problems focused on analyzing the finite termination of certain algorithms, for example, projected gradient descent \cite{polyakbook}, proximal-point method \cite{ferris1991finite}, and Newton's method \cite{burke1995gauss}. Sharpness of monotone variational inequalities with bounded feasible regions was studied by Marcotte and Zhu \cite{marcotte1998weak}, and they focused on the finite convergence of a two-stage algorithm. Notice the primal-dual problem~\eqref{eq:poi-primal-dual} is a special case of a variational inequality (VI)~\cite{harker1990finite}. Indeed, when the feasible region is bounded, the sharpness of primal-dual problems we propose, translating to VI form, is equivalent to the sharpness of variational inequalities up to a constant. Our results generalize the concept for unbounded regions (when translating to VI). More recently, sharpness plays an important role for developing faster convergence rates for first-order methods. In particular,
linear convergence for subgradient (restarted) descent on sharp non-smooth functions~\cite{yang2018rsg} and sharp non-convex (weakly-convex) minimization~\cite{davis2018subgradient}. Finally, similar to strong convexity, sharpness provides a lower bound on the optimality gap in distance to optimality. Both sharpness and strong convexity can be viewed as error bound conditions with different parameters~\cite{roulet2020sharpness}.

\textbf{Restart schemes.}
Restarting is a powerful technique in continuous optimization that can improve the practical and theoretical convergence of a base algorithm~\cite{pokutta2020restarting}. The approach does not require modifications to the base algorithm, and thus is in general easy to incorporate within standard implementations. Recently, there have been extensive work on this technique, in smooth convex optimization~\cite{o2015adaptive,roulet2020sharpness}, non-smooth convex optimization~\cite{yang2018rsg,freund2018new}, and stochastic convex optimization~\cite{lin2015universal,johnson2013accelerating,tang2018rest}.
Many of the previous works require estimating algorithmic parameters such as the strong convexity constant.
This estimation often requires computationally intensive hyperparameter searches.
An important topic is to avoid this estimation by developing adaptive methods, which were well-studied for accelerated gradient descent \cite{o2015adaptive,giselsson2014monotonicity, nesterov2013gradient, lin2014adaptive, fercoq2019adaptive, alamo2019restart}.
In this paper, we present two restart schemes for primal-dual methods: fixed frequency restarts and adaptive restarts, where the former needs an estimation of the sharpness constant, while the latter does not.

\textbf{Hoffman constant.} The Hoffman constant~\cite{hoffman1952approximate}, upper bounds the ratio between the distance to a non-empty polyhedron and the corresponding constraint violation. 
More formally, for a given non-empty polyhedron $P =\{z \in \R^{n + m} \mid K z \ge h\}$, the Hoffman constant $H(K)$ satisfies
$\dist(z, P)\le H(K) \|(h-K z)^+\|$ for all $z$.
The Hoffman constant can also be viewed as the inverse of the sharpness constant of the function $f(z)=\|(h-Kz)^+\|$, namely,
$$f(z)-\fStar\ge \frac{1}{H(K)} \dist(z, \ZStar)\ ,$$ 
by noticing the minimal objective is $\fStar=0$ when the polyhedron is non-empty. Furthermore, there is a natural connection between the Hoffman constant and Renegar's condition number~\cite{renegar1995incorporating,renegar1995linear}, which measures the distance to ill-posedness of a system of linear inequalities. The Hoffman constant also characterizes the inherent difficulty to solve a system of linear inequalities, and plays an important role in establishing the convergence properties of many optimization methods~\cite{luo1993error,renegar1995linear,gutman2020condition,pena2020new}.

\textbf{Concurrent work.} After posting the first version of our manuscript, we noticed a concurrent and independent work \cite{fercoq2021quadratic}\blue{. This work} presents a new regularity condition of primal-dual problems called quadratic error bound of the smoothed gap (QEB). Furthermore, \cite{fercoq2021quadratic} shows that LP problems satisfy QEB, and restarted PDHG has linear convergence under QEB condition if the QEB parameter is known to the user. Our condition is different from the QEB condition as it involves the sharpness of the normalized duality gap. Contributions of our paper that do not appear in \cite{fercoq2021quadratic} include: (i) we present an adaptive scheme that achieves the optimal convergence rate without knowing the sharpness parameter; (ii) we provide a lower bound that proves our restarted algorithms achieve the optimal linear convergence rate under our sharpness condition, whereas it is unclear whether the algorithms presented in \cite{fercoq2021quadratic} are optimal in their setting; (iii) our analysis applies to a variety of primal-dual algorithms including PPM, EGM, ADMM and PDHG, whereas \cite{fercoq2021quadratic} focuses on PDHG. 
On the other hand, contributions of \cite{fercoq2021quadratic} that do not appear in our paper include that (i) our sharpness condition is tied to LP, while the QEB condition is possibly more general; (ii) we show the linear convergence of the current iterate only in the bilinear setting whereas \cite{fercoq2021quadratic} analyzes the linear convergence of the current iterate under the QEB condition.

\section{A unified viewpoint of the sublinear ergodic convergence rate of primal-dual algorithms}\label{sec:unified-pda}
\blue{While there have been recent works on studying the last iterations of primal-dual algorithms~\cite{condat2022distributed},
many previous convergence bounds for primal-dual problems \eqref{eq:poi-primal-dual} are with respect to the average over the algorithm's iterates~\cite{nemirovski2004prox,chambolle2016ergodic,he20121}.} 
This bound is known as the algorithm's ergodic convergence rate. 
In this section, we discuss a unified viewpoint of the sublinear ergodic convergence rate for classic primal-dual algorithms, which we will later use to develop the analysis for our restart schemes.

We consider a generic class of primal-dual algorithms for primal-dual problems \blue{starting from $t = 0$} with iterate update
\begin{equation}\label{eq:pda}
    z^{t+1}, \tz^{t+1} = \PrimalDualStep(z^t, \eta) \ ,
\end{equation}
where $z^t$ is the \emph{iterate solution} at time $t$, $z^{t+1}$ is the iterate solution at time $t+1$, $\tz^{t+1}$ is the \emph{target solution} at time $t+1$, and $\eta~\blue{ \in (0,\infty)}$ is the step size of the algorithm. As we will see later, the convergence guarantee of some primal-dual algorithms may not be on the iterate solution $z^t$; hence, we introduce the target solution $\tz^t$. 
The output of the algorithm is a function of target sequence $\tz^1,\ldots, \tz^t$, and for the standard ergodic rate it is the average of target solutions, namely, $\bz^t=\frac{1}{t}\sum_{j=1}^t \tz^i$. Here we discuss four classic primal-dual algorithms that can be represented as \eqref{eq:pda}:
\begin{itemize}
    \item \textbf{(Proximal point method)} Proximal point method (PPM)~\cite{rockafellar1976monotone} is a classic algorithm for solving \eqref{eq:poi-primal-dual}. The update can be written as:
    \begin{align}\label{eq:PPM-update}
        \tz^{t+1}=z^{t+1}=(x^{t+1},y^{t+1}) \in \arg\min_{x\in X}\max_{y\in Y} \Lag(x,y)+\frac{1}{2\eta} \|x-x^t\|_2^2-\frac{1}{2\eta} \|y-y^t\|_2^2\ .
    \end{align}
    
    \item \textbf{(Extragradient method)} Extragradient method~\cite{tseng1995linear} (a special case of mirror prox~\cite{nemirovski2004prox}) is another classic method for solving  \eqref{eq:poi-primal-dual}. It is an approximation of PPM~\cite{nemirovski2004prox}. The update can be written as:
     \begin{align}\label{eq:extra-gradient-update}
    \begin{split}
        \tz^{t+1}&\in\argmin_{z\in Z} \bracket{F(z^t)^{\T} z + \frac{1}{2\eta} \|z-z^t\|_2^2 } \\ z^{t+1}&\in\argmin_{z\in Z} \bracket{F(\tz^{t+1})^{\T} z + \frac{1}{2\eta} \|z-z^t\|_2^2} \ .
    \end{split}
    \end{align}
    \blue{For this method, we assume $F(z)$ is $\Lip$-Lipschitz continuous.}
    
    \item \textbf{(Primal-dual hybrid gradient)} Primal-dual hybrid gradient (also called Chambolle and Pock method) targets primal-dual problems with bilinear interaction term, namely, $\Lag(x,y)=f(x)+y^{\T} \Mat x - g(y)$, and it is widely used in image processing~\cite{chambolle2011first,chambolle2016ergodic}. The update can be written as\footnote{\blue{PDHG is often presented in a form with different primal and dual step sizes~\cite{chambolle2011first,chambolle2016ergodic}. Here, we choose to use the same primal and dual step size for consistent notation with other primal-dual algorithms. Our results can easily extend to the case of different step sizes by rescaling. 
    In particular, by setting $\eta = \sqrt{\sigma \tau}$ and defining a rescaled space: $( \hat{x} , \hat{y}) = ( x \sqrt{\eta / \tau}, y \sqrt{\eta / \sigma})$ where $\tau \in (0,\infty)$ is the desired primal step and $\sigma \in (0,\infty)$ is the desired dual step size. 
    Applying \eqref{eq:PDHG-update} to this rescaled space, i.e., replacing $f(x)$ with $\hat{f}(\hat{x}) = f(\hat{x} / \sqrt{\eta / \tau})$, $g(y)$ with $\hat{g}(\hat{x}) = g(\hat{y} / \sqrt{\eta / \sigma})$, $X$ with $\hat{X} = \{ x \sqrt{\eta / \tau} : x \in X \}$,  $Y$ with $\hat{Y} = \{ y \sqrt{\eta / \sigma} : y \in Y \}$, $\| x - x^t \|$ with $\| \hat{x}^t - x^t \|$, and $\| y - y^t \|$ with $\| \hat{y}^t - y^t \|$, then substituting back $(\hat{x},\hat{y}) = ( x \sqrt{\eta / \tau}, y \sqrt{\eta / \sigma})$ and $(\hat{x}^{t},\hat{y}^{t}) = ( x^t \sqrt{\eta / \tau}, y^t \sqrt{\eta / \sigma})$
    yields the classic PDHG update:
    \begin{flalign*}
   x^{t+1}&\in\argmin_{x\in X} f(x)+(y^t)^{\T} \Mat x+\frac{1}{2 \tau} \|x-x^t\|_2^2 \\ y^{t+1}&\in\argmin_{y\in Y} g(y)-y^{\T} \Mat (2 x^{t+1}-x^t) +\frac{1}{2 \sigma} \|y-y^t\|_2^2 \ .
    \end{flalign*}
    }}:
    \begin{align}\label{eq:PDHG-update}
    \begin{split}
        x^{t+1}&\in\argmin_{x\in X} f(x)+(y^t)^{\T} \Mat x+\frac{1}{2\eta} \|x-x^t\|_2^2 \\ y^{t+1}&\in\argmin_{y\in Y} g(y)-y^{\T} \Mat (2 x^{t+1}-x^t) +\frac{1}{2\eta} \|y-y^t\|_2^2  \\
        \tz^{t+1}&=z^{t+1} =(x^{t+1}, y^{t+1}) \ .
    \end{split}
    \end{align}
    \item \textbf{(Alternating direction method of multipliers)}
    Alternating direction method of multipliers (ADMM)~\cite{douglas1956numerical,eckstein1992douglas,boyd2011distributed,he20121} targets linearly constrained convex optimization problems with separable structure
    \begin{align*}
        \min_{x_U\in {X}_{U}, x_V\in {X}_{V}}\  & 
        \ \theta_1(x_U) + \theta_2(x_V) \\
        \text{s.t.}\  &\  U x_U + V x_V = q \ ,
    \end{align*}
    for which the Lagrangian becomes
    \begin{flalign}\label{eq:dr-lag}
    \min_{x\in X}\max_{y\in Y} \Lag(x,y) := \theta(x) - y^{\T} \begin{pmatrix} \U & \V \end{pmatrix} \begin{pmatrix}
    x_{\U} \\
    x_{\V}
    \end{pmatrix} + q^{\T} y \ ,
    \end{flalign}
    where $x=(x_{\U},x_{\V})$ , $\theta(x) = \theta_{\U}(x_{\U}) + \theta_{\V}(x_{\V})$, $X=X_U \times X_V$, and $Y=\R^m$. The update of ADMM is
    \begin{align}\label{eq:dr-update}
    \begin{split}
        x_{\U}^{t+1}&\in \argmin_{x_{\U}\in X_{\U}}\bracket{\theta_1(x_{\U})+\frac{\eta}{2} \left\|(U x_U+V x_V^t - q )-\frac{1}{\eta} y^t\right\|^2_2} \\
        x_{\V}^{t+1}&\in \argmin_{x_{\V}\in X_{\V}}\bracket{\theta_2(x_{\V})+\frac{\eta}{2} \left\|(U x_U^{t+1}+V x_V - q )-\frac{1}{\eta} y^t\right\|^2_2} \\
        y^{t+1}&=y^t-\eta(U x_U^{t+1}+V x_V^{t+1} - q) \\
        \tz^{t+1} &= (x_{\U}^{t+1}, x_{\V}^{t+1}, y^t-\eta(U x_U^{t+1}+V x_V^{t} - q)) \\
        z^{t+1}&=(x^{t+1}_U, x^{t+1}_V, y^{t+1})
    \end{split}
    \end{align}
\end{itemize}

Although the above four classic primal-dual algorithms have different update rules and may even target different primal-dual problems, it turns out they all satisfy the following assumption, which provides a unified viewpoint of primal-dual algorithms as well as a unified analysis on their ergodic rate.

\begin{property}\label{property:sufficient-decay}
Consider a primal-dual algorithm with iterate update \eqref{eq:pda}. There exists $ C>0$ and a semi-norm $\|\cdot\|$ such that it holds for any $t \in \N$  and $z\in Z$ that
\begin{equation}
\label{eq:decay}
\Lag(\tx^{t+1},y) - \Lag(x,\ty^{t+1}) \le \frac{C}{2}\|z-z^{t}\|^{2}- \frac{C}{2}\|z-z^{t+1}\|^2 \ .
\end{equation}
\end{property}
Property \ref{property:sufficient-decay} presents an upper bound on the primal-dual gap at the target solution $\tz^{t+1}$. Notice the RHS of \eqref{eq:decay} involves the iterate solution $z^t$ instead of the target solution $\tz^t$, which is intentional.
The next proposition shows how classic primal-dual algorithms satisfy Property \ref{property:sufficient-decay} as well as their corresponding constants $C$ and semi-norms $\|\cdot\|$. These are known results for individual algorithms; here we provide a unified viewpoint.

\begin{proposition}\label{prop:assumption-one-holds}
PPM, EGM, PDHG and ADMM satisfy Property \ref{property:sufficient-decay} with
\begin{itemize}
    \item {(PPM)} $C=\frac{1}{\eta}$ and Euclidean norm for any $\eta>0$;
    \item {(EGM)} $C=\frac{1}{\eta}$ and Euclidean norm for $0<\eta\le \frac{1}{\Lip}$;
    \item {(PDHG)} $C=\frac{1}{\eta}$ and $\|z\|^2=z^{\T} M z$ for $0<\eta\le \frac{1}{\sigmaMax{A}}$, where $M=\twomatrix{\eye}{-\eta \Mat^{\T}}{-\eta \Mat}{\eye}$;
    \item {(ADMM)} $C=1$ and $\|z\|^2=z^{\T} M z$ for $\eta>0$, where $M=\begin{pmatrix}
    0 & 0 & 0 \\
    0 & \eta \V^{\T} \V & 0 \\
    0 & 0 & \frac{1}{\eta} \eye
    \end{pmatrix}$. %
\end{itemize}
\end{proposition}

\begin{proof}
{(PPM)} It follows from the optimality condition of \eqref{eq:PPM-update} that for all $ z\in Z $,
$$
\left\langle F(z^{t+1}) + \frac{1}{\eta} (z^{t+1}-z^t), z^{t+1} - z \right\rangle \le 0 \ .
$$
After rearranging the above inequality, we obtain
$$
\left\langle F(z^{t+1}), z^{t+1} - z \right\rangle \le \frac{1}{2\eta} \|z-z^t\|^2 - \frac{1}{2\eta} \|z-z^{t+1}\|^2 - \frac{1}{2\eta} \|z^t-z^{t+1}\|^2\le \frac{1}{2\eta} \|z-z^t\|^2 - \frac{1}{2\eta} \|z-z^{t+1}\|^2\ .
$$
We finish the proof for PPM by noticing $\Lag(x^{t+1},y) - \Lag(x,y^{t+1})\le \left\langle F(z^{t+1}), z^{t+1} - z \right\rangle$.

{(EGM)} See, for example, the proof of Theorem 3.2 in \cite{nemirovski2004prox} (a simplified version of the proof can be found in Theorem 18.2 in \cite{he2016mirror}).

{(PDHG)} See, for example, Lemma 1 and the proof of Theorem 1 in \cite{chambolle2016ergodic}.

{(ADMM)} See, for example, Lemma 3.1 in \cite{he20121}.
\end{proof}

The next proposition presents how to obtain the ergodic sublinear convergence rate and non-expansiveness of primal-dual algorithms as a direct consequence of Property \ref{property:sufficient-decay}, which will be later used in developing the convergence results for our restart scheme. Again, these are known results for each individual algorithm.

\begin{proposition}\label{prop:convergence}
Consider the sequences $\{\tz^t, z^t\}$ generated from a generic primal-dual algorithm with iterate update \eqref{eq:pda} and initial solution $z^{0}$.  Denote $\zind{t} = \frac{1}{t}\sum_{i=1}^t \tz^i$ as the average target solution. Suppose the primal-dual algorithm satisfies Property \ref{property:sufficient-decay}, then it holds for any $\zStar\in \ZStar$ and $z\in Z$ that
\begin{enumerate}[label=\roman*.,ref={\ref{prop:convergence}.\roman*}]
    \item\label{prop:generic-nonexpansive} (Non-expansiveness) $\|z^{t+1}- \zStar\|\le \|z^{t}- \zStar\|$.
    \item\label{prop:sublinear-rate} (Sublinear rate on the primal-dual gap) $\Lag(\bx^t, y)-\Lag(x,\by^t)\le \frac{C \|z-z^{0}\|^2}{2 t}$.
\end{enumerate}
\end{proposition}
\begin{proof}
First we prove i.  By setting $z=\zStar$ in \eqref{eq:decay} and rearranging the inequality, it holds for any $t$ that
\begin{align*}
\|\zStar-z^{t+1}\|^2 & \le  \|\zStar-z^{t}\|^{2} + \frac{2}{C} \Lag(\xStar,\ty^{t+1})-\frac{2}{C} \Lag(\tx^{t+1},\yStar) \\
& =  \|\zStar-z^{t}\|^{2} + \frac{2}{C} \Lag(\xStar,\ty^{t+1})- \frac{2}{C} \Lag(\xStar,\yStar) + \frac{2}{C} \Lag(\xStar,\yStar) - \frac{2}{C} \Lag(\tx^{t+1},\yStar) \\
& \le \|\zStar-z^{t}\|^{2}\ ,
\end{align*}
where the second inequality is because $\Lag(\xStar,\yStar)=\min_x \Lag(x,\yStar)\le \Lag(x^{t+1},\yStar)$ and $\Lag(\xStar,\yStar)=\max_y \Lag(\xStar,y)\ge \Lag(\xStar,y^{t+1})$.

We now prove ii:
\begin{align*}
    \Lag(\bx^t, y)-\Lag(x,\by^t) &\le \frac{1}{t}\pran{\sum_{i=1}^t \Lag(\tx^i,y) - \Lag(x,\ty^i)} \le \frac{C}{2t} \pran{\|z-z^{0}\|^2-\|z-\zind{t}\|^2}\le \frac{C \|z-z^{0}\|^2}{2t} \ ,
\end{align*}
where the first inequality is from convexity-concavity of $\Lag(x,y)$ and Jensen's inequality, and the second inequality holds by summing up \eqref{eq:decay} over $t$ and telescoping.
\end{proof}

The non-expansiveness in Proposition \ref{prop:convergence} implies that $z^t$ stays in a bounded region. The next assumption assumes the target solution $\tz^t$ is not far away from either $z^t$ or $z^{t-1}$, thus $\tz^t$ also stays in a bounded region.
\begin{property}
\label{property:iterates-are-close}
There exists $q > 0$ such that for all $t \in \N$ either $\|\tz^t-z^t\|\le q \dist{(z^t, \ZStar)}$, or $\|\tz^t-z^{t-1}\|\le q \dist{(z^t, \ZStar)}$.
\end{property}

The four classic primal-dual algorithms we discussed above satisfy Property \ref{property:iterates-are-close}:
\begin{proposition}\label{prop:show-iterates-are-close}
Under the same conditions as Proposition~\ref{prop:assumption-one-holds}, PPM, EGM, PDHG and ADMM satisfy assumption \ref{property:iterates-are-close} since
\begin{itemize}
\item (PPM) $\tz^t = z^{t}$ for all $t \in \N$, and therefore $q = 0$;
\item (EGM) $\| \tz^{t} - z^{t-1} \| \le 3 \dist( z^t, \ZStar)$ for all $t \in \N$, and therefore $q = 3$;
\item (PDHG) $\tz^t = z^{t}$ for all $t \in \N$, and therefore $q = 0$;
\item (ADMM) $\| \tz^t - z^{t} \| \le 2 \dist(z^t, \ZStar)$ for all $t \in \N$, and therefore $q = 2$.
\end{itemize}

\end{proposition}

\begin{proof}
Note that the results for PPM and PDHG hold immediately.

\emph{Proof for EGM.}
The proof for EGM is somewhat technical so we defer it to Appendix~\ref{sec:proof-of:prop:show-iterates-are-close}.

\emph{Proof for ADMM.}
Let $\zStar = \min_{z \in \ZStar} \| z^{t} - z \|$. By \eqref{eq:dr-update} we have $\hat{x}^{t+1}_U = x^{t+1}_U$, $\hat{x}^{t+1}_V = x^{t+1}_V$, and $\hat{y}^{t+1} = y^{t+1} + \eta V (x_{V}^{t+1} - x_{V}^t)$. Therefore,
\begin{align*} \|\tz^{t+1}-z^{t+1}\| &= \sqrt{  (x_{V}^{t+1} - x_{V}^t)^{\T} {\eta V^{\T} V} (x_{V}^{t+1} - x_{V}^t) } \\ 
&= \|x_V^{t+1}-x_V^{t}\|_{\eta V^{\T} V} \\
&\le \|z^{t+1}-z^{t}\| \\
&\le \| z^{t+1} - \zStar \| + \| z^{t} - \zStar \| \\
&\le 2 \| z^{t} - \zStar \| \ ,
\end{align*}
where the last inequality uses Proposition~\ref{prop:generic-nonexpansive}.
\end{proof}

Property~\ref{property:restart-assumption} captures the minimal algorithmic properties required for restart schemes to achieve linear convergence, as shown in Section~\ref{sec:restarts}. This technical assumption essentially states that the primal-dual algorithm imposes an order $1/t$ sublinear convergence on the decay of the normalized duality gap (Property \ref{subproperty:reduce-potential-function}) and the iterates do not move too far away from the starting point (Property \ref{subproperty:distance-bound}).  Proposition~\ref{coro:weaker-assumption} demonstrates that if Property~\ref{property:sufficient-decay} and \ref{property:iterates-are-close} hold, then Property~\ref{property:restart-assumption} also holds (with different parameters $q$ and $C$), so Property~\ref{property:restart-assumption} holds for the classic primal-dual algorithms (PPM, EGM, PDHG and ADMM).

\begin{property}\label{property:restart-assumption} 
There exist scalar constants $q, C >0$ such that for all initial points $z^{0} \in Z$, the output sequence $\{\zind{t}\}_{t=1}^{\infty}$ generated by the algorithm satisfies $\zind{t} \in Z$ and
\begin{enumerate}[label={\roman*.},ref={\ref{property:restart-assumption}.\roman*}]
\item\label{subproperty:reduce-potential-function} $\rho_{\| \zind{t} - z^{0} \|} (\zind{t}) \le
\frac{2 C \| \zind{t} - z^{0} \|}{t}$,
\item\label{subproperty:distance-bound}
$\| \zind{t} - z^{0} \| \le (q + 2) \dist( z^{0}, \ZStar )$.
\end{enumerate}
\end{property}

\begin{proposition}\label{coro:weaker-assumption}
Consider the sequences $\{\tz^t, z^t\}$ generated from a generic primal-dual algorithm with iterate update \eqref{eq:pda} and initial solution $z^{0}$. Suppose the primal-dual algorithm satisfies Property \ref{property:sufficient-decay} and \ref{property:iterates-are-close}, then Property~\ref{property:restart-assumption} holds with $\bz_t=\frac{1}{t}\sum_{i=1}^t \tz_t$.
\end{proposition}

\begin{proof}
Suppose $r=\| \zind{t} - z^{0} \|>0$, then
$$
\rho_r(\zind{t}) = \frac{1}{r}\max_{\tz \in \zBall{r}{\zind{t}}} \{\Lag(\bx^t, \ty)-\Lag(\tx,\by^t)\} \le  \frac{1}{r} \max_{\tz \in \zBall{2 r}{z^{0}}}\{ \Lag(\bx^t, \ty)-\Lag(\tx,\by^t)\} \le \frac{C (2 r)^2 }{2 t r} = \frac{2 C r}{t} \ ,
$$
where the first inequality is from $\zBall{r}{\zind{t}} \subseteq \zBall{2 r}{z^{0}}$ by the definition of $r$, and the second inequality uses Proposition~\ref{prop:sublinear-rate}.
Alternatively, if $\| \zind{t} - z^{0} \| = 0$ then it suffices to observe that $\rho_r(\zind{t}) \le 2 C r / t$ and $\rho_r(z) \ge 0$ implies $\limsup_{r \rightarrow 0^{+}}\rho_r(\zind{t}) = 0$.
This establishes Property~\ref{subproperty:reduce-potential-function} . 

We turn to proving Property~\ref{subproperty:distance-bound}.
Let $\zStar \in \argmin_{z \in \ZStar } \| z - z^{0} \|$, then for either $i = t$ or $i = t - 1$,
$$\| \tz^{t} - z^{0} \| \le \| \tz^{t} - z^{i} \| +  \| z^{i} - \zStar \| + \| \zStar - z^{0} \| \le \| \tz^{t} - z^{i} \| + 2 \| \zStar - z^{0} \| \le  (q + 2) \| z^{0} -  \zStar \|\ ,$$ where the first inequality uses the triangle inequality, the second inequality uses Proposition~\ref{prop:generic-nonexpansive}, and the third inequality uses $\| \tz^{t} - z^{i} \| \le q \dist( z^{i}, \ZStar )$ as per Property~\ref{property:iterates-are-close}. From $\| \tz^{t} - z^{0} \| \le (q + 2) \| z^{0} -  \zStar \|$ it follows that
$$
\| \zind{t} - z^{0} \| = \left\| \frac{1}{t} \sum_{i=1}^{t} \tz^{i}  - z^{0}\right\| \le \frac{1}{t} \sum_{i=1}^{t} \| \tz^{i}  - z^{0}\| \le \frac{1}{t} \sum_{i=1}^{t} (q + 2) \| z^{0} -  \zStar \| = (q + 2) \dist(z^{0}, \ZStar) \ ,
$$
which establishes Property~\ref{subproperty:distance-bound}.
\end{proof}

\begin{table}[]
\begin{center}
\begin{tabular}{@{}c|cccc@{}}
\toprule
    & $C$ & valid $\eta$ values & $q$ & semi-norm squared \\ \midrule
PPM & $\frac{1}{\eta}$ & $(0,\infty)$ & $0$ & $z^{\T} z$ \\ \midrule
EGM & $\frac{1}{\eta}$ & $(0, 1/\Lip]$ & $3$ & $z^{\T} z$ \\ \midrule
PDHG          & $\frac{1}{\eta}$ & $(0,1/\sigmaMax{A}]$ & $0$ & $z^{\T} \twomatrix{\eye}{-\eta \Mat^{\T}}{-\eta \Mat}{\eye} z$ \\ \midrule
ADMM           & $1$ & $(0,\infty)$ & $2$ & $z^{\T} \begin{pmatrix}
    0 & 0 & 0 \\
    0 & \eta \V^{\T} \V & 0 \\
    0 & 0 & \frac{1}{\eta} \eye
    \end{pmatrix} z$ \\
\bottomrule
\end{tabular}
\caption{The corresponding constants and norms in Property~\ref{property:restart-assumption} and  for different algorithms. Recall $\Lip$ is the Lipschitz constant of $F$.
The values in this table can be found by inspecting Proposition~\ref{prop:assumption-one-holds}, Proposition~\ref{prop:show-iterates-are-close} and 
Proposition~\ref{coro:weaker-assumption}.}
\label{tbl:summary-of-C-and-q-values}
\end{center}
\end{table}

\section{Sharpness for primal-dual problems}\label{sec:sharpness}

This section discusses the basic properties of the normalized duality gap and sharp primal-dual problems, and then presents the concept of sharp primal-dual problems along with examples.

Critical to the paper is the concept of a sharp primal-dual problem. Indeed, when this condition holds, we are able to establish linear convergence using restarts (Section~\ref{sec:restarts}).

\begin{definition}\label{def:sharpness}
We say a primal-dual problem 
\eqref{eq:poi-primal-dual} is $\alpha$-sharp on the set $\IterateSet\subseteq Z$ if $\rho_{r}$ is $\alpha$-sharp (see \eqref{eq:stand_sharp}) on $S$ for all $r \in (0, \diam(S)]$, i.e., it holds for all $z \in \IterateSet$ that
\begin{equation}\label{eq:sharpness}
    \alpha \dist( z, \ZStar) \le \rho_{r}(z)\ .
\end{equation}
\end{definition}

As we can see, the sharpness of a primal-dual problem essentially says the localized duality gap $\rho_r(z)$ is sharp in the traditional sense. A key concept is that the sharpness is defined on a set $S\subseteq Z$. Even if the domain $Z$ is unbounded, the sharpness can be defined on a bounded set $S$.

From the definition of the normalized duality gap given in \eqref{eq:rhorz} and the sharpness condition, we can immediately establish the monotonicity of sharpness in $r$ and $S$ (Facts~\ref{fact:monotone-r-rho-r} and \ref{fact:monotone-sharpness-on-sets} and Proposition \ref{prop:basic-properties}).

\begin{fact} \label{fact:monotone-r-rho-r}
It holds for any $z$ that $r \rho_r(z)$ is monotonically non-decreasing for $r \in [0,\infty)$.
\end{fact}

\begin{fact}
\label{fact:monotone-sharpness-on-sets}
The sharpness condition is monotone in set $S$, namely, suppose $S_1 \subseteq S_2\subseteq Z$ and $\rho_r(z)$ is $\alpha$-sharp in $S_2$, then $\rho_r(z)$ is $\alpha$-sharp in $S_1$.
\end{fact}

\begin{proposition}\label{prop:basic-properties} 
 It holds for any fixed $z$ that $\rho_r(z)$ is monotonically non-increasing for $r \in [0,\infty)$. 
\end{proposition}

\begin{proof}
We prove the result for $r \in (0,\infty)$ which immediately implies the result also holds for $r = 0$ by definition of $\rho_0(z)$.
Consider any $z\in Z$ and $0<r_1<r_2$. Let $z_2\in\argmax_{ \tz \in \zBall{r_2}{z}} \bracket{\Lag(x, \ty) - \Lag(\tx, y)}$. Then we have that $z_1:=z+\frac{r_1}{r_2} (z_2-z)\in\zBall{r_1}{z}$. Moreover, it follows from convexity of $\Lag(x,y)$ over $x$ that $\Lag(x_1, y)\le \frac{r_1}{r_2} \Lag(x_2, y) + (1- \frac{r_1}{r_2}) \Lag(x,y)$, which after rearrangement becomes
$$
\Lag(x,y)-\Lag(x_2,y) \le \frac{r_2}{r_1}\pran{ \Lag(x,y)-\Lag(x_1, y)} \ .
$$
Similarly, it follows from concavity of $\Lag(x,y)$ over $y$ that
$$
\Lag(x,y_2)-\Lag(x,y) \le \frac{r_2}{r_1}\pran{ \Lag(x,y_1)-\Lag(x, y)} \ .
$$
Therefore, it holds that
\begin{align*}
    \rho_{r_2}(z)&=\frac{\Lag(x,y_2)-\Lag(x_2,y)}{r_2} \\ 
    &=\frac{\pran{\Lag(x,y_2)-\Lag(x,y)}+\pran{\Lag(x,y)-\Lag(x_2,y)}}{r_2} \\ 
    &\le \frac{\pran{\Lag(x,y_1)-\Lag(x,y)}+\pran{\Lag(x,y)-\Lag(x_1,y)}}{r_1} \\
    &\le \rho_{r_1}(z) \ ,
\end{align*}
which finishes the proof for the proposition.
\end{proof}

Proposition~\ref{prop:rho-duality-gap} shows that when the normalized duality gap is a reasonable metric to measure the quality of a solution: the normalized duality gap is zero if and only if the primal-dual gap is also zero.

\begin{proposition}\label{prop:rho-duality-gap}
For any $r \in [0,\infty)$,
the primal-dual gap \eqref{eq:primal-dual-gap} at $z$ is zero if and only if $\rho_r(z) = 0$.
\end{proposition}
\begin{proof}
Suppose that the primal-dual gap at $z$ is zero, then by definition of $\rho_r$ we get for all $r \in (0,\infty)$ that $0 \le \rho_r(z) \le \frac{1}{r} \max_{\tz \in Z}\{ \Lag(x, \ty) - \Lag(\tx, y)\} = 0$. Taking the limit as $r \rightarrow 0^{+}$ shows $\rho_0(z) = 0$.

For the other direction argue by contrapositive.
In particular, we assume the primal-dual gap at $z$ is nonzero, and using this assumption show that for all $r\in [0,\infty)$ that $\rho_r(z)>0$.
From the assumption that the primal-dual gap is nonzero,
there exists $\hat{z} \in Z$ such that $\Lag(x, \ty) - \Lag(\tx, y) > 0$. For any $r \in (0,\infty)$, setting $\alpha = \min\left\{1 , \frac{r}{\| \hat{z} - z \|} \right\}$ yields  $\| \alpha (\hat{z} - z) \| = \alpha \| \tz - z \| \le r$ and $z + \alpha (\tz - z) = (1 - \alpha) z + \alpha \tz \in Z$. Therefore
$z + \alpha (\hat{z} - z) \in \zBall{r}{z}$. By definition of $\hat{z}$, $\alpha > 0$, convex-concavity of $\Lag$, and $z + \alpha (\hat{z} - z) \in \zBall{r}{z}$ we get
$$
0 < \frac{\Lag(x, \ty) - \Lag(\tx, y)}{\| \hat{z} - z \|} \le \frac{\Lag(x, y + \alpha (\ty - y)) - \Lag( y + \alpha (\ty - y), y)}{\alpha \| \hat{z} - z \|} \le \frac{r}{\alpha \| \hat{z} - z \|} \rho_r(z)\ ,
$$
thus $\rho_r(z)>0$ for $r \in (0,\infty)$. Moreover, since $\rho_r(z) > 0$ for $r \in (0,\infty)$ we deduce $\rho_0(z) > 0$ by Proposition~\ref{prop:basic-properties}. This finishes the proof.
\end{proof}

We next present examples of sharp primal-dual problems. Here we utilize Euclidean norm in the sharpness definition to illustrate the examples. 
For any other norms, the sharpness condition still holds (upto a constant) by noticing that any norms in finite space are equivalent to each other (See, for example, Theorem 5.36 in \cite{hunter2001applied}). As an illustrating example, we show below that the corresponding norm for PDHG in Proposition 1 is approximately Euclidean (up to a constant):

\begin{proposition}
\label{fact:pdhg-norm}
Let $\eta \in (0, 1/\sigmaMax{A})$.
If $\| z \|^2 =  \| x \|_2^2 - 2\eta y^{\T} A x +  \| y \|_2^2$, which is the PDHG norm, then
$$
{(1 - \eta \sigmaMax{A})} \| z \|_2^2  \le \| z \|^2 \le  (1 + \eta \sigmaMax{A}) \| z \|_2^2\ .
$$
\end{proposition}
\begin{proof}
Lower bounding $\| z \|^2$ gives,
$\| z \|^2 \ge  \| x \|_2^2 -2 \eta \sigmaMax{A} \| y \|_2 \| x \|_2 + \| y \|_2^2 = {(1 - \eta \sigmaMax{A})} \| z \|_2^2 + \eta \sigmaMax{A} (\| x \|_2 -  \| y \|_2)^2 \ge  {(1 - \eta \sigmaMax{A})} \| z \|_2^2$.
Similarly, we have the other side.
\end{proof}

\subsection{Sharpness of primal-dual problems with bounded feasible regions}

\begin{lemma} Consider a primal-dual problem \eqref{eq:poi-primal-dual} with a bounded feasible set $Z = X \times Y$. Suppose the constraint set $Z$ has diameter $R$, and $Z$ is equipped with the Euclidean norm. Let $P(x)=\max_{y\in Y} \Lag(x,y)$ and $D(y)=\min_{x\in X} \Lag(x,y)$ denote the primal objective and the dual objective, respectively. If $P(x)$ is $\alpha_1$-sharp in $x$ and $D(y)$ is $\alpha_2$-sharp in $y$, then the primal-dual problem \eqref{eq:poi-primal-dual} is $\min\{\alpha_1,\alpha_2\}/R$-sharp on $S=Z$. 
\end{lemma}

\begin{proof}
Let $\XStar$ and $\YStar$ denote the minimizers of $P(x)$ and $D(x)$ respectively. Then it holds for any $z\in Z$ and $r\le R$ that
\begin{flalign*}
    \rho_r(z) &\ge \rho_R(z) = \frac{\max_{\tz \in Z}\{ \Lag(x, \ty) - \Lag(\tx, y)\} }{R} = \frac{P(x)-D(y)}{R} \\
    &\ge \frac{\alpha_1 \dist(x,\XStar)+\alpha_2 \dist(y,\YStar)}{R} \ge \frac{\min\{\alpha_1,\alpha_2\}}{R} \dist(z, \ZStar) \ ,
\end{flalign*}
where the first inequality utilizes Proposition \ref{prop:basic-properties}, the first equality is from the definition of $P(x)$ and $D(y)$, the second inequality uses sharpness, and the third inequality uses the triangle inequality.
\end{proof}

\subsection{Sharpness of bilinear problems}

We will find the following well-known Lemma useful. The proof of Lemma~\ref{prop:singular-value} appears in the Appendix~\ref{sec:proof-of:lem:minimum-nonsingular-value} for completeness.

\begin{lemma}\label{prop:singular-value}
Suppose $\| \cdot \|$ is the Euclidean norm.
Let $n$ and $m$ be positive integers, $H \in \R^{m \times n}$ and $h \in \R^{m}$. Define $S := \{ z \in \R^{n} : H z = h \}$. If $S \neq \emptyset$, then 
$\dist(S, z) \le  \frac{1}{\sigmaMin{H}} \| H z - h \|$.
\end{lemma}

\begin{lemma}\label{example:bilinear} Suppose $\Lag(x,y)= c^{\T} x - y^{\T} \Mat x + b^{\T} y$, $Z=\R^{m+n}$ and $Z$ is equipped with the Euclidean norm. Suppose there exists a finite solution to \eqref{eq:poi-primal-dual}. Then the primal-dual problem \eqref{eq:poi-primal-dual} is $\alpha$-sharp on $Z$ with $\alpha=\sigmaMin{\Mat}$.
\end{lemma}

\begin{proof}
Since $\Lag(x,y) = c^{\T} x - y^{\T} \Mat x + b^{\T} y$ it holds for any $z$ and $\tz$ that
\begin{align}\label{eq:bilinear-0}
\begin{split}
    \Lag(x, \ty)-\Lag(\tx,y)&=\Lag(x, \ty)-\Lag(x,y)+ \Lag(x,y)-\Lag(\tx,y) \\
    &=(-\nabla_y \Lag(x,y)) ^{\T} (y-\ty) + \nabla_x \Lag(x,y) ^{\T} (x-\tx) \\
    &= F(z)^{\T} (z-\tz) \ .
\end{split}
\end{align}
Therefore, it holds that 
\begin{equation}\label{eq:bilinear-1}
    \rho_r(z)=\frac{1}{r}\max_{z' \in \zBall{r}{z}} F(z)^{\T} (\tz-z) = \|F(z)\| \ .
\end{equation}

Meanwhile, notice $\ZStar=\{z \in Z \mid F(z)=0\}=\{z \in Z \mid Hz=h\}$, where $H=\twomatrix{}{\Mat^{\T}}{\Mat}{}$ and $h=\vectorr{c}{b}$. It then follows from Lemma \ref{prop:singular-value} that
\begin{equation}\label{eq:bilinear-2}
    \|F(z)\|^2 = \|Hz-h\|^2 = \| A^\T y - c \|^2 + \| A x - b \|^2  \ge \sigmaMin{\Mat}^2 \dist(z, \ZStar)^2 \ .
\end{equation}
Combining \eqref{eq:bilinear-1} and \eqref{eq:bilinear-2}, we arrive at
\begin{align*}
    \rho_r(z) = \|F(z)\|\ge \sigmaMin{\Mat} \dist(z,\ZStar) \ ,
\end{align*}
which finishes the proof.
\end{proof}
\begin{remark}
In the above bilinear example, we see that the primal-dual problem \eqref{eq:poi-primal-dual} is $\alpha$-sharp if and only if $\|F(z)\|$ is $\alpha$-sharp in the standard sense. Indeed, $\|F(z)\|$ being a sharp function is a necessary condition for $\Lag(x,y)$ to be sharp for any unconstrained convex-concave primal-dual problem. This can be seen by noticing
\begin{align*}
\rho_0(z) = \|F(z)\| \ ,
\end{align*}
and by utilizing the monotonicity of $\rho_r(z)$ stated in Proposition \ref{prop:basic-properties}.
\end{remark}

\subsection{Sharpness of standard linear programming}\label{sec:sharpness-lp}

Consider a generic LP problem:
\begin{align}\label{eq:LP-primal}
\begin{split}
    \min_{x \in \R^{n}} &~ c^{\T} x\\
    s.t. &~ \Mat x= b \\
    &~ x \ge 0 \ ,
\end{split}
\end{align}
its dual:
\begin{align}\label{eq:LP-dual}
\begin{split}
    \max_{y \in \R^{m}} &~ b^{\T} y\\
    s.t. &~ \Mat^{\T} y \ge c,
\end{split}
\end{align}
and its Lagrangian form:
\begin{align}\label{eq:LP-Lagrangian}
    \min_{x\ge 0}\max_{y \in \R^{m}}  \Lag(x,y)= c^{\T} x + y^{\T} b-y^{\T} \Mat x \ .
\end{align}
Suppose \eqref{eq:LP-primal} and \eqref{eq:LP-dual} have feasible solutions. By strong duality we can reduce \eqref{eq:LP-primal} and \eqref{eq:LP-dual} to solving the system $K z \ge h$ where
\begin{flalign}
\label{eq:define:K-and-h-for-lp}
K := \begin{pmatrix}
\eye & 0 \\
-\Mat & 0 \\
\Mat &0 \\
0 & -\Mat^{\T} \\
-c^{\T} & b^{\T}
\end{pmatrix}\ , \quad h := \begin{pmatrix}
0 \\
-b \\
b \\
-c \\
0
\end{pmatrix}.
\end{flalign}
The term $\| (h - K z)^{+} \|$ is a common metric for the termination of algorithms for LP \cite{andersen2000mosek}.
Lemma~\ref{lem:bound-kkt-error} shows that this metric can be bounded via the normalized duality gap.

\begin{lemma}
\label{lem:bound-kkt-error}
Let $\| \cdot \|$ be the Euclidean norm.
Consider the Lagrangian \eqref{eq:LP-Lagrangian}, formed from the LP setup. Suppose that there is a finite solution to \eqref{eq:LP-Lagrangian}.
Then, for all $R \in (0,\infty)$, $r \in (0, R]$, and $z \in \zBall{R}{0}$ we have $$
\| (h - K z)^{+} \| \le  \rho_r(z) \sqrt{1 + R^2}\ .
$$
\end{lemma}

\begin{proof}
Consider $z \in Z$. Let $v = \vectorr{(-c+\Mat^{\T}y)^+}{b - \Mat x }$ and note that $F(z) = \vectorr{c-\Mat^{\T}y}{\Mat x - b}.$
Recall from \eqref{eq:bilinear-0} that $\Lag(x, \ty)-\Lag(\tx,y)=F(z)^{\T} (z-\tz)$. Therefore, it follows from the definition of $\rho_r$ that for all $\tz \in \zBall{r}{z}$,
\begin{align}\label{eq:lp-0}
    \begin{split}
        \rho_r(z) & \ge \frac{\Lag(x,\ty)-\Lag(\tx, y)}{r} = \frac{F(z)^{\T} (z-\tz)}{r}  \ .
    \end{split}
\end{align}

Let $z_1 = z + r \frac{v}{\|v\|} $, then $\|z_1-z\|= r$. Meanwhile we have $(-c+\Mat^{\T}y)^+ \ge 0$, thus $z_1\in Z$ by noticing $z\in Z$. Therefore, $z_1\in \zBall{r}{z}$ and it follows from \eqref{eq:lp-0} that
\begin{align}\label{eq:lp-1}
    \begin{split}
        \rho_r(z) & \ge -\frac{1}{\|v\|} F(z)^{\T} v \ge \| v \|  \ ,
    \end{split}
\end{align}
where the last inequality utilizes the definition of $v$.

If $z=0$, let $z_2=0\in \zBall{r}{z}$, then we have from \eqref{eq:lp-0} that
\begin{equation}\label{eq:lp-22}
    \rho_r(z)\ge 0 = c^{\T} x-b^{\T} y= \frac{1}{r} (c^{\T} x-b^{\T} y) \ .
\end{equation}
Otherwise, let $z_2= z - \min\left\{ \frac{r}{\|  z \|}, 1 \right\} z$, then $\|z_2-z\|\le \frac{r}{\|  z \|}\|z\|\le r$. Meanwhile, we have $x_2\ge x-x =0$, thus $z_2\in Z$. Therefore, $z_2 \in \zBall{r}{z}$ and it follows from \eqref{eq:lp-0} that
\begin{align}\label{eq:lp-2}
    \begin{split}
        \rho_r(z) & \ge  \frac{1}{r} \min\left\{ \frac{r}{\|  z \|}, 1 \right\} 
        F(z)^{\T}  z = \min\left\{ \frac{1}{\|  z \|}, \frac{1}{r} \right\} 
        \pran{c^{\T} x- b^{\T} y} \ .
    \end{split}
\end{align}
Taking the worst case bound in \eqref{eq:lp-22} and \eqref{eq:lp-2} by noticing that $0<r\le R$, $0\le \|z\|\le R$ and $\rho_r(z)\ge 0$ yields
\begin{align}\label{eq:lp-3}
    \begin{split}
        \rho_r(z) & \ge \frac{1}{R}
        (c^{\T} x- b^{\T} y)^+ \ .
    \end{split}
\end{align}
Combining \eqref{eq:lp-1} and \eqref{eq:lp-3}, we obtain
\begin{align*}
    (1 + R^2)\rho_r(z)^2 \ge ((c^{\T} x- b^{\T} y)^+)^2 + \|(-c+\Mat^{\T} y)^+\|^2 +  \| b-\Mat x \|^2 = \|(h-Kz)^+\|^2\ ,
\end{align*}
where the equality utilizes the fact $x \ge 0$ and the property of $\ell_2$ norm. Taking the square root finishes the proof.
\end{proof}

Let $H(K)$ be the Hoffman constant \cite{hoffman1952approximate} of the matrix $K$ in Euclidean norm, i.e., it holds for any $z \in \R^{n+m}$ that
\begin{equation}\label{eq:hoffman}
    \dist(z, \ZStar)\le H(K) \|(h-Kz)^+\| \ .
\end{equation}

\begin{remark}
A popular characterization of $H(\Mat)$ for linear inequalities is (see for example \cite{klatte1995error, guler1995approximations, pena2020new})
\begin{equation}\label{eq:hoffman-2}
    H(\Mat)= \max_{\substack{J \subseteq \{1,...,2m+2n+1\} \\ \Mat_J \emph{\text{ has full row rank}}}} \frac{1}{\min_{v\in \R_+^J, \|v\|=1}\|\Mat_J^{\T} v\|}\ ,
\end{equation}
where $\Mat_J$ is the matrix with the corresponding rows of $\Mat$ indexed by $J$. 
In terms of the singular values of $\Mat$, a slightly looser bound would be 
$$
H(\Mat)\ge \max_{\substack{J \subseteq \{1,...,2m+2n+1\} \\ \Mat_J \emph{\text{ has full row rank}}}} \frac{1}{\sigmaMin{\Mat_J}}\ ,
$$
which ignores the constraints $v\in \R_+^J$ in \eqref{eq:hoffman-2} and take advantage of the Euclidean norm.

\end{remark}

\begin{lemma}\label{example:lp}
\textbf{(Sharpness of linear programming)} 
Let $R \in (0, \infty)$ and assume \eqref{eq:LP-Lagrangian} has a solution. Then, the primal-dual problem \eqref{eq:LP-Lagrangian} is $\alpha$-sharp on the
set $\zBall{R}{{0}}$ where $\alpha=\frac{1}{H(K)\sqrt{1+4R^2}}$.
\end{lemma}

\begin{proof}
For all $r \in (0, \diam(\zBall{R}{0})]  \subseteq (0, 2R]$, by Lemma~\ref{lem:bound-kkt-error} and \eqref{eq:hoffman},
$$
\rho_r(z) \sqrt{1 + 4 R^2} \ge \| (h - K z)^{+} \| \ge \frac{1}{H(K)} \dist(z, \ZStar)\ .
$$
\end{proof}

\subsection{Sharpness of ADMM for linear programming}
Following Lemma~\ref{example:lp}, we can show different formulations of LP may also be sharp. In particular, one popular method for LP is using ADMM~\cite{o2016conic}.
Let $X_U=\{x \in \R^{n} \mid Ax=b\}$ and $X_V=\{x \in \R^{n} \mid x\ge 0\}$.
Consider the following form of LP:
\begin{align*}
    \min_{x_U\in X_U, x_V\in X_V}\  & 
    \ c^{\T} x_V \\
    \text{s.t.}\  &\  x_U -  x_V = 0 \ ,
\end{align*}
and its Lagrangian form
\begin{flalign}\label{eq:lp-ADMM}
\min_{x_U\in X_U, x_V\in X_V}\max_{y\in\R^m} \Lag(x,y) = c^{\T} x_{V} - y^{\T} \begin{pmatrix} \eye & -\eye \end{pmatrix} \begin{pmatrix}
x_{\U} \\
x_{\V}
\end{pmatrix} \ .
\end{flalign}

\begin{lemma}\label{example:admm}
\textbf{(Sharpness of ADMM for linear programming)} 
Suppose \eqref{eq:lp-ADMM} has a solution.
Then there exists a matrix $K'$ and vector $h'$ such that the set of solutions to \eqref{eq:lp-ADMM} are equal to $\{ z : K' z \ge h' \}$.
Let $\zStar \in \ZStar$,
then it holds for any set $S(R, \zStar)=\{z \in Z : \| \zStar - z \| \le R\}$ that the primal-dual problem \eqref{eq:lp-ADMM} is $\alpha$-sharp where $\alpha=\frac{1}{ \max\{ \eta^2, 1/\eta^2 \} H(K')\sqrt{1+4R^2}}$.
\end{lemma}

The proof of Lemma~\ref{example:admm} appears in Section~\ref{app:admm-lp-proof} and follows the same structure as the proof of Lemma~\ref{example:lp}.

\section{Restart schemes for primal-dual methods}\label{sec:restarts}

In this section, we present a generic restart scheme for primal-dual methods and show its linear convergence rate for solving sharp primal-dual problems. We discuss two restart schemes: fixed frequency restart and adaptive restarts, where the former requires knowledge of the sharpness constant $\alpha$ while the latter does not.
\begin{algorithm}[]
\SetAlgoLined

 {\bf Input:} An initial solution $z^{0,0}$, a primal-dual algorithm \PrimalDualStep, a step-size $\eta$\;
 Initialize outer loop counter $n\gets 0$\;
 \Repeat{$z^{n,0}$ convergence}{
 \textbf{initialize the inner loop.} inner loop counter $t\gets 0$ \;
 \Repeat{one of the two restart conditions (see texts) holds}{
    $z^{n,t+1}, \tz^{n,t+1}\gets\PrimalDualStep(z^{n,t}, \eta)$\;
    $\bz^{n,t+1}\gets\frac{1}{t+1} \sum_{i=1}^{t+1} \tz^{n,i}= \frac{t}{t+1} \bz^{n,t}+\frac{1}{t+1} \tz^{n,t+1}$
    \label{line:average}\;  \label{line:output-is-average-of-iterates} 
    $t\gets t+1$\;
 }
  \textbf{restart the outer loop.} $\len{n}\gets t$, $z^{n+1,0}\gets \bz^{n,\len{n}}$, $n\gets n+1$\;
 }
 {\bf Output:} $z^{n,0}$.
 \caption{Restarted Primal-Dual Algorithms}
 \label{al:restarted+algorithm}
\end{algorithm}

Algorithm \ref{al:restarted+algorithm} presents our nested-loop restarted primal-dual algorithm. We initialize with $z^{0,0}\in Z$, a suitable primal-dual algorithm, and a step-size of the algorithm $\eta$. At each outer loop, we keep running $\PrimalDualStep$ until one of the restart conditions holds (to be discussed later). More specifically,
at the $t$-th inner loop of the $n$-th outer loop, we call \PrimalDualStep \ to update the solution $z^{n,t}$ and keep track of the target solution $\tz^{n,t}$ as well as the running average $\bz^{n,t}$. At the end of each outer loop, we restart the next outer loop from the output of a primal-dual algorithm, namely, the running average  $\bz^{n,\len{n}}$, and store the length of its inner loops as $\len{n}$. 
Here we propose two restart schemes:

\paragraph{Fixed frequency restarts.} Suppose we know the sharpness constant $\alpha$ of the primal-dual problem and an upper bound on $C$. In this scheme, we break the inner loop and restart the outer loop with a fixed frequency $\tStar$. Namely, we restart the algorithm if
\begin{flalign}\label{define:t-star}
t \ge \tStar := \ceil[\bigg]{ \frac{2 C ( q + 2)}{\alpha \ReducePotentialBy}}\ ,
\end{flalign}
where $C$ and $q$ are the corresponding parameters of a given primal-dual algorithm, as stated in Property \ref{property:sufficient-decay} and \ref{property:iterates-are-close}. The parameter $\ReducePotentialBy \in (0,1)$ controls when a restart is triggered. The value of $\beta$ can be tuned to improve both practical performance and the constants in the theoretical bound.
Concretely, it follows from Proposition \ref{prop:assumption-one-holds}, Proposition \ref{prop:show-iterates-are-close} and Proposition~\ref{coro:weaker-assumption} that
\begin{itemize}
    \item For PPM with $\eta \in (0,\infty)$, $\tStar = \ceil[\bigg]{ \frac{4}{\alpha \ReducePotentialBy \eta}}$;
    \item For PDHG with $\eta \in (0, 1/\sigmaMax{A}]$, $\tStar = \ceil[\bigg]{ \frac{4}{\alpha \ReducePotentialBy \eta }}$;
    \item For EGM with $\eta \in (0,1/\Lip]$, $\tStar = \ceil[\bigg]{ \frac{10}{\alpha \ReducePotentialBy \eta}}$;
    \item For ADMM, $\tStar = \ceil[\bigg]{ \frac{8}{\alpha \ReducePotentialBy}}$.
\end{itemize}

\paragraph{Adaptive restarts.} In practice, it can be non-trivial to compute the sharpness constant $\alpha$ of a given primal-dual problem. Here we propose an adaptive restart scheme that does not need an estimate of $\alpha$. In this scheme, we restart when the normalized duality gap has sufficient decay; more specifically, we restart the algorithm if
\begin{flalign}
\label{eq:adaptive-restart-scheme}
\begin{cases}
\rho_{\| \zind{n,t} - z^{n, 0} \|}(\zind{n,t}) \le \ReducePotentialBy \rho_{\| z^{n, 0} - z^{n-1,0} \|}(z^{n, 0}) & \text{ if } n\ge 1\\
t \ge \len{0} & \text{ if } n = 0
\end{cases}
\end{flalign}
where the first restart interval length $\len{0} \in \N$ can be selected by the user as a hyperparameter of the algorithm.
In practice, one can pick $\len{0} = 1$ for simplicity. 

\medskip
Although the two restarting schemes \eqref{define:t-star} and \eqref{eq:adaptive-restart-scheme} look disconnected, it turns out the restart interval length in \eqref{eq:adaptive-restart-scheme} is upper bounded by $\tStar$, as shown later in Theorem~\ref{thm:adaptive}. Furthermore, although the parameter $\beta \in (0,1)$ seems to play different roles in fixed frequency restart scheme \eqref{define:t-star} and adaptive restart scheme \eqref{eq:adaptive-restart-scheme}, it determines the contraction of the distance to optimal solution set for both schemes.

\medskip

\medskip
The rest of this section presents the linear convergence rate of the two restart schemes stated above. We start from the fixed frequency restart scheme.

Proposition~\ref{fact:avoid-trivial-setting} is useful throughout this section to ignore awkward cases in our proofs when $\| \zind{n,t} - z^{n, 0} \| = 0$. For example, technically sharpness (Definition~\ref{def:sharpness}) is only defined for $r > 0$.

\begin{proposition}\label{fact:avoid-trivial-setting}
If Property~\ref{property:restart-assumption} holds and $\| \zind{n,t} - z^{n, 0} \| = 0$ then $\zind{n,t} = z^{n, 0} \in \ZStar$.
\end{proposition}

\begin{proof}
Property~\ref{property:restart-assumption} implies $\rho_0(\zind{n,t}) = 0$ and therefore by Proposition~\ref{prop:rho-duality-gap} we get $\zind{n,t} \in \ZStar$.
\end{proof}

\begin{theorem}[Fixed frequency restarts]\label{thm:fixed-frequency}
Consider the sequence $\{ z^{n, 0} \}_{i=0}^{\infty}$, $\{ \len{n} \}_{i=1}^{\infty}$ generated by Algorithm~\ref{al:restarted+algorithm} with the fixed frequency restart scheme, namely, we restart the outer loop if \eqref{define:t-star} holds. Suppose the \eqref{eq:pda} satisfies Property \ref{property:restart-assumption}, and the primal-dual problem \eqref{eq:poi-primal-dual} is $\alpha$-sharp on the set $\zBall{R}{z^{0,0}}$ with $R = \frac{q + 2}{1 - \beta} \dist(z^{0,0},\ZStar)$. Then it holds for each outer iteration $n \in \N \cup \{ 0 \}$ that:
\begin{flalign}
\label{eq:fixed-frequency-result}
\dist(z^{n, 0}, \ZStar) \le \ReducePotentialBy^{n} \dist(z^{0,0}, \ZStar)\ .
\end{flalign}
\end{theorem}

\begin{proof}
We argue by induction. Note this hypothesis trivially holds for $n = 0$. Suppose the hypothesis
\eqref{eq:fixed-frequency-result} holds for all $n \le N$. Then it holds that
\begin{flalign*}
\| z^{N + 1, 0} - z^{0, 0} \| &\le \sum_{n=0}^N \| z^{n + 1, 0} - z^{n, 0} \| \le (q + 2) \sum_{n=0}^N \dist( z^{n, 0}, \ZStar ) \\
&\le  (q + 2)  \sum_{n=0}^N \beta^{n} \dist( z^{0, 0}, \ZStar ) \le  \frac{q+2}{1 - \beta} \dist( z^{0, 0}, \ZStar ) \ ,
\end{flalign*}
where the first inequality uses triangle inequality, the second inequality is from Property~\ref{subproperty:distance-bound}, the third inequality uses the induction hypothesis, and the last inequality is from bounds on sum of a geometric sequence. This implies $z^{N + 1, 0} \in \zBall{R}{z^{0,0}}$.
It then follows that 
\begin{align}\label{eq:fixed-restart-proof}
\begin{split}
\dist(z^{N+1,0}, \ZStar) &\le \frac{\rho_{\| z^{N+1,0} - z^{N,0} \|}(z^{N+1,0})}{\alpha} \le \frac{2 C \| z^{N+1,0} - z^{N,0} \|}{\alpha \tStar} \\
& \le \frac{2 C (q + 2)}{\alpha \tStar} \dist(z^{N,0}, \ZStar) \le  \beta \dist(z^{N,0}, \ZStar)  \le \beta^{N+1} \dist(z^{0,0}, \ZStar) \ ,
\end{split}
\end{align}
where the first inequality uses $\alpha$-sharpness of the primal-dual problem by choosing $r=\|z^{N+1,0}-z^{N,0}\|$ and $z=z^{N + 1, 0} \in \zBall{R}{z^{0,0}}$ in \eqref{eq:sharpness}, the second inequality utilizes Property~\ref{subproperty:reduce-potential-function}, the third inequality is from Property~\ref{subproperty:distance-bound} by noticing $z^{N+1,0}=\bz^{\tStar}$, and the fourth inequality comes from the definition of $\tStar$. This finishes the proof by induction.
\end{proof}

\begin{remark}
As a direct consequence of Theorem \ref{thm:fixed-frequency}, we need $O(\frac{C}{\alpha}\log(\frac{1}{\epsilon}))$ total iterations of PDHG to find an approximate solution $z$ such that the distance $\dist(z,Z^*)\le \epsilon$.
\end{remark}

\begin{remark}
We comment the proof logic of Theorem~\ref{thm:fixed-frequency}, and compare it with the standard linear convergence proofs for restart schemes in convex optimization. Consider restarted accelerated gradient descent to minimize a convex, $L$-smooth, and that grows $\mu$-quadratically\footnote{A function $f$ grows $\mu$-quadratically if ${f(z) - \fStar} \ge {\mu}\dist(z, \ZStar)^2$ .} function $f$. The standard analysis of restarted accelerated gradient descent can be written as (see, for example, \cite{nesterov2013gradient})
$$
\dist(z^{n+1, 0}, \ZStar)^2 \le \frac{f(z^{n+1, 0}) - \fStar}{\mu} \le \frac{2L \dist(z^{n+1, 0}, \ZStar)^2}{\mu (\tHatStar)^2} \le \beta^2 \dist( z^{n, 0}, \ZStar)^2 \ ,
$$
where the second inequality uses the $O(1/k^2)$ rate of accelerated gradient descent \cite{nesterov1983method,nesterov2013introductory} when the inner loop iteration is larger than $\tHatStar$, and the last inequality is from $\tHatStar = \frac{1}{\ReducePotentialBy} \sqrt{\frac{ 2L}{\mu}}$. This showcases the (accelerated) linear convergence rate of restarted accelerated gradient descent. Our analysis of Theorem \ref{thm:fixed-frequency} follows from a similar argument: sharpness allows us to transform the sublinear bound on the normalized duality gap $\rho_{\| z^{n+1, 0} - z^{n, 0} \|}(z^{n+1, 0})$ (of a classic primal-dual algorithm) to a constant factor contraction in the distance to optimality after $\tStar$ inner iterations. A major difference is that we are able to evaluate the normalized duality gap in our setting, while in convex optimization, it can be non-trivial to evaluate the optimality gap $f(z^{n+1,0})-\fStar$. Indeed, this feature is crucial as we develop the adaptive restarting scheme.
\end{remark}

\medskip
The next theorem presents the linear convergence rate of the adaptive restart scheme:
\begin{theorem}[Adaptive restarts]\label{thm:adaptive}
Consider the sequence $\{ z^{n, 0} \}_{i=0}^{\infty}$, $\{ \len{n} \}_{i=1}^{\infty}$ generated by Algorithm~\ref{al:restarted+algorithm} with the adaptive restart scheme, namely, we restart the outer loop if \eqref{eq:adaptive-restart-scheme} holds. Suppose \eqref{eq:pda} satisfies Property \ref{property:restart-assumption}. Suppose there exists a set $S\subseteq Z$ such that $z^{n,0}\in S$ for any $n\ge 0$, and the primal-dual problem \eqref{eq:poi-primal-dual} is $\alpha$-sharp on the set $S$. 
Then it holds for each outer iteration $n \in \N$ that
\begin{enumerate}[label=\roman*.]
    \item The restart length, $\len{n}$, is upper bounded by $\tStar$:
   \begin{equation}\label{eq:tau_n}
       \len{n} \le \tStar = \ceil[\bigg]{ \frac{2 C (q+2)}{\alpha \ReducePotentialBy}} \ ;
   \end{equation}
    \item The distance to primal-dual solution set decays linearly:
    $$
    \dist(z^{n, 0}, \ZStar) \le \ReducePotentialBy^{n} \frac{\tStar}{\len{0}}\dist(z^{0,0}, \ZStar)\ .
    $$
\end{enumerate}
\end{theorem}

\begin{proof}
We prove \eqref{eq:tau_n} by contradiction. If it is not satisfied, there exists $n\ge 1$ such that $\tau^n>\tStar$, and it holds for $t=\tStar$ that 
\begin{align*}
\rho_{\| \zind{n,\tStar} - z^{n, 0} \|}(\zind{n,\tStar}) &\le \frac{2 C \| \zind{n,\tStar} - z^{n,0} \|}{\tStar} \le \frac{2 C}{\alpha \tStar} \frac{\| \zind{n,\tStar} - z^{n,0} \|}{ \dist(z^{n,0}, \ZStar)} \rho_{\| z^{n,0} - z^{n-1,0} \|}(z^{n,0})   \\
&\le \frac{2 C (q + 2)}{\alpha \tStar} \rho_{\| z^{n,0} - z^{n-1,0} \|}(z^{n,0})   \le \ReducePotentialBy \rho_{ \| z^{n,0} - z^{n-1,0} \|}(z^{n,0}) \ , 
\end{align*}
where the first inequality utilizes Property \ref{subproperty:reduce-potential-function}, the second inequality follows from sharpness of the primal-dual problem \eqref{eq:sharpness} by choosing $r=\| z^{n,0} - z^{n-1,0} \|$ and $z=z^{n,0}\in S$, the third inequality utilizes Property~\ref{subproperty:distance-bound}, and the last inequality is from the definition of $\tStar$. 
This is a contradiction: the restart condition \eqref{eq:adaptive-restart-scheme} should have been triggered at $t = \tStar < \len{n}$. Therefore \eqref{eq:tau_n} holds.

For ii., it follows from sharpness condition that
\begin{flalign*}
\dist(z^{n,0}, \ZStar) &\le \frac{\rho_{\| z^{n,0} - z^{n-1,0} \| }(z^{n,0})}{\alpha}  \le \beta^{n-1} \frac{\rho_{\| z^{1,0} - z^{0,0} \|}(z^{1,0})}{\alpha}  \\
&\le \beta^{n-1} \frac{2 C \| z^{1,0} - z^{0,0} \|}{\alpha \len{0}} \le \beta^{n} \frac{2 C (q+2)}{\alpha \ReducePotentialBy \len{0}} \dist( z^{0,0}, \ZStar) \ ,
\end{flalign*}
where the second inequality recursively utilizes the restarting condition \eqref{eq:adaptive-restart-scheme}, the third inequality is from Property~\ref{subproperty:reduce-potential-function} and the last inequality utilizes Property~\ref{subproperty:distance-bound}.
\end{proof}

\begin{remark}[Choice of $\beta$]\label{remark:optimal-beta-choice}
From Theorem~\ref{thm:adaptive} we can obtain
\begin{flalign}\label{eq:alternative-adaptive-bound}
\sum_{i=0}^{n} \len{i} \le \len{0} + \frac{\tStar}{\log(1/\beta)} \left( \log\left( \frac{\tStar}{\len{0}} \frac{\dist(z^{0,0}, \ZStar)}{\dist(z^{n,0}, \ZStar)} \right) \right).
\end{flalign}
If we assume $\frac{\dist(z^{0,0}, \ZStar)}{\dist(z^{n,0}, \ZStar)} \gg \tStar / \len{0}$ and use $\tStar \approx \frac{2 C (q+2)}{\alpha \ReducePotentialBy}$ the right hand side of \eqref{eq:alternative-adaptive-bound} is approximately:
$$
\len{0} + \frac{\ReducePotentialBy}{\log(1/\ReducePotentialBy)} \frac{2 C (q+2)}{\alpha}  \log\left( \frac{\dist(z^{0,0}, \ZStar)}{\dist(z^{n,0}, \ZStar)} \right),
$$
optimizing this with respect to $\beta$ yields $\beta = \exp(-1)$.
\end{remark}

Compared with the assumptions for fixed frequency restarts (Theorem \ref{thm:fixed-frequency}), Theorem \ref{thm:adaptive} additionally requires that $z^{n,0}$ stays in a set $S\subseteq Z$, where the primal-dual problem is sharp. Suppose the feasible region $Z$ is bounded, we can simply choose $S=Z$. However, in the LP example (i.e., Lemma~\ref{example:lp}) where $Z$ can be unbounded, and the primal-dual problem is sharp only on a bounded region. 
Below we show that the iterates of restarted PDHG, PPM, EGM and ADMM stay in a bounded region:

\begin{proposition}\label{thm:pdhg-ppm:bound-adaptive-restart-scheme}
Consider the sequence $\{ z^{n, 0} \}_{i=0}^{\infty}$, $\{ \len{n} \}_{i=1}^{\infty}$ generated by Algorithm~\ref{al:restarted+algorithm} with proper step-size $\eta$ (see Table \ref{tbl:summary-of-C-and-q-values}) and $\beta \in (0,1)$. 
For restarted PDHG or restarted PPM, there exists a constant $R >0$ such that  $z^{n,0} \in \zBall{R}{z^{0,0}}$ for all $n \in \N$.
\end{proposition}

\begin{proof}
For PDHG and PPM, the proof uses the non-expansiveness property of $\tz^{t,n}=z^{t,n}$ (see Proposition \ref{prop:generic-nonexpansive}). In particular, let $\zStar \in \argmin_{z \in \ZStar} \| z - z^{0,0} \|$, then it holds for any $n, t$ that
$$
\| \zStar - \bar{z}^{n,t} \| = \left\| \frac{1}{t} \sum_{i=1}^{t} (\zStar - z^{n,t}) \right\| \le \frac{1}{t} \sum_{i=1}^{t} \| \zStar - z^{n,t} \| \le \frac{1}{t} \sum_{i=1}^{t} \| \zStar - z^{n,0} \| = \| \zStar - z^{n,0} \|\ ,
$$
where the first inequality is from triangle inequality and the second inequality uses Proposition~\ref{prop:generic-nonexpansive}. Thus, we have
$$\| \zStar - z^{n+1,0} \| = \| \zStar - \bar{z}^{n,\len{n}} \| \le  \| \zStar - z^{n,0} \| \ ,$$
whereby $\| \zStar - z^{n,0} \|\le \| \zStar - z^{0,0} \|$ by induction.
It then follows from the triangle inequality that $$\| z^{n, 0} - z^{0,0} \| \le \| z^{0,0} - \zStar \| + \|  \zStar - z^{n, 0} \| \le 2 \| z^{0,0} - \zStar \|=2 \dist(z^{0,0}, \ZStar) \ ,$$ which concludes the proof for PDHG and PPM with $R=2 \dist(z^{0,0}, \ZStar)$.
\end{proof}

\begin{proposition}\label{thm:admm-egm:bound-adaptive-restart-scheme}
Consider the sequence $\{ z^{n, 0} \}_{i=0}^{\infty}$, $\{ \len{n} \}_{i=1}^{\infty}$ generated by Algorithm~\ref{al:restarted+algorithm} with the adaptive restart scheme, namely, we restart the outer loop if \eqref{eq:adaptive-restart-scheme} holds. 
If $0 < \beta <  \frac{1}{q+3}$ then, there exists a constant $R >0$ such that  $z^{n,0} \in \zBall{R}{z^{0,0}}$ for all $n \in \N$ in the following settings:
\begin{itemize}
    \item restarted EGM applied to LP (Lemma~\ref{example:lp});
    \item restarted ADMM applied to LP (Lemma~\ref{example:admm}) .
\end{itemize}
\end{proposition}

The proof of Proposition \ref{thm:admm-egm:bound-adaptive-restart-scheme} is more technical than Proposition~\ref{thm:pdhg-ppm:bound-adaptive-restart-scheme} because $\tz^{t,n}=z^{t,n}$ no longer holds, and we defer it to Appendix~\ref{sec:proof-of-bound-adaptive-restart-scheme}.

\section{Tight convergence results for bilinear problems}\label{sec:tightness}

Bilinear problems are a special case of LP where there are no nonnegativity or inequality constraints. 
Moreover, in many situations the asymptotic behaviour of primal-dual methods applied to LP is equivalent to applying the method to a bilinear problem.
In particular, if after a certain number of iterations the set of active variables\footnote{The active variables are variables not at their bounds.} does not change, then the subset of the matrix, objective, and right hand side corresponding to this active set defines such a bilinear subproblem.
Recent results~\cite{liang2018local,lewis2018partial} show that for many primal-dual algorithms, including PDHG and ADMM, if the problem is non-degenerate\footnote{In the LP case \eqref{eq:LP-Lagrangian}, non-degeneracy means that the algorithm converges to a primal-dual solution that satisfies strict complimentary.} then after a finite number of iterations the active set freezes.  See~\cite{lewis2018partial} for a differential geometry argument for this phenomenon. Thus, the convergence of methods for bilinear problems can characterize their asymptotic behavior on nondegenerate LP problems.

Already we have some basic results in this setup. In particular, Lemma~\ref{example:bilinear}, and Proposition~\ref{fact:pdhg-norm}
imply that the primal-dual problem is $\Omega(\sigmaMin{A})$-sharp on bilinear problems with respect to the PDHG norm for $\eta \le 1/(2\sigmaMax{A})$. Combining the sharpness result with Table~\ref{tbl:summary-of-C-and-q-values}, and Theorem~\ref{thm:fixed-frequency} (or Theorem~\ref{thm:adaptive} with a small loss in the log term) implies that after
\begin{flalign*}
O\pran{\frac{\sigmaMax{A}}{ \sigmaMin{A}}\log\pran{\frac{1}{\epsilon}}}
\end{flalign*}
total matrix-vector multiplications restarted PDHG finds a point $z^{n, 0}$ satisfying
$$
\frac{\dist(z^{n, 0}, \ZStar)}{\dist(z^{0,0}, \ZStar)} \le \epsilon\ .
$$
The remainder of this section proceeds as follows. 
First we present a lower bound showing that, up to a constant factor, restarted PDHG yields the best convergence rates for solving bilinear problems.
Next, we show tight convergence bounds for the average and last iterate of PDHG which we use to demonstrate our restarted PDHG convergence rate represents a strict improvement over the latter two methods. 
Finally, we compare the performance of other methods from literature for solving bilinear problems.

\subsection{Lower bounds for primal-dual algorithms}\label{sec:lower-bounds}

In this subsection, we present a lower bound for solving a sharp primal-dual problem \eqref{eq:poi-primal-dual} demonstrating that restarting gives the best possible worst-case convergence bound for a large class of primal-dual methods. 
First, we review related lower bounds for first-order unconstrained function minimization.

\begin{definition}\label{def:span-respecting-unconstrained}
An algorithm is span-respecting for a minimization problem $\min_{x \in \R^{m}} f(x)$ if
\begin{flalign*}
x^{t} &\in x^{0} + \Span{\grad f(x^{i}) : \forall i \in \{1, \dots, t - 1 \}} 
\end{flalign*}
for all $t \in \N$.
\end{definition}

Span-respecting algorithms cover a large number of unconstrained optimization algorithms including gradient descent, conjugate gradient, and accelerated gradient descent.
As is well-known one can form a lower bound using convex quadratics for these types of algorithms
as given in Theorem~\ref{thm:classic-lb}.

\begin{theorem}[Theorem 2.1.12 of \cite{nesterov2013introductory}]\label{thm:classic-lb}
For all $\gamma_{\max} > \gamma_{\min} > 0$, $m \in \N$, there exists a positive definite matrix $H \in \R^{m \times m}$ and vector $h \in \R^{m}$ such that $\sigma_{\max}(H) = \gamma_{\max}$, $\sigma_{\min}(H) = \gamma_{\min}$, and any span-respecting unconstrained minimization algorithm for solving 
\begin{flalign}\label{eq:unconstrained-quadratic-min}
\min_{x \in \R^{m}} \frac{1}{2} x^\T H x + h^\T x.
\end{flalign}
satisfies for $t < m$ that
$$
\dist(x^{t}, \XStar ) \ge \left( 1 - \sqrt{\frac{\sigma_{\min}(H)}{\sigma_{\max}(H)}} \right)^{t} \dist(x^{0}, \XStar ) \ .
$$

\end{theorem}

Our goal is to supply a similar result to Theorem~\ref{thm:classic-lb} but in the primal-dual setting.

\begin{definition}\label{def:primal-dual-span-respecting}
An algorithm is span-respecting for a unconstrained primal-dual problem $\max_{y \in \R^{n}} \min_{x \in \R^{n}} \Lag(x, y)$ if:
\begin{flalign*}
x^{t} &\in x^{0} + \Span{\grad_{x} \Lag(x^{i},y^{j}) : \forall i , j \in \{1, \dots, t - 1 \}} \\
y^{t} &\in y^{0} + \Span{\grad_{y} \Lag(x^{i},y^{j}) : \forall i \in \{1, \dots, t \}, j \in \{1, \dots, t - 1 \}} \ .
\end{flalign*}
\end{definition}

Definition~\ref{def:primal-dual-span-respecting} is analogous to Definition~\ref{def:span-respecting-unconstrained} but in the primal-dual setting. 
If $\Lag$ is bilinear then
with appropriate indexing of their iterates, primal-dual algorithms including primal-dual hybrid gradient, extragradient and their restarted variants satisfy 
Definition~\ref{def:primal-dual-span-respecting}.
Corollary~\ref{coro:primal-dual-lb} provides a lower bound on span-respecting primal-dual algorithms. The proof is by reduction to Theorem~\ref{thm:classic-lb}.

\begin{corollary}\label{coro:primal-dual-lb}
Let $\| \cdot \|$ be the Euclidean norm.
For all $\alpha, C \in (0, \infty)$, $m \in \N$,
there exists a matrix $A \in \R^{m \times m}$ and vector $b \in \R^{m}$ such that  
$$
\max_{y \in \R^{m}} \min_{x \in \R^{m}} \Lag(x,y) = y^\T A x + b^\T y
$$
is $\alpha$-sharp on $\R^{m \times m}$, $C$-smooth and any span-respecting unconstrained primal-dual algorithm satisfies
$$
\dist(z^{t}, \ZStar ) \ge \left( 1 - \frac{\alpha}{C} \right)^{t} \dist(z^{0}, \ZStar )
$$
for $t < m$.
\end{corollary}

\begin{proof}
Consider initial solution $y^0 = x^0 = \mathbf{0}$, where $A$ and $b$. Then, for any span-respecting primal-dual method satisfies
\begin{subequations}
\begin{flalign*}
x^{t} &\in \Span{\Mat^\T y^{0}, \dots, \Mat^\T y^{t}} \\
y^{t} &\in \Span{\Mat x^{0} - b, \dots, \Mat x^{t-1} - b} \ ,
\end{flalign*}
\end{subequations}
thus it holds that
\begin{flalign}\label{span:Axt}
x^{t} &\in \Span{\Mat^{\T} (\Mat x^{0} - b), \dots, \Mat^\T (\Mat x^{t-1} -  b)}\ .
\end{flalign}
Let $\gamma_{\max} = C^2$, $\gamma_{\min} = \alpha^2$,  and consider the positive definite matrix $H$ and vector $h$ supplied by Theorem~\ref{thm:classic-lb}.
Define $A = H^{1/2}, \quad b = H^{-1/2} b$. Note $H^{1/2}$ exists and is positive definite as $H$ is positive definite. Also, note that the unique solution to the primal-dual problem is $\xStar = A^{-1} b$ and $\yStar = A^{-1} \mathbf{0} = \mathbf{0}$ with $\xStar$ matching the optimal solution to \eqref{eq:unconstrained-quadratic-min}.
Then by \eqref{span:Axt} we have 
$x^{t} \in \Span{H x^{0} - h, \dots, H x^{t-1} -  h}$
which implies $x^{t}$ is span-respecting for \eqref{eq:unconstrained-quadratic-min}.
Applying the iteration bound of Theorem~\ref{thm:classic-lb} and using that $\yStar = y^{0} = \mathbf{0}$ yields the result.
\end{proof}

Therefore we have established that
the convergence bounds of our restart schemes (Theorem~\ref{thm:fixed-frequency}
and \ref{thm:adaptive}) matches the lower bound (Corollary~\ref{coro:primal-dual-lb}) up to a constant.
In other words, a restarted method gives the best worst-case bounds.

\subsection{Tight bounds for the average and last iterate}\label{sec:unconstrained-bilinear}

Previous section shows that our restart schemes match the lower bound (up to a constant) for solving sharp primal-dual problems. 
However, at this point it is unclear if standard approaches, e.g., the last and average iterate do not also obtain this worst-case bound with a careful analysis for sharp primal-dual problems. In this section, we show for bilinear problems that the last iterate and the average iterate of PDHG cannot match convergence bounds of restarted PDHG. Such statement is genuinely the case for other primal-dual first-order methods with a similar arguments.  

Simple empirical tests support the hypothesis that the restarted PDHG is geninuely faster than the average or last iterate.
In particular, Figure \ref{fig:bilinear-iterates} plots the last iterate ($z_t$) and the average iterate ($\bz_t$) of both non-restarted PDHG (i.e., PDHG without restarting) and restarted PDHG (with fixed frequency) for solving a simple two-dimensional bilinear problem $\Lag(x,y)=x y$ with step-size $\eta=0.2$. As we can see, the average solution of the restarted algorithm has faster convergence to the primal-dual solution $(0,0)$ than the average or last iterate of PDHG.

\begin{figure}
    \centering
    \includegraphics[height=250pt]{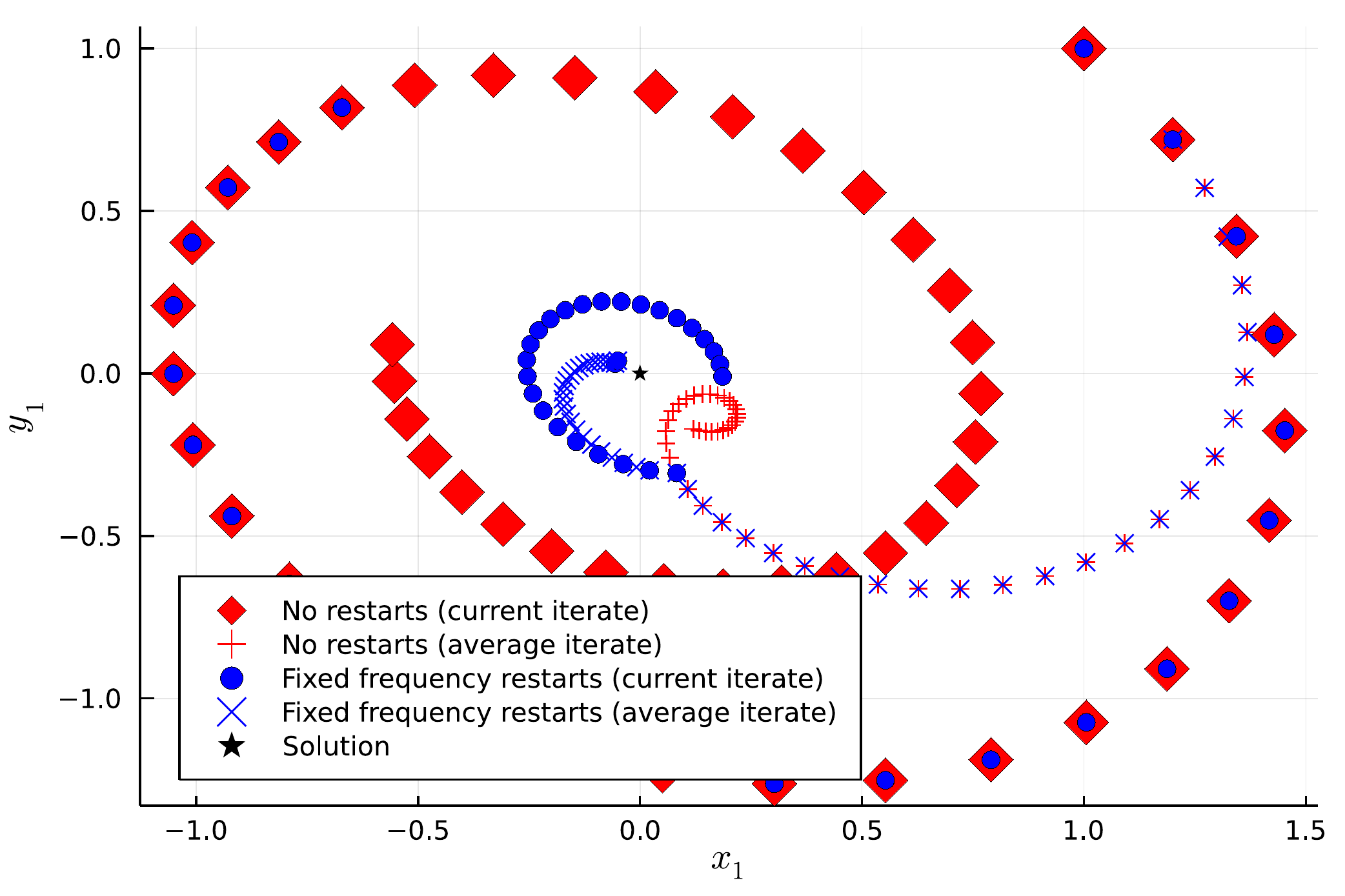}
    \caption{Plot of the first $50$ iterates \blue{of non-restarted and restarted PDHG} for a simple two-dimensional bilinear problem $\Lag(x,y) = x y$ with $\eta = 0.2$. The restart length is for the fixed frequency restarts is $25$. The unique optimal solution is $(0,0)$.
    }
    \label{fig:bilinear-iterates}
\end{figure}

We can formalize this improvement from restarted PDHG.
Table \ref{tab:complexity-simple} compares the complexity of the last iterate of PDHG, the average iterate of PDHG, and restarted PDHG on bilinear problems. From the table we can see that only the last iterate achieves linear convergence and its linear convergence constant is a square root factor worse than restarted PDHG. 
We expect similar results would also apply to other primal-dual algorithms, such as EGM.
The remainder of this subsection is devoted to proving the results in Table~\ref{tab:complexity-simple}. 

\begin{table}[]
\centering
\begin{tabular}{c c c}
\toprule
Last Iterate & Average Iterate & Restarted  \\ 
\midrule
 $\Theta\left( \left( \frac{\sigmaMax{A}}{ \sigmaMin{A}} \right)^2 \log\left( \frac{1}{\epsilon} \right)  \right)
$ & $\Theta\left( \frac{\sigmaMax{A}}{\sigmaMin{A} \epsilon} \right)$ &
$\Theta\left( \frac{\sigmaMax{A}}{\sigmaMin{A}} \log\left( \frac{1}{\epsilon} \right) \right)$  \\
\bottomrule
\end{tabular}
\caption{Comparison of the convergence of the last iterate, average iterate, and restarted PDHG on bilinear problems with $\eta = 1/(2 \sigmaMax{A})$. The algorithms terminate when $\frac{\|z^t-\zStar\|_{2}^2}{\|z^0-\zStar\|_{2}^2} \le \epsilon$.
We also assume $\epsilon \in (0, 1/2]$. Additionally for the average iterate lower bound we assume $\epsilon \le \sigmaMin{A}^2 / \sigmaMax{A}^2$.}
\label{tab:complexity-simple}
\end{table}

Consider an bilinear problem with $\Lag(x,y)=y^{\T} A x$. Without loss of generality, we can assume $A$ is a diagonal matrix with entries $0<\sv_1\le ... \le \sv_m$ and $x,y\in\R^m$ (See the beginning of the proof of Theorem~6 in Applegate et.\ al \cite{applegate2021infeasibility} for an explanation.)  

The unique optimal solution to this primal-dual problem is $(0,0)$. PDHG with step-size $\eta\le \frac{1}{\sv_m}$ has iterate update
\begin{equation}\label{eq:PDHG_simple}
    z^{t+1} = \twomatrix{1}{-\eta \Mat^\T}{\eta \Mat}{1-2\eta^2 \Mat \Mat^\top} z^t\ ,
\end{equation}
where $z^t=\vectorr{x^t}{y^t}$. We study the convergence of three sequences: the last iterates of PDHG ($z^t$), the average iterates of PDHG ($\bz^{K}=\frac{1}{K}\sum_{t=1}^{K} z^t$) 
and the solution $\bz^{n,t}$ obtained in our restart scheme with PDHG as a subroutine. 

Notice the PDHG dynamic \eqref{eq:PDHG_simple} is separable over coordinates, and there are $m$ independently evolving $2$-d linear dynamical systems given by:
\begin{equation}\label{eq:2d-dynamic}
    z^{t+1}_i = P_i z^{t}_i\ ,
\end{equation}
for $i = 1, \dots, m$ where $z_i=\vectorr{x_i}{y_i}$and $P_i=\twomatrix{1}{-\eta\sv_i}{\eta\sv_i}{1-2\eta^2\sv_i^2}$. Let $P_i=Q_i\PEV_i Q^{-1}_i$ be the eigen-decomposition of $P_i$, where 
$$
\PEV_i = \begin{pmatrix} 
1-\eta^2 \sv_i^2 - \ii \eta \sv_i \sqrt{1-\eta^2 \sv_i^2} & \\
& 1-\eta^2 \sv_i^2 + \ii \eta \sv_i \sqrt{1-\eta^2 \sv_i^2}\ ,
\end{pmatrix}
$$ 
$$
Q_i = \begin{pmatrix}
\eta \sigma_i - \ii \sqrt{1 - \eta^2 \sigma_i^2} & \eta \sigma_i + \ii \sqrt{1 - \eta^2 \sigma_i^2} \\
1 & 1 
\end{pmatrix} \ ,
$$
and $\ii$ denotes the imaginary part of a complex number.

Moreover,
with $\tilde{z}_i^t = Q_i^{-1} z_i^{t}$, $B_i = (Q_i^{-1})^\CT Q_i^{-1}$ and $\| v \|_M = \sqrt{v^\CT M v}$, where $\CT$ is the conjugate transpose, we get
\begin{flalign}
    \| z^t_i- \zStar_i \|_{B_i}^2 &= \| z_i^t \|_{B_i}^2 = \| Q_i^{-1} z_i^t \|_{2}^2 = \|\hz_i^t\|^2_2 = \| \PEV_i^t \hz_i^{0} \|^2_2 =  (1-\eta^2 \sv_i^2)^t \| \hz_i^{0} \|_2^2 = (1-\eta^2 \sv_1^2 )^t \|z^{0}_i-\zStar_i\|_{B_i}^2 \ , \label{eq:lin-converge}
\end{flalign}
where the second to last equality uses
\begin{flalign}\label{eq:lambda-i-abs}
\abs{\Pev_i^{\pm}}^2 = (1-\eta^2 \sv_i^2)^2 + \eta^2 \sv_i^2 (1-\eta^2 \sigma_i^2) = 1-\eta^2 \sv_i^2\ .
\end{flalign}

Moreover, note that
$$
Q_i Q_i^{\CT} = \begin{pmatrix}
\eta \sigma_i - \ii \sqrt{1 - \eta^2 \sigma_i^2} & \eta \sigma_i + \ii \sqrt{1 - \eta^2 \sigma_i^2} \\
1 & 1 
\end{pmatrix} \begin{pmatrix}
\eta \sigma_i + \ii \sqrt{1 - \eta^2 \sigma_i^2} & 1 \\
\eta \sigma_i - \ii \sqrt{1 - \eta^2 \sigma_i^2} & 1 
\end{pmatrix} = 2 \begin{pmatrix}
 1 & \eta \sigma_i \\
\eta \sigma_i & 1
\end{pmatrix}
$$
with eigenvalues $2 \pm 2 \sv_i \eta $.
Therefore, as $\eta \in [0, 1/(2 \sv_m)]$
we get $1 \le 2 \pm 2 \sv_i \eta \le 3$, which implies
\begin{flalign}\label{eq:B_i-approx-identity}
\| v \|_2^2 \le v^\CT Q_i Q_i^{\CT} v \preceq 3 \| v \|_2^2 \Rightarrow \frac{\| v \|_2^2}{3} \preceq v^\CT B_i v  \preceq \| v \|_2^2
\end{flalign}
for all $v \in \mathcal{C}^{n}$ where $\mathcal{C}$ is the set of complex numbers and the implication uses that $B_i Q_i Q_i^\CT = \eye$.

By \eqref{eq:B_i-approx-identity} we get $\frac{1}{3} \|z^t-\zStar\|_{2}^2 \le \sum_{i=1}^{n} \| z_i - z_i \|_{B_i}^2 \le \|z^t-\zStar\|_{2}^2$ and therefore \eqref{eq:lin-converge} implies that we reach a solution $z^t$ such that 
$$
\frac{\|z^t-\zStar\|_{2}^2}{\|z^0-\zStar\|_{2}^2} \le \epsilon \ ,
$$
within $O\pran{\pran{\frac{1}{\eta \sv_1}}^2 \log\pran{\frac{1}{\epsilon}}}$ iterations.

From \eqref{eq:B_i-approx-identity} and \eqref{eq:lin-converge} we deduce that
$$
\| z^t_1 - \zStar_1 \|_{2}^2 \ge \frac{1}{3} (1 - \eta^2 \sigma_1^2)^t \| z^t_1 - \zStar_1 \|_{2}^2\ .
$$
Therefore by $\sv_1 \le \sv_m$ and $\eta = 1 / (2 \sv_m)$ we get
$\Omega\left( \pran{\frac{\sv_m}{\sv_1}}^2 \log\pran{\frac{1}{\epsilon}} \right)$ as a lower bound for the last iterate convergence of PDHG.

Define $\| M \|_2 := \sup_{\| v \|_2 = 1} \| M v \|_2$ be the $\ell_2$ norm of a matrix $M$. Now we study the solution of the average iterates $\bz^K=\frac{1}{K}\sum_{t=1}^{K} z^t$, and we claim $\bz^K$ converges to the unique optimal solution $\zStar=0$ with sublinear complexity $\Omega(\frac{\sigma_m}{\sigma_1 \epsilon})$. Notice that
\begin{align}\label{eq:average-simple}
\begin{split}
    \norm{\bz_i^K -\zStar_i}_{B_i} &= \norm{\frac{1}{K}\sum_{t=1}^{K} z^t_i}_{B_i} \\
    &= \norm{\frac{1}{K}\sum_{t=1}^{K} \hz^t_i}_{2} \\
    &= \frac{1}{K}\norm{\sum_{t=1}^{K} \PEV_i^t \hz_i^{0}}_{2} \\
    &= \frac{1}{K}\norm{\PEV_i (\eye-\PEV_i)^{-1} (\eye-\PEV_i^{K})  \hz_i^{0}}_2 \\
    &\le \frac{1}{K}\norm{\PEV_i}_2 \norm{(\eye-\PEV_i)^{-1}}_2 \norm{\eye-\PEV_i^K}_2 \norm{  \hz^{0}}_2 \\
    &\le \frac{1}{K (\eta\sv_1)}\norm{\PEV_i}_2\pran{\norm{\eye}_2+\norm{\PEV_i^K}_2} \norm{  \hz^{0}_i}_2 \\
    &\le \frac{2}{K \eta\sv_1} \norm{  \hz^{0}_i }_2 \\
    &= \frac{2}{K \eta\sv_1} \norm{z^{0}_i -\zStar_i}_{B_i} \ ,
\end{split}
\end{align}
where the second inequality uses \begin{flalign}\label{eq:one-minus-gamma}
\abs{1 - \Pev^{\pm}_i} = \sqrt{\eta^4 \sigma_i^4 + \eta^2 \sigma_i^2 (1 - \eta^2 \sigma_i^2)} = \eta \sigma_i
\end{flalign}
and the third inequality utilizes $\|\PEV_i\|_2\le 1$. Setting $\eta = \frac{1}{2 \sv_m}$ and using \eqref{eq:B_i-approx-identity} demonstrates $\bz^t$ converges to $\zStar$ in $O\pran{\frac{\sv_m}{\sv_1 \epsilon}}$ iterations, which is consistent with the ergodic rate of PDHG.

We will now establish a matching $\Omega\pran{\frac{\sv_m}{\sigma_1 \epsilon}}$ lower bound for the average iterate. Observe that
\begin{flalign}
\abs{ 1 - (\Pev_i^{\pm})^K } &\ge  1 - \abs{\Pev_i^{\pm}}^K \notag \\
&\ge1 - (1 - \eta^2 \sigma_i^2 )^{K/2} \notag\\
&\ge 1 - \frac{1}{1 + (K/2) \eta^2 \sigma_i^2} \notag \\
&\ge \frac{K \eta^2 \sigma_i^2}{2 + K \eta^2 \sigma_i^2}\ . \label{eq:bound-gamma-to-the-K}
\end{flalign}
where the third inequality uses the fact that $(1 - a)^b \le \frac{1}{1 + b a}$ for $a \ge 0$, $b \ge 0$ (to see this note that $f(a,b) = (1 - a)^b (1 + b a)$ satisfies $f(0,b) = 1$ and $\frac{\partial f}{\partial a} \le 0$).
Therefore,
\begin{flalign*}
\norm{\bz_1^K -\zStar_1}_{B_1}  &= \frac{1}{K}\norm{\PEV_1(\eye-\PEV_1)^{-1} (\eye-\PEV_1^K)  \hz_1^{0}}_2 \\
&= \frac{1}{K} \sqrt{ (\hx_1^{0})^2 \abs{\Pev_1^{-}}^2 \abs{1 - \Pev_1^{-}}^{-2} \abs{1 - (\Pev_1^{-})^K}^2 + (\hy_1^{0})^2 \abs{\Pev_1^{+}}^2 \abs{1 - \Pev_1^{+}}^{-2} \abs{1 - (\Pev_1^{+})^K}^2 } \\
&\ge \frac{1}{K} \| \hz_1^{0} \|_2 \abs{\Pev_1^{\pm}} \abs{1 - \Pev_1^{\pm}}^{-1} \left( 1 - \abs{\Pev_1^{\pm}}^K \right) \\
&\ge \frac{\| \hz^{0}_1 \|_2}{K} \frac{\sqrt{1-\eta^2 \sv_1^2}}{\eta \sigma_1} \times \frac{K \eta^2 \sigma_1^2}{2 + K \eta^2 \sigma_1^2} \\
&\ge \frac{\| \hz^{0}_1 \|_2}{K}  \frac{\sqrt{3/4}}{\eta \sigma_1} \times \frac{K \eta^2 \sigma_1^2}{2 + K \eta^2 \sigma_1^2} \\
&=  \frac{\sqrt{3}}{2} \times \frac{\eta \sigma_1 \| z^{0}_1 - \zStar_1 \|_{B_1}}{2 + K \eta^2 \sigma_1^2} \ ,
\end{flalign*}
where the second equality uses that the absolute value is completely multiplicative, the second inequality uses \eqref{eq:lambda-i-abs}, \eqref{eq:one-minus-gamma} and \eqref{eq:bound-gamma-to-the-K} and the final inequality uses $\eta = \sv_m / 2$. This establishes the lower bound for the average iterate. We note that this lower bound only matches the upper bound for $K = \Omega(\frac{1}{\eta^2 \sv_1^2})$. We believe matching the upper bound across all $K$ is possible but the analysis appears to be much more technical.

\subsection{Comparison with other methods}\label{sec:comparision-other-methods}

Mokhtari et al. \cite{mokhtari2020unified} study EGM and the optimistic gradient descent, in this setting they show an
$$
O\left( \frac{\sigma_{\max}(A^{\T} A)}{\sigma_{\min}(A^{\T} A)} \log( 1/\varepsilon ) \right) = O\left( \left( \frac{\sigmaMax{A}}{\sigma_{\min}(A)} \right)^2 \log( 1/\varepsilon  ) \right)
$$
bound for both methods on bilinear problems (recall $\sigma_{\min}(A^{\T} A) = \sigma_{\min}(A)$ and $\sigma_{\max}(A^{\T} A) = \sigmaMax{A}$). As expected, this matches the last iterate performance of PDHG.

We can use other methods to solve bilinear problems. 
One approach is reduce the bilinear game to a pair of nonsmooth optimization problems:
\begin{flalign*}
\min_{x \in \R^{m}}&~ \| A x - b \|_2 \\
\min_{y \in \R^{m}}&~ \| A^{\T} y - c \|_2\ .
\end{flalign*}
Each of these problems is $\sigmaMin{A}$-sharp (Lemma~\ref{prop:singular-value}) and $\sigmaMax{A}$-Lipschitz. Therefore according to \cite[Corollary 9]{yang2018rsg} restarted subgradient method requires 
$$
O\left( \left( \frac{\sigmaMax{A}}{\sigmaMin{A}} \right)^2 \log(1 / \varepsilon ) \right)
$$
iterations on each problem to find an approximate solution satisfying $\frac{\| A x - b \|_2}{\| A x\ind{0} - b \|_2} \le \varepsilon$ and $\frac{\| A^{\T} y - c \|_2}{\| A^{\T} y\ind{0} - c \|_2} \le \varepsilon$.
Another approach is to reformulate solving bilinear game as a pair of smooth unconstrained problems:
\begin{flalign*}
\min_{x \in \R^{m}}&~ \| A x - b \|_2^2 \\
\min_{y \in \R^{m}}&~ \| A^{\T} y - c \|_2^2\ .
\end{flalign*}
One can apply conjugate gradient or accelerated gradient descent to each of these subproblems.
The well-known convergence bounds for these methods \cite{nesterov2013introductory} imply that they have the following iteration bounds
$$
O\left( \sqrt{\frac{\sigma_{\max}(A^{\T} A)}{\sigma_{\min}^{+}(A^{\T} A)}} \log( 1/\varepsilon ) \right) = O\left( \frac{\sigmaMax{A}}{\sigmaMin{A}} \log( 1/\varepsilon  ) \right)\ ,
$$
which matches the bounds we obtained by restarted PDHG. We also point the reader to Gilpin et al. \cite{gilpin2012first}, who obtain similar results to ours in the LP (with bounded feasible region) and bilinear setting. 
The main difference is that they modify Nesterov’s smoothing algorithm for nonsmooth convex optimization \cite{nesterov2005smooth}, which is a primal method, whereas we combine restarts with primal-dual methods to achieve our result.

\section{Efficient computation of the normalized duality gap}\label{sec:efficent_comp}

\subsection{Separable norm}\label{sec:general-case-computation}

This subsection shows that the normalized duality gap can be computed efficiently under the mild assumption that the norm is separable:
$$
\| (x,y) \|^2 = \| x \|^2 + \| y \|^2\ .
$$ 
To evaluate \eqref{eq:rhorz-pos} we need to find a solution to
\begin{flalign}\label{eq:trust-region}
\max_{\tz \in \zBall{r}{z}} \Lag(x, \ty) - \Lag(\tx, y)\ .
\end{flalign}
Recall that due to Proposition~\ref{fact:avoid-trivial-setting} solving \eqref{eq:rhor-zero} is unnecessary for adaptive restarts.

Solving \eqref{eq:trust-region} can be done using
standard methods which we briefly outline here.
By the assumption that $\Lag$ is convex-concave and by duality, there exists a $\lambda \in [0,\infty]$ such that \eqref{eq:trust-region} is equivalent to
\begin{flalign}\label{eq:joint-prox-problem}
\max_{(\tx,\ty) \in Z} \Lag(x, \ty) - \Lag(\tx, y) - \frac{\lambda}{2} ( \| \tx - x \|^2 + \| \ty - y \|^2)\ .
\end{flalign}
For any $\lambda \in (0,\infty)$, define
$$
h(\lambda) := \| \bz(\lambda) - z \| - r \ ,
$$
where $\bz(\tau)$ is the unique solution to \eqref{eq:joint-prox-problem} (the solution is unique because $\Lag$ is convex-concave which means \eqref{eq:joint-prox-problem} is strongly concave). Observe the function $h(\lambda)$ is nonincreasing.
If $h(\lambda) < 0$ for all $\lambda \in (0,\infty)$ then  the solution to \eqref{eq:joint-prox-problem}  with $\lambda = 0$ that is closest to $z$ solves \eqref{eq:trust-region} (and can be approximated by picking an arbitrarily tiny $\lambda$).
If $h(\lambda) > 0$ for all $\lambda \in (0,\infty)$ then  $r = 0$, and the optimal solution to \eqref{eq:trust-region} is $\tz = z$.
The final possibility is that there exist $\lambda, \lambda' \in (0,\infty)$ such that $h(\lambda) < 0 < h(\lambda')$ in which case there exists a root of $h$ which we can find using a standard root finding methods such as the bisection method.
This root corresponds to the optimal solution of \eqref{eq:trust-region}.

Problem \eqref{eq:joint-prox-problem} is separable into
\begin{subequations}\label{eq:split-prox-problem}
\begin{flalign}
\argmin_{\tx \in X} \Lag(\tx,y) + \frac{\lambda}{2} \| \tx - x \|^2 \\
\argmin_{\ty \in Y} -\Lag(x, \ty) + \frac{\lambda}{2} \| \ty - y \|^2 \ ,
\end{flalign}
\end{subequations}
which is equivalent to the subproblem iterations for PPM, PDHG, and ADMM (Section~\ref{sec:unified-pda}).
Even if this root finding only requires evaluating \eqref{eq:split-prox-problem} a small number of times it could significantly increase the cost of running the algorithm. Therefore, it makes sense to only test condition \eqref{eq:adaptive-restart-scheme} at fixed intervals with large separation, e.g., every 200 iterations.

\subsection{Primal-dual hybrid gradient}\label{pdhg:practical-details}

Section~\ref{sec:general-case-computation} considered the case
where $\| \cdot \|$ is seperable. However, the (natural) norm that we previously
used to analyze PDHG is
$\| z \| = \sqrt{ \| x \|_2^2 - 2 \eta y^{\T} A x + \| y \|_2^2}$
which is not separable.
Fortunately, the following Corollary shows that
Property~\ref{property:restart-assumption} also holds for PDHG with the
Euclidean norm. This allows us the possibility of computing $\rho_r$ with $\| \cdot \|$ being the
Euclidean norm and retaining our theoretical guarantees.
Indeed this is how we implement the adaptive restart scheme for PDHG in
Section~\ref{sec:experimental-results}.

\begin{corollary}\label{coro:pdhg-with-euclidean-norm}
For PDHG with $\eta \in (0, 1/\sigmaMax{A})$,
Property~\ref{property:restart-assumption} holds when $\| \cdot \|$ is the
Euclidean norm,
$C = \frac{2}{\eta (1 - \eta \sigmaMax{A})}$ and
$q = 4 \frac{1 + \eta \sigmaMax{A}}{1 - \eta \sigmaMax{A}}$.
\end{corollary}

The proof uses that Property~\ref{property:restart-assumption} holds for the PDHG norm, and Proposition~\ref{fact:pdhg-norm} which establishes that $\| z \|$ is approximately Euclidean.
The proof appears in Appendix~\ref{sec:pdhg-with-euclidean-norm}.

By Proposition~\ref{thm:pdhg-ppm:bound-adaptive-restart-scheme} we also know that the iterates will remain bounded even after this change of norm. Therefore, using Lemma~\ref{example:lp} we have established the premise of Theorem~\ref{thm:adaptive} for PDHG with the Euclidean norm on LP problems.

\subsection{Linear time implementation for linear programming}\label{sec:linear-time-implementation}

This subsection shows that we can more efficiently compute the normalized duality gap in the LP setting.

\subsubsection{Linear time algorithm for linear trust region problems}
\label{sec:linear-time-implementation:trustregion}

The key ingredient to the linear time computation of the normalized duality gap is a linear time algorithm for linear trust region problems.
In particular, consider a trust region problem with a linear objective $g$, center $z$, Euclidean norm, and
lower bounds $l$ (possibly infinite) on the variables:
\begin{flalign}\label{eq:lp-trust-region-problem}
\argmin_{\tz \in \R^{n+m} : l \le \tz, \|\tz - z \| \le r}
g^{\T} \tz\ .
\end{flalign}
If $g_i < 0$ then setting $g_i = -g_i$ and $l_i = -\infty$ will create an equivalent problem satisfying $g_i > 0$. Therefore without loss of generality we can assume $g_i > 0$ for all $i$.

The optimal $\tz$ will be of the form
\begin{flalign}\label{eq:lp-trust-region:max}
\tz(\lambda) & := \max(z - \lambda g, l)\ ,
\end{flalign}
for some $\lambda \in [0,\infty)$, so (\ref{eq:lp-trust-region-problem}) reduces
to the univariate problem
\begin{flalign}\label{eq:lp-trust-region:lambda}
\max \lambda \; : \; \|\tz(\lambda) - z \| \le r\ .
\end{flalign}
Letting $\hat{\lambda}_i$ be the value of $\lambda$ at which $\tz(\lambda)_i$ encounters the bound, that is, $\hat{\lambda}_i := (z_i - l_i) / g_i$, we have
\begin{flalign}\label{eq:lp-trust-region:summation}
\| \tz(\lambda) - z \|^2 & = \sum_{i\;:\;\hat{\lambda}_i \le \lambda} (l_i - z_i)^2 + \lambda^2 \sum_{i\;:\;\hat{\lambda}_i > \lambda} g_i^2\ .
\end{flalign}
This forms the basis for a linear time algorithm: by computing (\ref{eq:lp-trust-region:summation}) for the median value $\lambda_{\text{med}}$ of $\hat{\lambda}_i$ (in linear time~\cite{BFPRT1973}), and comparing it with $r^2$, we determine whether or not the optimal $\lambda$ for (\ref{eq:lp-trust-region:max}) is greater or less than $\lambda_{\text{med}}$.  This lets us eliminate half of the $\hat{\lambda}_i$, collapsing their contributions into the appropriate summand.  Applying this recursively gives a running time recurrence $f(n+m) = O(n+m) + f((n+m)/2)$, i.e., $f(n+m) = O(n+m)$.  Details are in Appendix~\ref{appendix:linear-time-implementation:trustregion}.

\subsubsection{Extragradient and primal-dual hybrid gradient}\label{sec:linear-time-implementation:EGM-and-PDHG}

Recall that for extragradient the norm for defining $\rho_r$ is Euclidean.  
For LP with extragradient (Section~\ref{sec:sharpness-lp}), we use the Lagrangian:
\begin{flalign}
\min_{x \ge 0} \max_{y \in \R^{m}} \Lag(x,y) & = c^{\T} x + y^{\T} b-y^{\T} \Mat x
\end{flalign}
for which the trust region problem (\ref{eq:trust-region}) is equivalent to
\begin{flalign}
\argmin_{\tx \ge 0, \ty \in \R^m : 
\|(\tx, \ty) - (x, y)\| \le r}
(c^{\T} - y^{\T} \Mat) \tx - (b^{\T} - (\Mat x)^{\T}) \ty\ .
\end{flalign}
which is a trust region problem with a linear objective, Euclidean
norm, and lower bounds on variables (the dual lower bounds are $-\infty$).  
Given $y^{\T} \Mat$ and $\Mat x$, this
can be solved in time $O(n+m)$ using the algorithm in
Section~\ref{sec:linear-time-implementation:trustregion}. The same result hold for PDHG with Euclidean norm as described in Corollary~\ref{coro:pdhg-with-euclidean-norm}.

\subsubsection{ADMM}

Recall for ADMM for LP (Lemma~\ref{example:admm}) that
\begin{flalign*}
\min_{x_U\in X_U, x_V\in X_V}\max_{y} ~\Lag(x,y) = c^{\T} x_{V} - y^{\T} \begin{pmatrix} \eye & -\eye \end{pmatrix} \begin{pmatrix}
x_{\U} \\
x_{\V}
\end{pmatrix} \
\end{flalign*}
and $\| \cdot \| = \| \cdot \|_M$ with
$$
M = \begin{pmatrix}
    0 & 0 & 0 \\
    0 & \eta \eye & 0 \\
    0 & 0 & \frac{1}{\eta} \eye
    \end{pmatrix}\ .
$$
Therefore, for $Q_r(x_{V}, y) = \{ (\tx_{V}, \ty) : \tx_{V} \ge 0,
\blue{ \eta \| x_V - \tx_{V} \|_2^2 + 
\frac{1}{\eta} \| y - \ty \|_2^2 } \le r^2 \}$ we have
\begin{flalign}
\rho_r(\bar{z}^t) &= \frac{1}{r} \sup_{\tz \in \zBall{r}{\bar{z}^t}} c \cdot (\bar{x}^t_{V} - \tx_{V}) + \bar{y}^{t} \cdot (\tx_{\U} - \tx_{\V})  - \ty \cdot (\bar{x}^t_{\U} - \bar{x}^t_{\V})  \notag \\
&= \frac{1}{r} \left( \sup_{\hat{x}_{\U} \in X_{U}} \bar{y}^t \cdot \hat{x}_{\U} + \sup_{(\tx_{V}, \ty) \in Q_r(\bar{x}^t_{V}, \bar{y}^t)} c \cdot (\bar{x}^t_{V} - \tx_{V}) - \bar{y}^t \cdot \tx_{\V} - \ty \cdot (\bar{x}^t_{\U} - \bar{x}^t_{\V}) \right) \notag \\
&= \frac{1}{r} \left( \bar{y}^t \cdot \bar{x}^t_{\U} + \sup_{(\tx_{V}, \ty) \in Q_r(\bar{x}^t_{V}, \bar{y}^t)} c \cdot (\bar{x}^t_{V} - \tx_{V}) - \bar{y}^t \cdot \tx_{\V} - \ty \cdot (\bar{x}^t_{\U} - \bar{x}^t_{\V}) \right) \notag \\
&= \blue{\frac{1}{r}  \sup_{(\tx_{V}, \ty) \in Q_r(\bar{x}^t_{V}, \bar{y}^t)} -(\bar{y}^t+c) \cdot (\tx_{V} - \bar{x}^t_{V}) + (\bar{x}^t_{\V} - \bar{x}^t_{\U}) \cdot (\ty - \bar{y}^t)} \ ,  \label{eq:rho-r}
\end{flalign}
where the first equality follows from definition of the normalized duality gap, the second equality uses the definition of $M$ and $\zBall{r}{z}$, and the third equality uses that $\bar{y}^t \cdot \tx_{U} = \bar{y}^t \cdot \bar{x}_U$. To see this claim that $\bar{y}^t \cdot \tx_{U} = \bar{y}^t \cdot \bar{x}_U$ note that $\bar{x}^t_U \in X_{U} = \{ x : A x = b\}$ and that if $\bar{y}^t \cdot \tx_{U} \neq \bar{y}^t \cdot \bar{x}^t_{U}$ for $\tx_{U} \in X_{U}$ then $\bar{x}^t_{U} + \alpha (\tx_{U} - \bar{x}^t) \in X_{U}$ for all $\alpha \in \R$ and either $\lim_{\alpha \rightarrow \infty} \bar{y}^t \cdot ( \bar{x}^t_{U} + \alpha (\tx_{U} - \bar{x}^t) ) = \infty$  or $\lim_{\alpha \rightarrow -\infty} \bar{y}^t \cdot ( \bar{x}^t_{U} + \alpha (\tx_{U} - \bar{x}^t) ) = \infty$, this contradicts that $\rho_r(\bar{z}^t) < \infty$ (Proposition~\ref{coro:weaker-assumption}).

Equation \eqref{eq:rho-r} is again a trust region problem with a linear objective, Euclidean norm, and bounds on the variables, so can be solved in linear time using the algorithm from Section~\ref{sec:linear-time-implementation:trustregion}.

\section{Numerical experiments}\label{sec:experimental-results}

\blue{As validation of the linear convergence rates predicted by Theorems~\ref{thm:fixed-frequency} and \ref{thm:adaptive}, we implemented and tested the proposed restart schemes for PDHG, EGM and ADMM applied to LP.}
Indeed, for our test problems we are able to achieve very high accuracy solutions, commensurate with the accuracy expected from second-order methods on LP problems.
The experiments also show that our adaptive scheme is competitive with the best fixed frequency restart scheme without any need for hyperparameter search. \blue{This section primarily covers the results for PDHG. Additional details and results for EGM and ADMM are presented in Appendix \ref{sec:num-egm-and-admm}. The code for the numerical experiments is publicly available at \url{https://github.com/google-research/google-research/tree/master/restarting_FOM_for_LP}.
}

We select four LP instances, \texttt{qap10}, \texttt{qap15}, \texttt{nug08-3rd}, and \texttt{nug20}, from the Mittelmann collection set~\cite{mittelmann_benchmark}, a standard benchmark set for LP. These four problems are relaxations of quadratic assignment problems~\cite{qapbounds2002}, a classical NP-hard combinatorial optimization problem. 
These instances are known to be challenging for traditional solvers and amenable to first-order methods~\cite{GalabovaHall2020}. The instance sizes are given in Table~\ref{table:test-problem-sizes}.

\newcommand{\PrimalWeight}[0]{\omega}
For the step size of PDHG we use $\eta = 0.9 \sigmaMax{A}$ which ensures the theoretical bounds for PDHG hold, where $\sigmaMax{A}$ is estimated using power iteration.

It is known that a major factor in the practical performance of PDHG is the relative scaling of the primal and dual \cite{goldstein2015adaptive, chambolle2011first}. We call this the primal weight ($\PrimalWeight$). It can be viewed as a scaling factor on the objective value, i.e., setting $c_{\text{new}} = c / \PrimalWeight^2$, or as the relative step size for the primal and dual, i.e., the primal step size is $\eta / \PrimalWeight$ and the dual step size is $\eta \PrimalWeight$.
Before running PDHG with restarts on each of these problems we estimate the optimal primal weight by running the non-restarted PDHG algorithm for $5000$ iterations across the range $\omega \in \{ 4^{-5}, 4^{-4}, 4^{-3}, 4^{-2}, 4^{-1}, 1, 4^{1}, 4^{2}, 4^{3}, 4^{4}, 4^{5} \}$. 
We then choose the $\omega$ value with the
smallest value of 
the KKT error $\| (K z - h)^+ \|_2$ where $K$ and $h$ are defined as per \eqref{eq:define:K-and-h-for-lp}.
We use this $\omega$ value for our experiment across all restart schemes.

\begin{table}[]
    \centering
    \begin{tabular}{c c c c}
\toprule
            & rows  & cols  & nonzeros  \\
\midrule
qap10       & 1820 & 4150 & 18200   \\
qap15       & 6330 & 22275 & 94950  \\
nug8-3rd    & 19728 & 20448 & 139008    \\
nug20       & 15240 & 72600 & 304800  \\
\bottomrule
    \end{tabular}
    \caption{Test problem sizes}
    \label{table:test-problem-sizes}
\end{table}

On these problems we run PDHG with: no restarts, fixed frequency restarts with different restart lengths, and adaptive restarts (where the norm used for computing the normalized duality gap is Euclidean as described in Section~\ref{sec:linear-time-implementation:EGM-and-PDHG}).
Fixed restart lengths of $4^{1}, 4^2, \dots, 4^9$ were tested and the best three restart lengths\footnote{The restart lengths were ordered descending by iterations to find a normalized duality gap below $10^{-7}$, and then by the normalized duality gap at the maximum number of iterations. The top three restart lengths were then selected for display.} were plotted for each problem.
In Figure~\ref{fig:pdhg-experiments}, we plot the value of the normalized duality gap against the number of iterations for each of the method where the input radius for the normalized duality gap is equal to the distance traveled since the last restart (for no restarts this is the starting point).
For the adaptive restarts we use $\ReducePotentialBy = \exp(-1)$ as per Remark~\ref{remark:optimal-beta-choice} and only evaluate \eqref{eq:adaptive-restart-scheme} every 30 iterations to ensure the cost of evaluating the normalized duality gap is negligible.
There are a few consistent patterns.
One is that often a few iterations after a restart, the normalized duality gap spikes and then falls quickly below its previous value. This is particularly apparent for nug20.
Another observation is that after the active set is identified, the linear convergence is often apparent in the plot.
However, for qap10, qap15, and nug20, restarted PDHG outperforms non-restarted PDHG, much earlier than when the active set freezes.

\begin{figure}
    \centering
    \includegraphics[height=150pt]{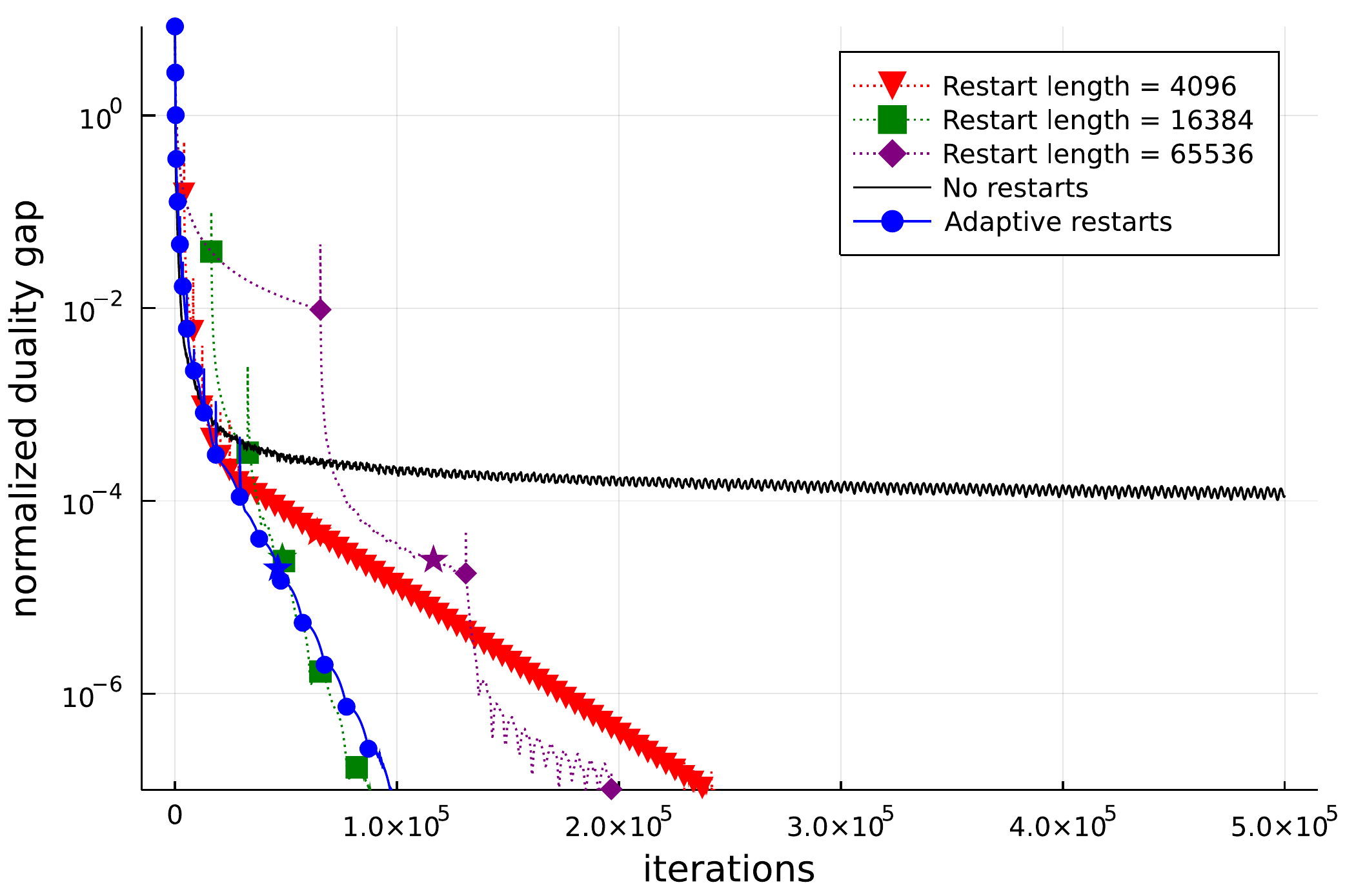}
    \includegraphics[height=150pt]{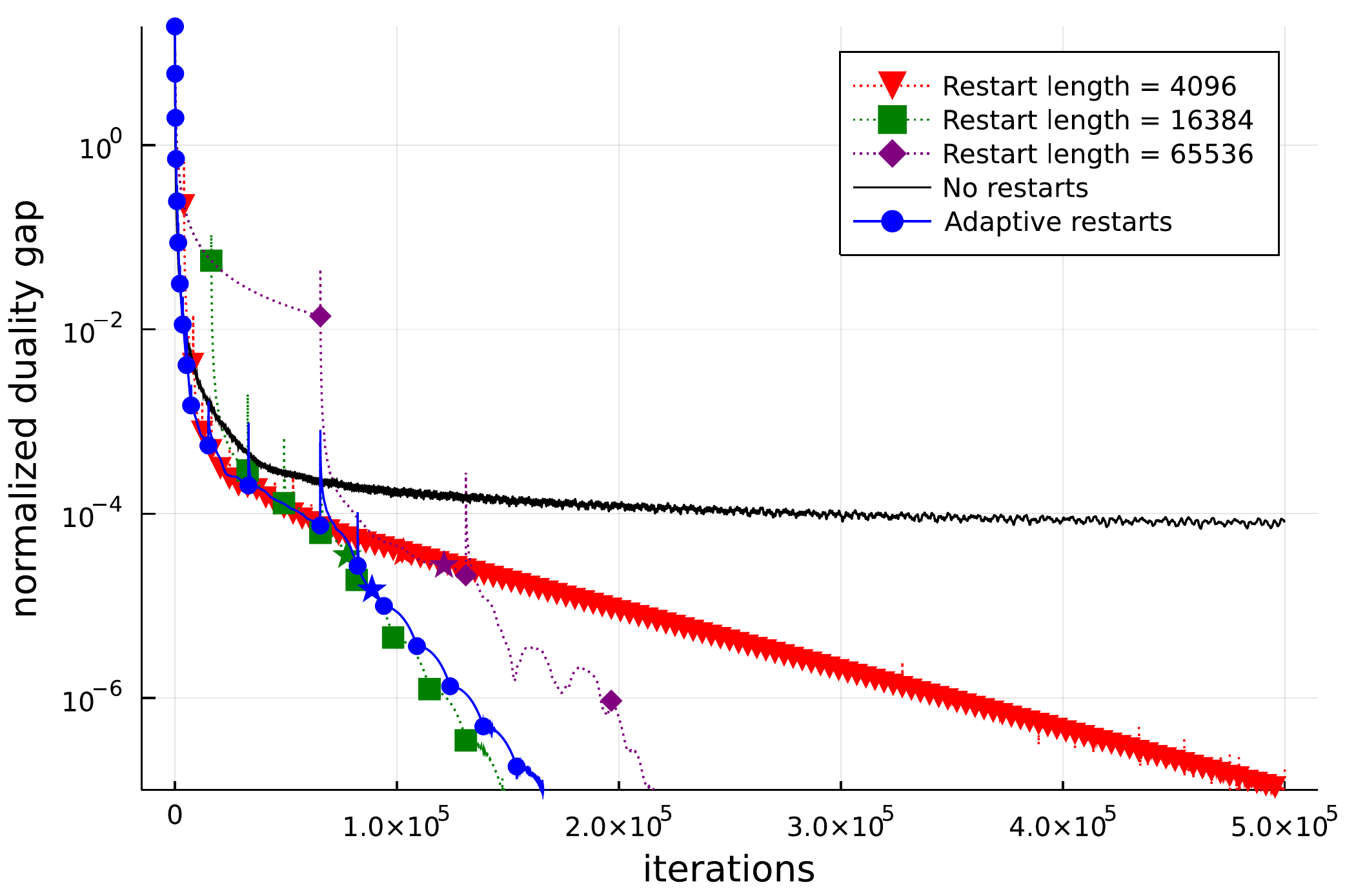}
    \includegraphics[height=150pt]{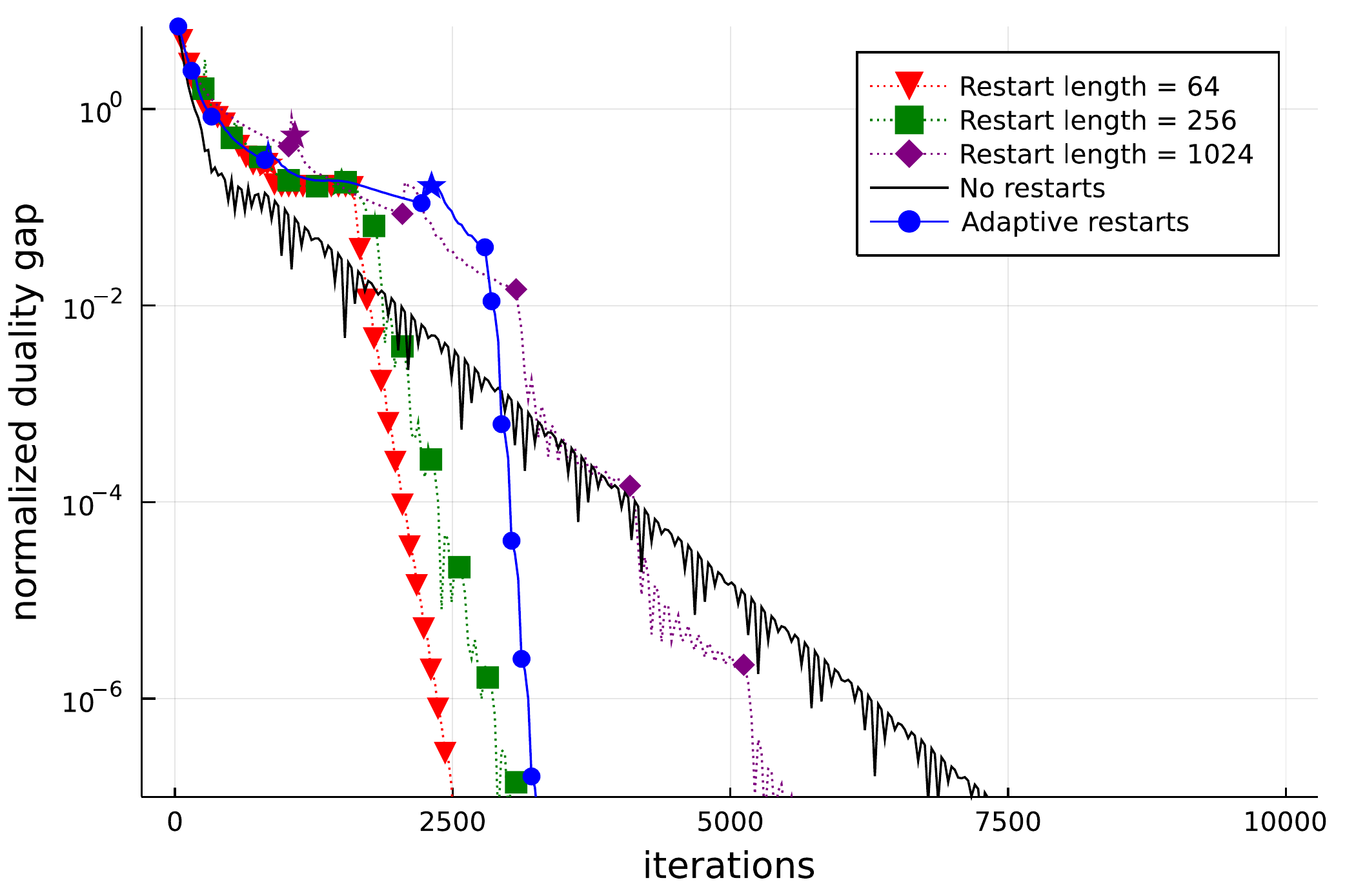}
    \includegraphics[height=150pt]{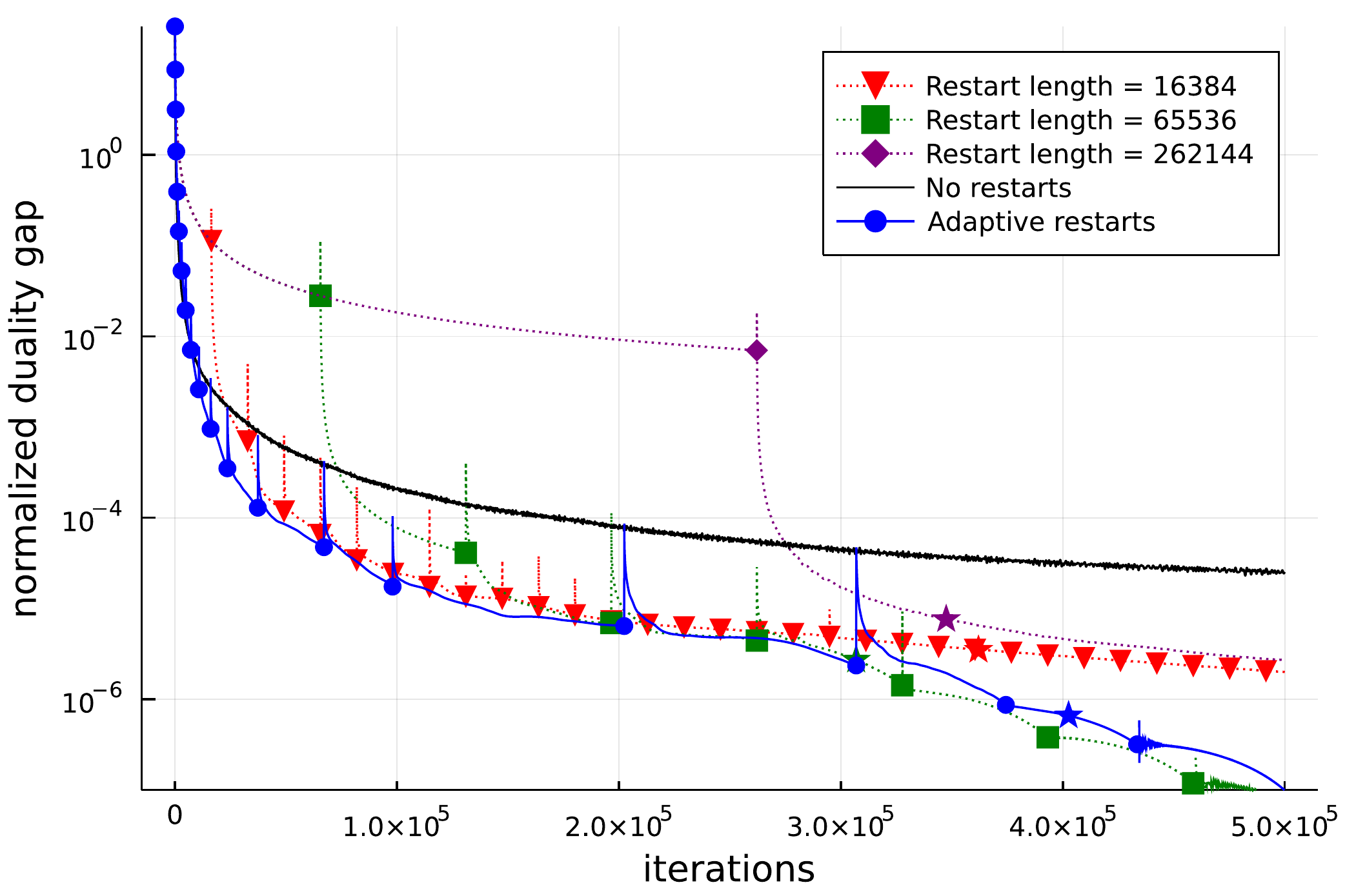} \\
    \caption{Plots show normalized duality gap (in $\log$ scale) versus number of iterations for restarted PDHG. Results from left to right, top to bottom for qap10, qap15, nug08-3rd, and nug20. Each plot compares no restarts, adaptive restarts, and the three best fixed restart lengths.
    For the restart schemes we evaluate the normalized duality gap at the average, for no restarts we evaluate it at the last iterate as the average performs much worse.
    A star indicates the last iteration that the active set changed before the iteration limit was reached.
    Series markers indicate when restarts occur.}
    \label{fig:pdhg-experiments}
\end{figure}

The normalized duality gap is not a typical metric for determining solution quality. Typically, the KKT error is a more standard \cite{o2016conic} and algorithm-independent metric. Indeed as Lemma~\ref{lem:bound-kkt-error} shows, the KKT error can be bounded above by the normalized duality gap times a constant.

For completeness, Table~\ref{table:kkt-error} reports the number of iterations required by each approach to drive the KKT error below $10^{-6}$ (either for the average or last iterate). This table exhibits very similar patterns to Figure~\ref{fig:pdhg-experiments}. In particular, non-restarted PDHG fails to drive the KKT error below $10^{-6}$ except for the nug08-3rd problem.
In Table~\ref{table:kkt-error} we compare the adaptive restarts with the best fixed frequency restart scheme \blue{for PDHG, EGM and ADMM}.
Notably, the number of iterations is fairly similar for both approaches. However, adaptive restarts do not require any hyperparameter search over the restart lengths and therefore require much less computational effort.

\begin{remark}[Flexible restarts]\label{rem:flexible}
A major benefit of using the normalized duality gap as a potential function is that provides substantial flexibility in algorithm design while preserving the  convergence properties given in Theorem~\ref{thm:adaptive}. For example, we can develop an algorithm that chooses either to restart to the average or last iterate. To do this we add an additional line below Line~\ref{line:average} of Algorithm~\ref{al:restarted+algorithm} as follows
\begin{flalign}\label{eq:flexible-restart-scheme}
\zind{n,t} \gets \begin{cases}
z^{n,t} & \rho_{\| z^{n,t} - z^{n, 0} \|}(z^{n,t}) < \rho_{\| \zind{n,t} - z^{n, 0} \|}(\zind{n,t}) \\
\zind{n,t} & \text{otherwise}. \\
\end{cases}
\end{flalign}
In other words, we consider the last iterate as a candidate for restarting, if it has a better normalized duality gap.
Inspection of Theorem~\ref{thm:adaptive} shows that it still holds with modification to \eqref{eq:flexible-restart-scheme}. This is because Property~\ref{subproperty:distance-bound} also holds for $z^{n,t}$ (Proposition~\ref{prop:assumption-one-holds} and~\ref{prop:generic-nonexpansive} establishes $\|z^{t+1}- \zStar\|\le \|z^{t}- \zStar\|$, then use Proposition~\ref{fact:pdhg-norm} to change the norm from PDHG to Euclidean).
Indeed one can imagine many other ways of producing restart candidates that satisfy Property~\ref{subproperty:distance-bound} (e.g., convex combinations of iterates) and therefore maintain the guarantees of Theorem~\ref{thm:adaptive}.
We implement \eqref{eq:flexible-restart-scheme} and report the results 
Table~\ref{table:kkt-error} under the name `flexible restarts'. One can see that the flexible restart scheme is in general slightly faster than the original adaptive restart scheme; for nug08-3rd it is much faster. 
\end{remark}

\begin{table}
\blue{
\begin{tabular}{cccccc}
\toprule
problem name & method & no restarts & best fixed frequency & adaptive restarts & flexible restarts\\
\midrule
qap10 & PDHG & $-$ & $76230$ & $86820$ & $84240$ \\
 & EGM & $-$ & $76170$ & $86820$ & $84240$ \\
 & ADMM & $-$ & $19590$ & $26340$ & $21180$\\
\midrule
qap15 & PDHG & $-$ & $144060$ & $153780$ & $154560$ \\
 & EGM & $-$ & $144060$ & $153780$ & $154590$ \\
 & ADMM & $-$ & $46620$ & $44880$ & $36660$\\
\midrule
nug08-3rd & PDHG  & $6600$ & $2280$ & $3300$ & $1590$ \\
 & EGM & $6300$ & $2520$ & $3180$ & $1590$ \\
 & ADMM & $930$ & $2520$ & $2700$ & $2520$\\
\midrule
nug20 & PDHG & $-$ & $399840$ & $447300$ & $425610$ \\
 & EGM & $-$ & $399840$ & $445590$ & $426630$ \\
 & ADMM & $-$  & $148110$ & $124380$ & $132360$ \\
\bottomrule
\end{tabular}
}
\caption{Number of iterations required for the KKT error (the maximal absolute primal residual, dual residual and primal dual gap) to fall below $10^{-6}$.
\blue{The stopping condition is checked every $30$ iterations.} 
The $-$ symbol indicates the iteration limit \blue{($500,000$ for PDHG and EGM; $200,000$ for ADMM as its iterations are slower)} was reached.
The best fixed frequency restart scheme is the fixed frequency restart scheme from Figure~\ref{fig:pdhg-experiments} that requires the least iterations to achieve $10^{-6}$ KKT error. Flexible restarts are described in Remark~\ref{rem:flexible}.
}
\label{table:kkt-error}
\end{table}

\section*{Acknowledgements}

We thank Brendan O’Donoghue, Vasilis Charisopoulos, and Warren Schudy for useful discussions and feedback on this work.

\bibliographystyle{amsplain}
\bibliography{bio.bib}

\newpage

\appendix

\section{Proof of Lemma~\ref{prop:singular-value}}\label{sec:proof-of:lem:minimum-nonsingular-value}

\begin{proof}
Let $U \Sigma V^{\T}$ be a singular value decomposition of $H$. Recall $U$ and $V$ are orthogonal matrices ($U^{\T} U = U U^{\T} = \eye$ and $\| U z \| = \| z \|$), and $\Sigma$ is diagonal.  First, by assumption there exists some $\bar{z}$ such that $H \bar{z} = h$. This implies that for $\bar{p} = V^{\T} \bar{z}$ we have 
$U^{\T} h = U^{\T} (U \Sigma V^{\T}) \bar{z} = \Sigma V^{\T} \bar{z} = \Sigma \bar{p}$.
Consider some $z \in \R^{n}$. Define $p = V^{\T} z$,
$$
p^{*}_i = \begin{cases} \Sigma_{ii}^{-1} (U^{\T} h)_i & \Sigma_{ii} \neq 0 \\
z_i & \text{otherwise}
\end{cases}
$$
and $\zStar = V p^{*}$. If $\Sigma_{ii} \neq 0$ then by definition of $p^{*}$, $(\Sigma p^{*})_i = (U^{\T} h)_i$. If $\Sigma_{ii} = 0$ then $(U^{\T} h)_i = (\Sigma \bar{p})_i = \Sigma_{ii} \bar{p}_i = 0 = \Sigma_{ii} p^{*}_i = (\Sigma p^{*})_i$. Hence $\Sigma p^{*} = U^{\T} h$. Therefore, 
\begin{flalign*}
\| \zStar - z \| &= \| V^{\T} \zStar - V^{\T} z \| \\
&= \| p^{*} - p \| \\
&\le \frac{1}{\sigmaMin{H}} \| \Sigma p^{*} - \Sigma p \| \\
&= \frac{1}{\sigmaMin{H}} \| (U^{\T} h) - \Sigma p \| \\
&= \frac{1}{\sigmaMin{H}} \| h - U \Sigma V^{\T} z \| \\
&= \frac{1}{\sigmaMin{H}} \| h - H z \|.
\end{flalign*}
\end{proof}

\section{Proof of Proposition~\ref{prop:show-iterates-are-close} for EGM}\label{sec:proof-of:prop:show-iterates-are-close}
\begin{proof}
Let $\zStar = \min_{z \in \ZStar} \| z^{t} - z \|$. By definition of $\tz^t$ and that $F$ is $\Lip$-Lipschitz, we have
\begin{flalign}
&F(z^{t-1})^{\T} (\tz^t - z^{t-1}) + \frac{1}{2 \eta} \| \tz^t - z^{t-1} \|^2 \notag\\
\le& F(z^{t-1})^{\T} (\zStar - z^{t-1}) + \frac{1}{2 \eta} \| \zStar - z^{t-1} \|^2 \notag \\
=& F(\zStar)^{\T} (\zStar - z^{t-1}) + (F(z^{t-1}) - F(\zStar))^{\T} (\zStar - z^{t-1}) + \frac{1}{2 \eta} \| \zStar - z^{t-1} \|^2 \ .
\label{eq:Fstar-bound}
\end{flalign}
Note that
\begin{flalign}\label{eq:Ft-equality}
\begin{split}
F(z^{t-1})^{\T} (\tz^t - z^{t-1}) =F(\zStar)^{\T} (\tz^t - z^{t-1}) +
(F(z^{t-1}) - F(\zStar))^{\T} (\tz^t - \zStar) + \\
(F(z^{t-1}) - F(\zStar))^{\T} (\zStar - z^{t-1}) \ .
\end{split}
\end{flalign}
Substituting \eqref{eq:Ft-equality} into the LHS of \eqref{eq:Fstar-bound} and cancelling terms yields
\begin{align}\label{eq:F-intermediate}
\begin{split}
&F(\zStar)^{\T} (\tz^t - z^{t-1}) + (F(z^{t-1}) - F(\zStar))^{\T} (\tz^t - \zStar) + \frac{1}{2 \eta} \| \tz^t - z^{t-1} \|^2  \\
&\le F(\zStar)^{\T} (\zStar - z^{t-1}) +  \frac{1}{2 \eta} \| \zStar - z^{t-1} \|^2\ .
\end{split}
\end{align}
By $F(\zStar)^{\T} (\zStar - z^{t-1}) \le F(\zStar)^{\T} (\tz^t - z^{t-1})$, \eqref{eq:F-intermediate}, $F$ is $\Lip$-Lipschitz, and the triangle inequality one gets
\begin{flalign}
0 &\ge F(\zStar)^{\T} (\zStar - z^{t-1}) - F(\zStar)^{\T} (\tz^t - z^{t-1}) \notag \\
&\ge \frac{1}{2 \eta} \| \tz^t - z^{t-1} \|^2 + (F(z^{t-1}) - F(\zStar))^{\T} (\tz^t - \zStar)  - \frac{1}{2 \eta} \| \zStar - z^{t-1} \|^2 \notag \\
&\ge  \frac{1}{2 \eta} \| \tz^t - z^{t-1} \|^2 - \Lip \| z^{t-1} - \zStar \| \| \tz^t - \zStar \| - \frac{1}{2 \eta} \| \zStar - z^{t-1} \|^2 \notag \\
&= \frac{1}{2 \eta} \| \tz^t - z^{t-1} \|^2 - \Lip \| z^{t-1} - \zStar \| \| z^{t-1} - \zStar + \tz^t - z^{t-1} \| - \frac{1}{2 \eta} \| \zStar - z^{t-1} \|^2 \notag \\
&\ge \frac{1}{2 \eta} \| \tz^t - z^{t-1} \|^2 - \Lip \| z^{t-1} - \zStar \| \| \tz^t - z^{t-1} \| - \left( \Lip + \frac{1}{2 \eta} \right) \| \zStar - z^{t-1} \|^2 \notag \\
& = \pran{\frac{1}{2\eta} \| \tz^t - z^{t-1} \| - \pran{\frac{1}{2\eta}+\Lip}\| \zStar - z^{t-1} \| } \pran{\| \tz^t - z^{t-1} \|+\| \zStar - z^{t-1} \|} \
\label{eq:quadratic-inequality}
\end{flalign}
where the first inequality uses 
$F(\zStar)^{\T} (\zStar - z^{t-1}) \le F(\zStar)^{\T} (\tz^t - z^{t-1})$, the second inequality uses \eqref{eq:F-intermediate}, the third inequality uses that $F$ is $\Lip$-Lipschitz, and the final inequality uses the triangle inequality.

Notice that $\| \tz^t - z^{t-1} \|+\| \zStar - z^{t-1} \| > 0$ otherwise the result trivially holds, thus by \eqref{eq:quadratic-inequality} we get $\frac{1}{2\eta} \| \tz^t - z^{t-1} \| - \pran{\frac{1}{2\eta}+\Lip}\| \zStar - z^{t-1} \|\le 0$, rearranging yields
$$
\| \tz^t - z^{t-1} \| \le (1+ 2\eta \Lip) \| \zStar - z^{t-1} \|\le 3 \| \zStar - z^{t-1} \|\ ,
$$
where the last inequality uses that $\eta \le 1 / \Lip$ from the premise of Proposition~\ref{prop:assumption-one-holds}.
\end{proof}

\section{Proof of Lemma~\ref{example:admm}}\label{app:admm-lp-proof}

\begin{proof}
Recall for ADMM the norm is $\| z \| = \sqrt{z^{\T} M z}$ with
$$M=
\begin{pmatrix}
0 & 0 & 0 \\
0 & \eta \V^{\T} \V & 0 \\
0 & 0 & \frac{1}{\eta} \eye
\end{pmatrix} = \begin{pmatrix}
0 & 0 & 0 \\
0 & \eta \eye & 0 \\
0 & 0 & \frac{1}{\eta} \eye
\end{pmatrix}\ ,
$$
where we ustilize $V=I$ in \eqref{eq:lp-ADMM}.
Suppose that $\eta = 1$.
Consider $z \in Z$, $r \in (0, 2R]$. Let
$$
v = \begin{pmatrix}
0 \\
(c + y)^{-} \\
x_{V} - x_{U}
\end{pmatrix} \text{ and note that } F(z) = \begin{pmatrix}
-y \\
c + y \\
x_{U} - x_{V}
\end{pmatrix}.
$$

Let $z_1 = z + r \frac{v}{\| v \|}$, then $z_1\in Z$ by the definition of $Z$. It follows from \eqref{eq:lp-0} that
\begin{align}\label{eq:admm-lp-1}
    \begin{split}
        \rho_r(z) & \ge \frac{F(z)^{\T} (z - z_1)}{r} = -\frac{1}{\| v \|} F(z)^{\T} v =  \left\| v \right\|_{2} \ .
    \end{split}
\end{align}

Consider any optimal solution $\zStar = (\xStar_{V}, \xStar_{V}, \yStar)$, define $\hat{z} = (\xStar_{V}, \xStar_{V}, y) \in \ZStar$.
Let $z_2= z - \mu (z - \hat{z})$ for $\mu = \min\left\{ \frac{r}{\| z - \hat{z} \|}, 1 \right\}$, then $\|z_2-z\|\le \frac{r}{\|  z - \hat{z} \|}\| z - \hat{z} \|\le r$. Meanwhile, we have $z_2 = (1- \mu) z +  \mu \hat{z}$, thus $z_2\in Z$ by convexity.
We conclude that $z_2 \in \zBall{r}{z}$.
Therefore, it follows from \eqref{eq:lp-0} that
\begin{align}\label{eq:admm-lp-2}
    \begin{split}
        \rho_r(z) & \ge \frac{1}{r} \min\left\{ \frac{r}{\|  z - \hat{z} \|}, 1 \right\} 
        F(z)^{\T}  (z - \hat{z}) = \min\left\{ \frac{1}{\|  z - \hat{z} \|}, \frac{1}{r} \right\} \begin{pmatrix}
        -y \\
        c+y \\
        x_{U} - x_{V}
        \end{pmatrix} \cdot \begin{pmatrix}
        x_U - \xStar_V \\
        x_{V} - \xStar_V \\
        y
        \end{pmatrix} \\
        &= \min\left\{ \frac{1}{\|  z - \hat{z} \|}, \frac{1}{r} \right\} 
         c^{\T} (x_{V} -  \xStar_{V})  \ .
    \end{split}
\end{align}
Substituting $r \in (0, 2R]$ and $\| z - \hat{z} \| \le \| z - \zStar \| \in [0,2R]$ into \eqref{eq:admm-lp-2} and noticing $\rho_r(z)\ge 0$ yields
\begin{align}\label{eq:admm-lp-3}
    \begin{split}
        \rho_r(z) & \ge \frac{1}{2R}
         (c^{\T} x_{V} -  c^{\T} \xStar_{V})^+ \ .
    \end{split}
\end{align}
Combining \eqref{eq:admm-lp-1} and \eqref{eq:admm-lp-3}, we deduce there exists $K'$ and $h'$ such that
\begin{flalign*}
    (4R^2 +1)\rho_r(z)^2 &\ge ((c^{\T} x_{V} -  c^{\T} \xStar_{V})^+)^2 + \| (c + y)^{-} \|_2^2 + \| x_V - x_U \|_2^2 \\
    &= ((c^{\T} x_{V} -  c^{\T} \xStar_{V})^+)^2 + \| (c + y)^{-} \|_2^2 + \| x_V - x_U \|_2^2 +  \| x_U^{-} \|_2^2 + \| A x_{V} - b \|_2^2  \\
    &= \|(h'-K' z )^+\|^2_2 \ge \frac{1}{H(K')^2}\dist(z,\ZStar)^2 \ ,
\end{flalign*}
where the last inequality is from duality, Hoffman condition (Equation~\eqref{eq:hoffman}), and the fact that $\| z \|_2 \ge \| z \|$. Taking the square root obtains the result for $\eta = 1$.

Next, we consider the case $\eta \neq 1$. Let us denote the corresponding normalized duality gap as $\rho_r^{\eta}$ and distance as $\dist^{\eta}$ then with $\theta = \max\{ \eta, 1/\eta \}$ we get
$$
\rho_{r}^{\eta}(z) \ge \frac{\rho_{\theta r}^{1}(z)}{\theta}  \ge \frac{1}{\theta H(K') \sqrt{4 R^2 + 1}} \mathbf{dist}^{1}(z, \ZStar) \ge \frac{1}{\theta^2 H(K') \sqrt{4 R^2 + 1}} \mathbf{dist}^{\eta}(z, \ZStar) \ .
$$
\end{proof}

\section{Proof of Proposition~\ref{thm:admm-egm:bound-adaptive-restart-scheme}}\label{sec:proof-of-bound-adaptive-restart-scheme}

The next lemma is used in the proof of Proposition~\ref{thm:admm-egm:bound-adaptive-restart-scheme}.

\begin{lemma}\label{lem:bound-adaptive-restart-scheme}
Consider the sequence $\{ z^{n, 0} \}_{i=0}^{\infty}$, $\{ \len{n} \}_{i=1}^{\infty}$ generated by Algorithm~\ref{al:restarted+algorithm} with the adaptive restart scheme and $\beta<\frac{1}{q+3}$. Suppose \eqref{eq:pda} satisfies Property \ref{property:restart-assumption}, and there exists $a \in (0,\infty)$ s.t. $\rho_r(z) \ge \frac{a}{1 + \| z \|} \dist(z, \ZStar)$.
Then $z^{n,0}$ stays in a bounded region, in particular, $z^{n,0} \in \zBall{R}{z^{0,0}}$ for all $n \in \N$ with $R = \frac{2C (q + 3)^3}{\beta a (1 - \beta (q+3))}  \dist( z^{0, 0}, \ZStar) (1 + \| z^{0, 0} \| + \dist(z^{0,0}, \ZStar))$.
\end{lemma}

\begin{proof}

First, notice that it holds for any $n>0$ that
$\dist(z^{n, 0}, \ZStar) \le \dist(z^{n-1, 0}, \ZStar) + \| z^{n, 0} - z^{n-1,0} \| \le (q + 3) \dist(z^{n-1, 0}, \ZStar)$
where the first inequality uses the triangle inequality, and the second inequality is from Property~\ref{subproperty:distance-bound}.
Recursively applying this inequality yields
\begin{align}
\label{eq:bound-distance}
\dist(z^{n, 0}, \ZStar) \le (q + 3)^n \dist(z^{0, 0}, \ZStar) \ .
\end{align}
Thus, we have
\begin{flalign}
\| z^{N, 0} - z^{0, 0} \| &\le \sum_{n=0}^{N-1} \| z^{n+1, 0} - z^{n,0} \| \le (q + 3) \sum_{n=0}^{N-1} \dist(z^{n, 0}, \ZStar) \notag \\
&\le \dist(z^{0,0}, \ZStar) \sum_{n=0}^{N-1} (q+3)^{n+1} = \dist(z^{0,0}, \ZStar) \frac{(q + 3)^{N + 1} - (q+3)}{q+2} \notag  \\
&\le \frac{(q + 3)^{N + 1}}{2} \dist(z^{0,0}, \ZStar) \ , \label{eq:bound-total-distance-travelled-by-q}
\end{flalign}
where the first inequality uses the triangle inequality, the second inequality is from Property~\ref{subproperty:distance-bound}, the third inequality utilizes \eqref{eq:bound-distance} for $n=0,...,N-1$, the equality uses the formula for the sum of a geometric series, and the last inequality uses $q\ge 0$.

Furthermore, we have
\begin{flalign}\label{eq:bound-rho-n-by-dist}
\rho_{\| z^{n, 0} - z^{n-1, 0} \|}(z^{n,0}) &\le \beta^{n-1} \rho_{\| z^{1, 0} - z^{0, 0} \|}( z^{1,0} ) \le \beta^{n-1} 2 C \| z^{1, 0} - z^{0, 0} \| \notag \\
&\le \beta^{n-1} 2 C (q + 2)  \dist( z^{0, 0}, \ZStar) \ ,
\end{flalign}
where the first inequality recursively uses \eqref{eq:adaptive-restart-scheme}, the second inequality is from Property~\ref{subproperty:reduce-potential-function} and $\len{0} \ge 1$, and the last inequality uses Property~\ref{subproperty:distance-bound}.

Therefore, it holds that
\begin{align}\label{eq:bound-1}
\begin{split}
    \| z^{n+1,0} - z^{n,0} \| &\le (q + 2) \dist(z^{n,0}, \ZStar)  \\
&\le (q + 2) \frac{\rho_{\| z^{n, 0} - z^{n-1, 0} \|}(z^{n,0})}{a} \left( 1 + \| z^{n,0} \|  \right) \\
&\le \beta^{n-1} \frac{2 C (q + 2)^2 \dist( z^{0, 0}, \ZStar)}{a}  \left( 1 + \| z^{n,0} \|  \right) \\
&\le \beta^{n-1} \frac{2 C (q + 2)^2 \dist( z^{0, 0}, \ZStar)}{a}  \left( 1 + \| z^{0, 0} \| + \| z^{n,0} - z^{0,0} \|   \right) \\
&\le \beta^{n-1}  \frac{2 C (q + 2)^2 \dist( z^{0, 0}, \ZStar)}{a} \left(1 + \| z^{0, 0} \| + \frac{(q + 3)^{n+1}}{2} \dist(z^{0,0}, \ZStar) \right) \\
&\le (\beta (q+3))^{n} \frac{C (q + 3)^3}{\beta a}  \dist( z^{0, 0}, \ZStar) \left(1 + \| z^{0, 0} \| + \dist(z^{0,0}, \ZStar) \right) \ ,
\end{split}
\end{align}
where the first inequality uses Property~\ref{subproperty:distance-bound}, the second inequality is from that $\rho_r(z) \ge \frac{a}{1 + \| z \|} \dist(z, \ZStar)$, the third inequality uses \eqref{eq:bound-rho-n-by-dist}, the fourth inequality uses the triangle inequality, the fifth inequality uses \eqref{eq:bound-total-distance-travelled-by-q}, and the last inequality uses $q+3 \ge 2$.

Therefore, we have

\begin{flalign*}
\| z^{N, 0} - z^{0, 0} \| &\le \sum_{n=0}^{N-1} \| z^{n+1,0} - z^{n,0} \| \\
&\le \sum_{n=0}^{N-1}  (\beta (q+3))^{n} \frac{C (q + 3)^3}{\beta a}  \dist( z^{0, 0}, \ZStar) \left(1 + \| z^{0, 0} \| + \dist(z^{0,0}, \ZStar) \right) \\
&\le \frac{C (q + 3)^3}{\beta a (1 - \beta (q+3))}  \dist( z^{0, 0}, \ZStar) \left(1 + \| z^{0, 0} \| + \dist(z^{0,0}, \ZStar) \right) \ ,
\end{flalign*}
where the first inequality is the triangle inequality, the second inequality uses \eqref{eq:bound-1}, and the last inequality is from the bound on the sum of a geometric series by noticing $\beta(q+3)<1$. This finishes the proof.
\end{proof}

\emph{Proof of Proposition~\ref{thm:admm-egm:bound-adaptive-restart-scheme} for EGM and ADMM appled to LP.}
Recall that the Lagrangian form of an LP instance is $\alpha$-sharp on $S(R)=\{z \in Z : \| z \| \le R\}$ with $\alpha=\frac{1}{H(K)\sqrt{1+4R^2}}$ (see Lemma~\ref{example:lp}) and the ADMM form of a LP instance is $\alpha$-sharp on $S(R)=\{z \in Z : \| z \| \le R\}$ with $\alpha=\frac{1}{ \max\{ \eta^2, 1/\eta^2 \} H(K')\sqrt{1+4R^2}}$ (see Lemma~\ref{example:admm}).
Thus the sharpness constant $\alpha$ satisfies the condition in Lemma~\ref{lem:bound-adaptive-restart-scheme} with $a=\frac{1}{2 H(K)}$ for standard form LP (Lemma~\ref{example:lp}) and and $a=\frac{1}{2 \max\{ \eta^2, 1/\eta^2 \} H(K')}$ for ADMM form LP. We finish the proof by setting $R = \frac{2C (q + 3)^3}{\beta a (1 - \beta (q+3))}  \dist( z^{0, 0}, \ZStar) (1 + \| z^{0, 0} \| + \dist(z^{0,0}, \ZStar))$ as stated in Lemma \ref{lem:bound-adaptive-restart-scheme}. \qed

\section{Proof of Corollary~\ref{coro:pdhg-with-euclidean-norm}}\label{sec:pdhg-with-euclidean-norm}

\begin{proof}
Let $\| z \|_M = \sqrt{\| x \|^2 - 2 \eta y^{\T} A x + \| y \|^2}$. Recall that  Property~\ref{property:restart-assumption} holds for $\| \cdot \|_M$ with $q = 0$ and $C = 2 / \eta$ (refer to Table~\ref{tbl:summary-of-C-and-q-values}).
Define $\hat{B}_{r}(z) = \{ z \in Z : \| z \|_{M} \le r \}$ then with $\hat{r} = r \sqrt{\frac{1}{1 - \eta \sigmaMax{A}}}$ for any $z \in Z$ we get
\begin{flalign*}
\rho_r(z) &= \frac{\max_{\tz \in \zBall{r}{z}}\{ \Lag(x, \ty) - \Lag(\tx, y)\} }{r} \\
&\le \frac{\max_{\tz \in \hat{B}_{\hat{r}}(z)}\{ \Lag(x, \ty) - \Lag(\tx, y)\} }{r} \\
&= \sqrt{\frac{1}{1 - \eta \sigmaMax{A}}} \frac{\max_{\tz \in \hat{B}_{\hat{r}}(z)}\{ \Lag(x, \ty) - \Lag(\tx, y)\} }{\hat{r}} \\
&\le  \frac{1}{\eta} \sqrt{\frac{1}{1 - \eta \sigmaMax{A}}} \frac{2 \hat{r}}{t} = \frac{1}{\eta (1 - \eta \sigmaMax{A})} \frac{2 r}{t}
\end{flalign*}
where the first inequality uses that $\zBall{r}{z} \subseteq \hat{B}_{\hat{r}}(z)$ by Proposition~\ref{fact:pdhg-norm}, the second inequality uses Property~\ref{subproperty:reduce-potential-function}.
This establishes $C = \frac{2}{\eta (1 - \eta \sigmaMax{A})}$.
Moreover, with $\zStar \in \argmin_{z \in \ZStar} \| z^{0} - z \|_M$ we have
$$
\| \bz^{t} - z^{0} \|^2 \le \frac{\| \bz^{t} - z^{0} \|_M^2}{1 - \eta \sigmaMax{A}}  \le \frac{4 \| \zStar - z^{0} \|_{M}^2}{1 - \eta \sigmaMax{A}}  \le \frac{4 (1 + \eta \sigmaMax{A})}{1 - \eta \sigmaMax{A}} \| \zStar -  z^{0} \|^2
$$
where the first inequality uses by Proposition~\ref{fact:pdhg-norm}, the second inequality uses Property~\ref{subproperty:distance-bound}, and the third inequality uses Proposition~\ref{fact:pdhg-norm}.
This establishes $q = 4 \frac{1 + \eta \sigmaMax{A}}{1 - \eta \sigmaMax{A}} - 2$.
\end{proof}

\section{Linear time algorithm for linear trust region problems}
\label{appendix:linear-time-implementation:trustregion}

\newcommand{\lo}{{\text{lo}}}
\newcommand{\mi}{{\text{mid}}}
\newcommand{\hi}{{\text{hi}}}

\begin{algorithm}[]
 \SetAlgoLined
 \SetKwInput{KwAssumptions}{Assumptions}
 \SetKwInput{KwInvariants}{Invariants}
 \DontPrintSemicolon

 \KwIn{Center $z \in \R^{n+m}$, lower bound $l \in (\R \cup \{-\infty\})^{n+m}$, objective $g \in \R^{n+m}$, radius $r \ge 0$}
 \KwAssumptions{$l \le z$, $g > 0$, $r \ge 0$}
 \KwResult{Trust region solution $\tz$ solving (\ref{eq:lp-trust-region-problem})}
 \lIf{$\|l - z\| \le r$}{\Return{$\tz(\infty) = l$}} \nllabel{al:linear-trust-region:return-l}
 $\lambda_\lo \gets 0$,
 $\lambda_\hi \gets \infty$\;
 $f_\lo \gets \sum_{i : \hat{\lambda}_i \leq \lambda_\lo} (l_i - z_i)^2$,
 $f_\hi \gets \sum_{i : \hat{\lambda}_i \geq \lambda_\hi} g_i^2$\;
 $\mathcal{I} \gets \{i\;:\;\lambda_\lo < \hat{\lambda}_i < \lambda_\hi\}$\;
 \While{$\mathcal{I} \neq \emptyset$}{\nllabel{al:linear-trust-region:while}
  \KwInvariants{$\mathcal{I} = \{i\;:\;\lambda_\lo < \hat{\lambda}_i < \lambda_\hi\}$ \linebreak
  For $\lambda_\lo < \lambda < \lambda_\hi$, $\| \tz(\lambda) - z \|^2 = f_\lo + \lambda^2 f_\hi + \sum_{i \in \mathcal{I}} (\tz(\lambda)_i - z_i)^2$ \linebreak
  $\|\tz(\lambda_\lo) - z\|^2 \le r^2 \le \|\tz(\lambda_\hi) - z\|^2$}
  $\lambda_\mi \gets \text{median}(\hat{\lambda}_i : i \in \mathcal{I})$
  \nllabel{al:linear-trust-region:median}\;
  $f_\mi \gets f_\lo + \sum_{i \in \mathcal{I}} (\tz(\lambda_\mi)_i - z_i)^2 + f_\hi \lambda_\mi^2$ \tcp*{$f_\mi = \|\tz(\lambda_\mi) - z\|^2$}
  \uIf{$f_\mi < r^2$} {
   $\lambda_\lo \gets \lambda_\mi$\;
   $f_\lo \gets f_\lo + \sum_{i \in \mathcal{I} : \hat{\lambda}_i \leq \lambda_\mi} (l_i - z_i)^2$\;
   $\mathcal{I} \gets \{i \in \mathcal{I} : \hat{\lambda}_i > \lambda_\mi\}$\;
  }
  \Else{
   $\lambda_\hi \gets \lambda_\mi$\;
   $f_\hi \gets f_\hi + \sum_{i \in \mathcal{I} : \hat{\lambda}_i \ge \lambda_\mi} g_i^2$\;
   $\mathcal{I} \gets \{i \in \mathcal{I} : \hat{\lambda}_i < \lambda_\mi\}$\;
  }
 }
 $\lambda_\mi \gets \sqrt{(r^2 - f_\lo)/f_\hi}$
 \nllabel{al:linear-trust-region:final}
 \tcp*{solving $r^2 = \|\tz(\lambda_\mi) - z\|^2 =  f_\lo + \lambda^2 f_\hi$}
 \Return{$\tz(\lambda_\mi)$}\;
 \caption{Linear objective trust region algorithm}
 \label{al:linear-trust-region}
\end{algorithm}

\begin{theorem}Algorithm~\ref{al:linear-trust-region} exactly solves (\ref{eq:lp-trust-region-problem}), the trust region problem with linear objective, Euclidean norm, and bounds.
\end{theorem}
\begin{proof}
If the algorithm returns from line~\ref{al:linear-trust-region:return-l}, (\ref{eq:lp-trust-region:lambda}) is unbounded, and the algorithm returns $\tz(\infty)$.  Otherwise, in each iteration of the while loop (line~\ref{al:linear-trust-region:while}), the algorithm maintains the invariants
\begin{itemize}
    \item $\mathcal{I} = \{i\;:\;\lambda_\lo < \hat{\lambda}_i < \lambda_\hi\}$
    \item For $\lambda_\lo < \lambda < \lambda_\hi$, $\| \tz(\lambda) - z \|^2 = f_\lo + \lambda^2 f_\hi + \sum_{i \in \mathcal{I}} (\tz(\lambda)_i - z_i)^2$
    \item $\|\tz(\lambda_\lo) - z\|^2 \le r^2 \le \|\tz(\lambda_\hi) - z\|^2$
\end{itemize}
The while loop (line~\ref{al:linear-trust-region:while}) is finite, since on each iteration $|\mathcal{I}|$ is reduced by at least a factor of 2.  When the while loop exits (with $\mathcal{I} = \emptyset$), these invariants mean that the final $\lambda_\mi$ computed on line~\ref{al:linear-trust-region:final} optimizes \eqref{eq:lp-trust-region:lambda} and the returned $\tz(\lambda_\mi)$ solves \eqref{eq:lp-trust-region-problem}.
\end{proof}

\begin{theorem}Algorithm~\ref{al:linear-trust-region} runs in $O(m+n)$ time.
\end{theorem}
\begin{proof}
The work outside the while loop (line~\ref{al:linear-trust-region:while}) is clearly $O(m+n)$.  In each pass through the while loop, the median (line~\ref{al:linear-trust-region:median}) can be found in $O(|\mathcal{I}|)$ time~(\cite{BFPRT1973}), and the rest of the loop also takes $O(|\mathcal{I}|)$ time.  Since at least half of the elements of $\mathcal{I}$ are removed in each iteration, and initially $|\mathcal{I}| \le m+n$, the total time in the while loop is at most $O(\sum_{i=0}^\infty \frac{m+n}{2^i}) = O(m+n)$ by the formula for the sum of an infinite geometric series.
\end{proof}

\pagebreak 

\section{Numerical results for restarted EGM and ADMM}\label{sec:num-egm-and-admm}

\blue{
In this section, we repeat the numerical experiments of Section \ref{sec:experimental-results} with EGM and ADMM instead of PDHG.
We tune the primal weight for EGM and the step size $\eta$ for ADMM using the same procedure as with PDHG (see Section \ref{sec:experimental-results}).

\begin{figure}
    \centering
    \includegraphics[height=150pt]{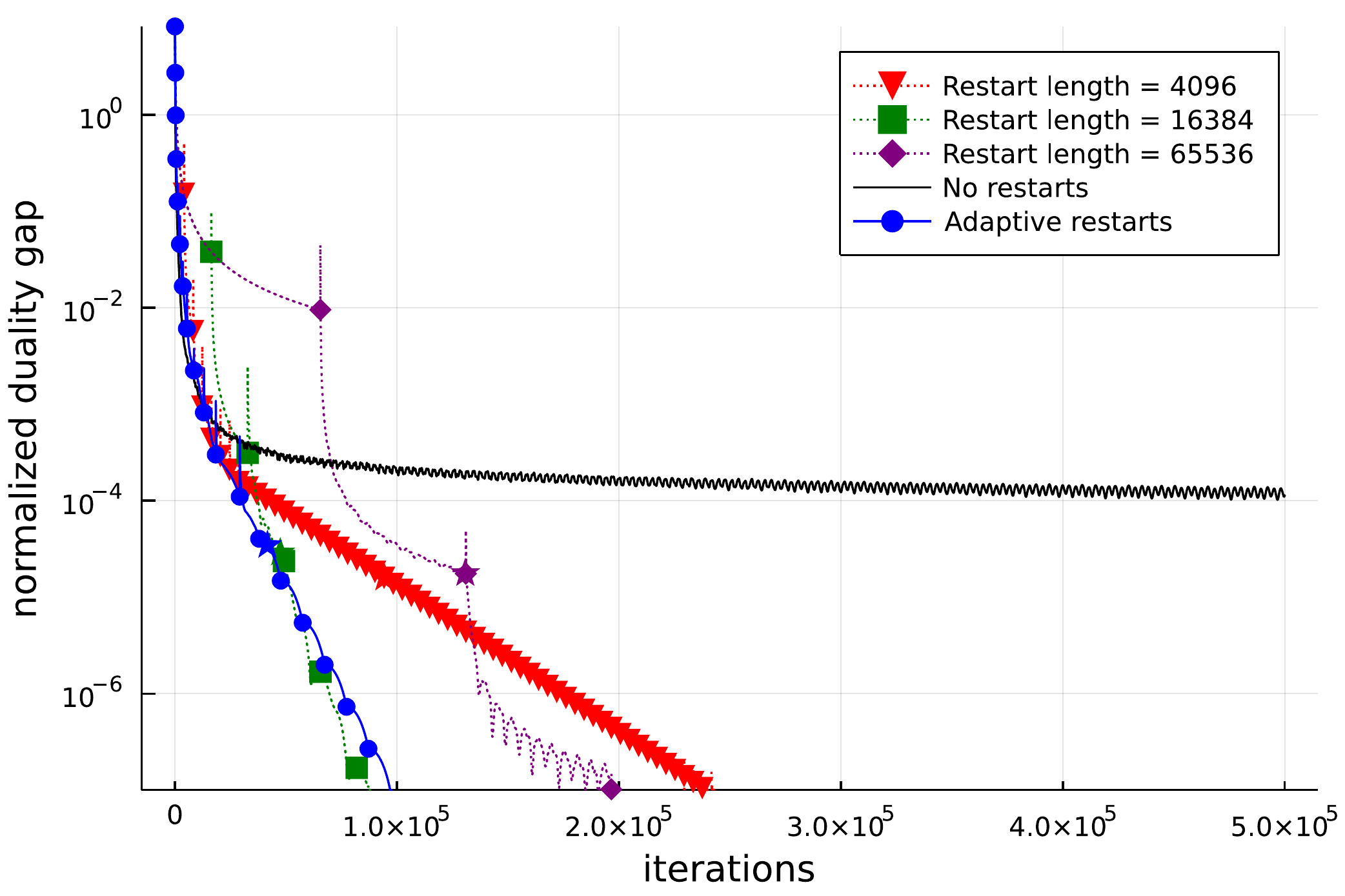}
    \includegraphics[height=150pt]{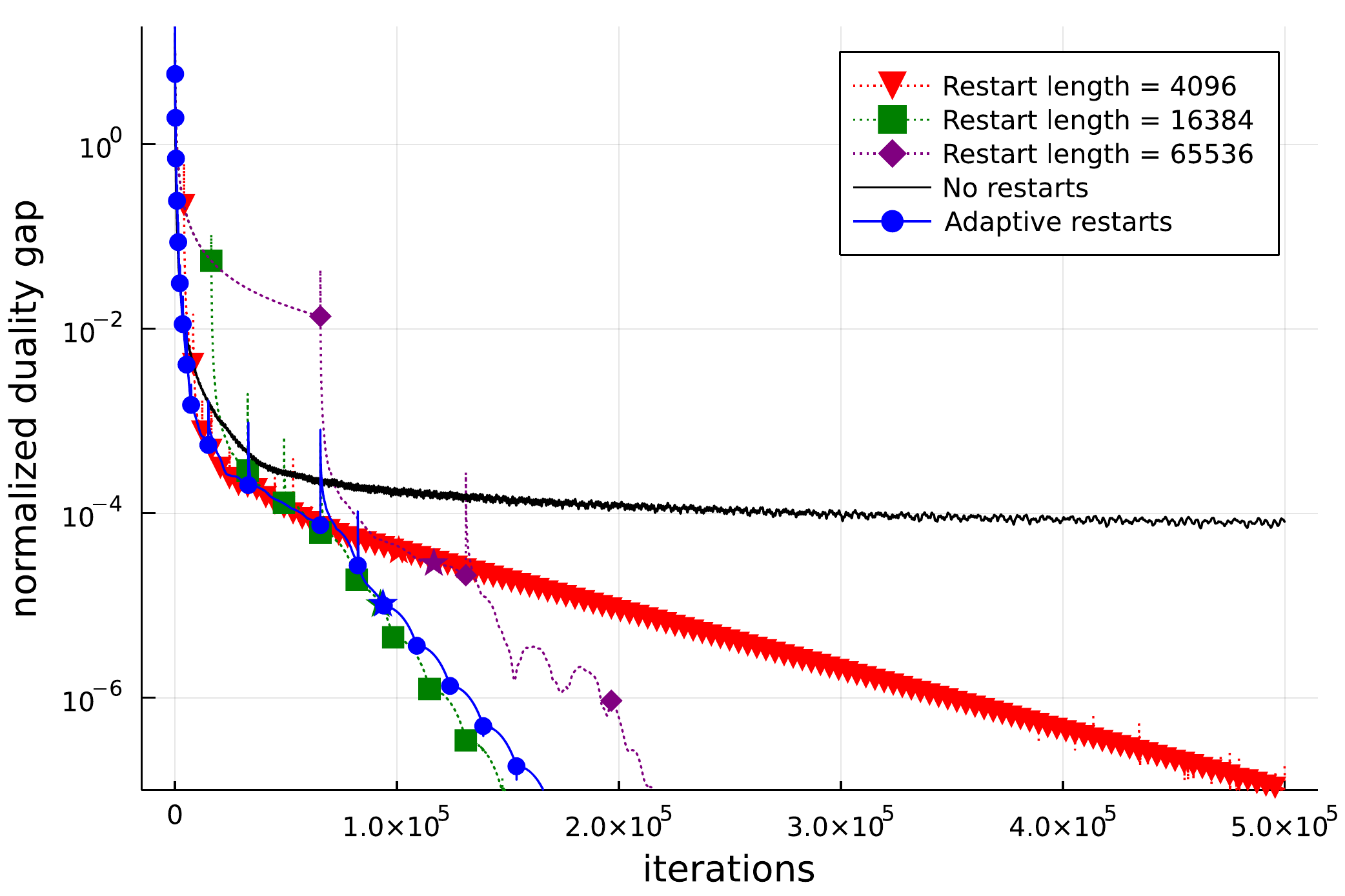}
    \includegraphics[height=150pt]{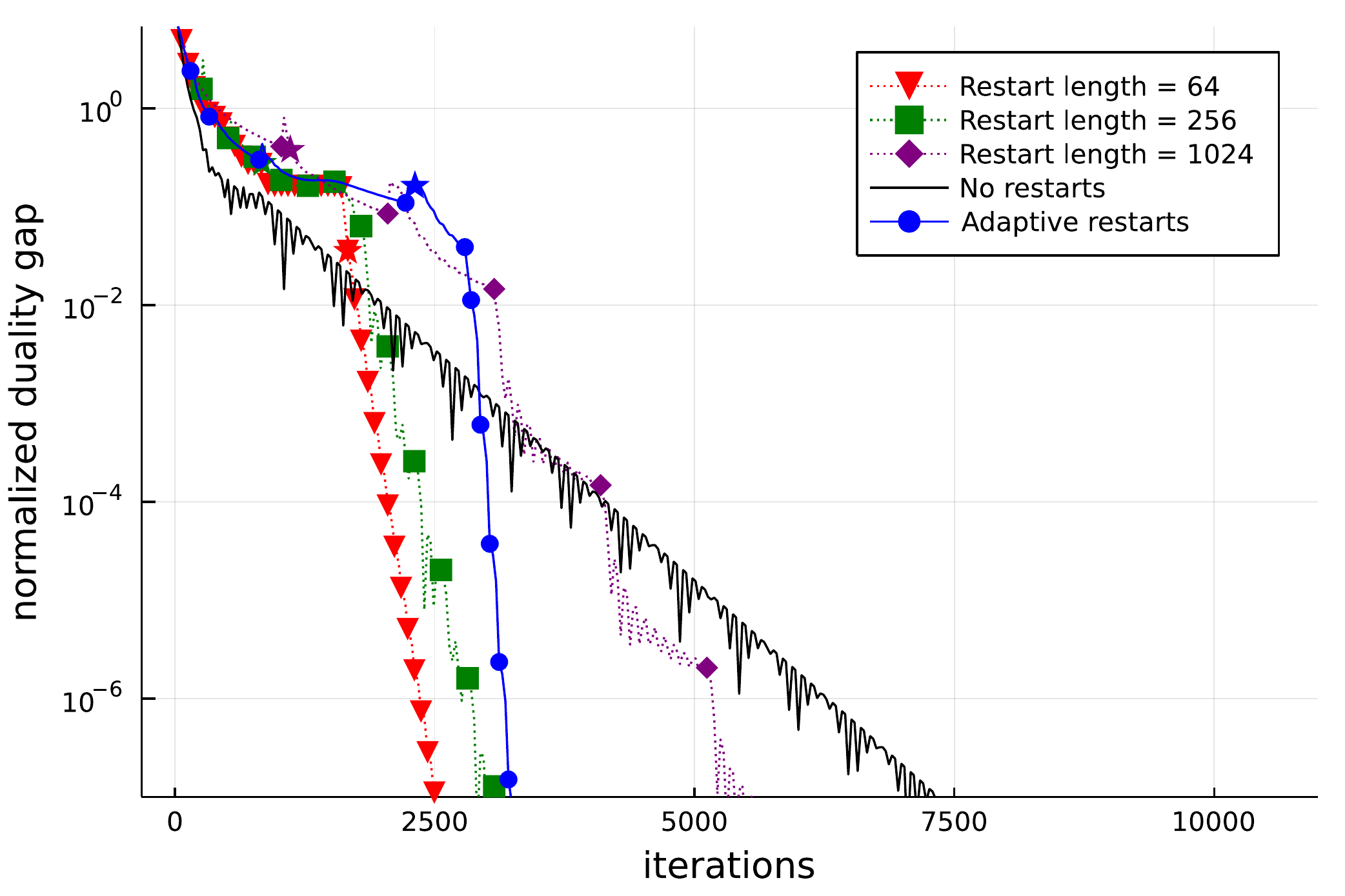}
    \includegraphics[height=150pt]{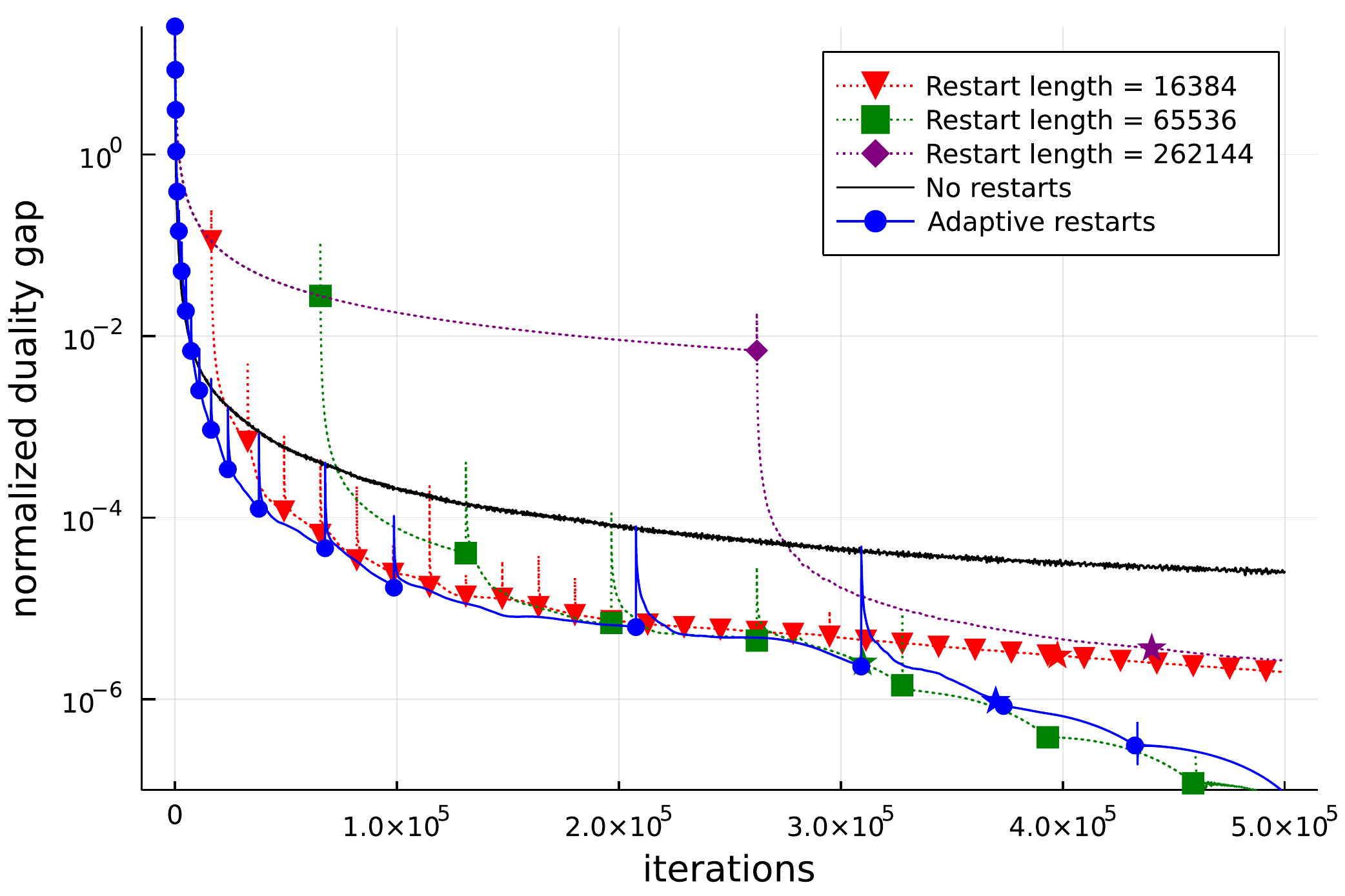} \\
    \caption{Plots show normalized duality gap (in $\log$ scale) versus number of iterations for (restarted) EGM.  Results from left to right, top to bottom for qap10, qap15, nug08-3rd, and nug20. Each plot compares no restarts, adaptive restarts, and the three best fixed restart lengths.
    For the restart schemes we evaluate the normalized duality gap at the average, for no restarts we evaluate it at the last iterate as the average performs much worse.
    A star indicates the last iteration that the active set changed before the iteration limit was reached.
    Series markers indicate when restarts occur.}
    \label{fig:egm-experiments}
\end{figure}

Figure \ref{fig:egm-experiments} (EGM) is almost identical to Figure \ref{fig:pdhg-experiments} (PDHG), which verifies again the improved performance of restarted algorithms. Furthermore, note that one iteration of EGM is about twice as expensive as a PDHG iteration, because EGM requires four matrix-vector multiplications per iteration, while PDHG requires two matrix-vector multiplications.

Indeed, to see the difference between  EGM and  PDHG 
requires a close examination of Table~\ref{table:kkt-error}. Recall that termination criteria is only checked every 30 iterations, which causes the iteration counts to be identical in several instances.
}

\blue{
\begin{figure}
    \centering
    \includegraphics[height=150pt]{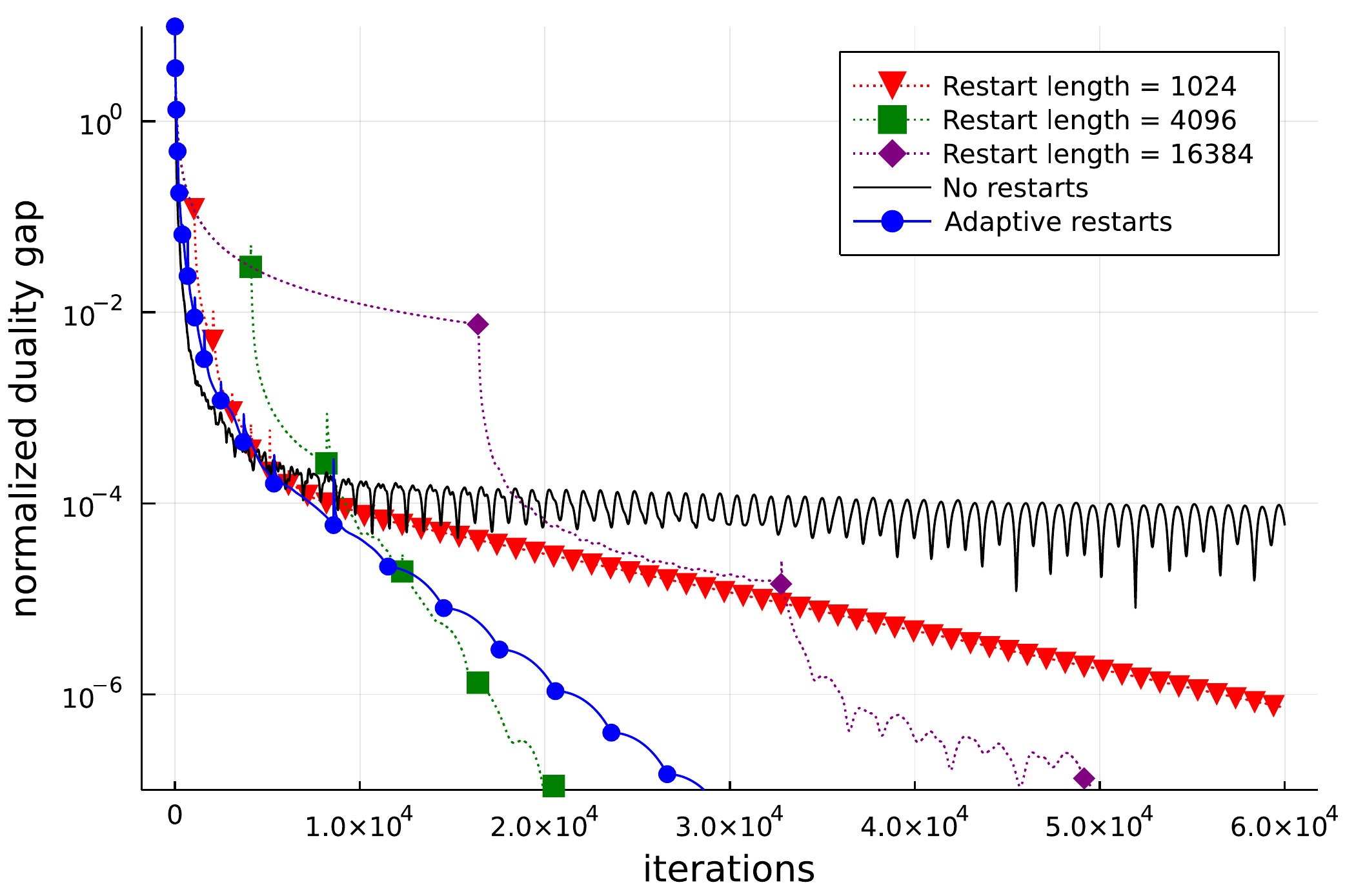}
    \includegraphics[height=150pt]{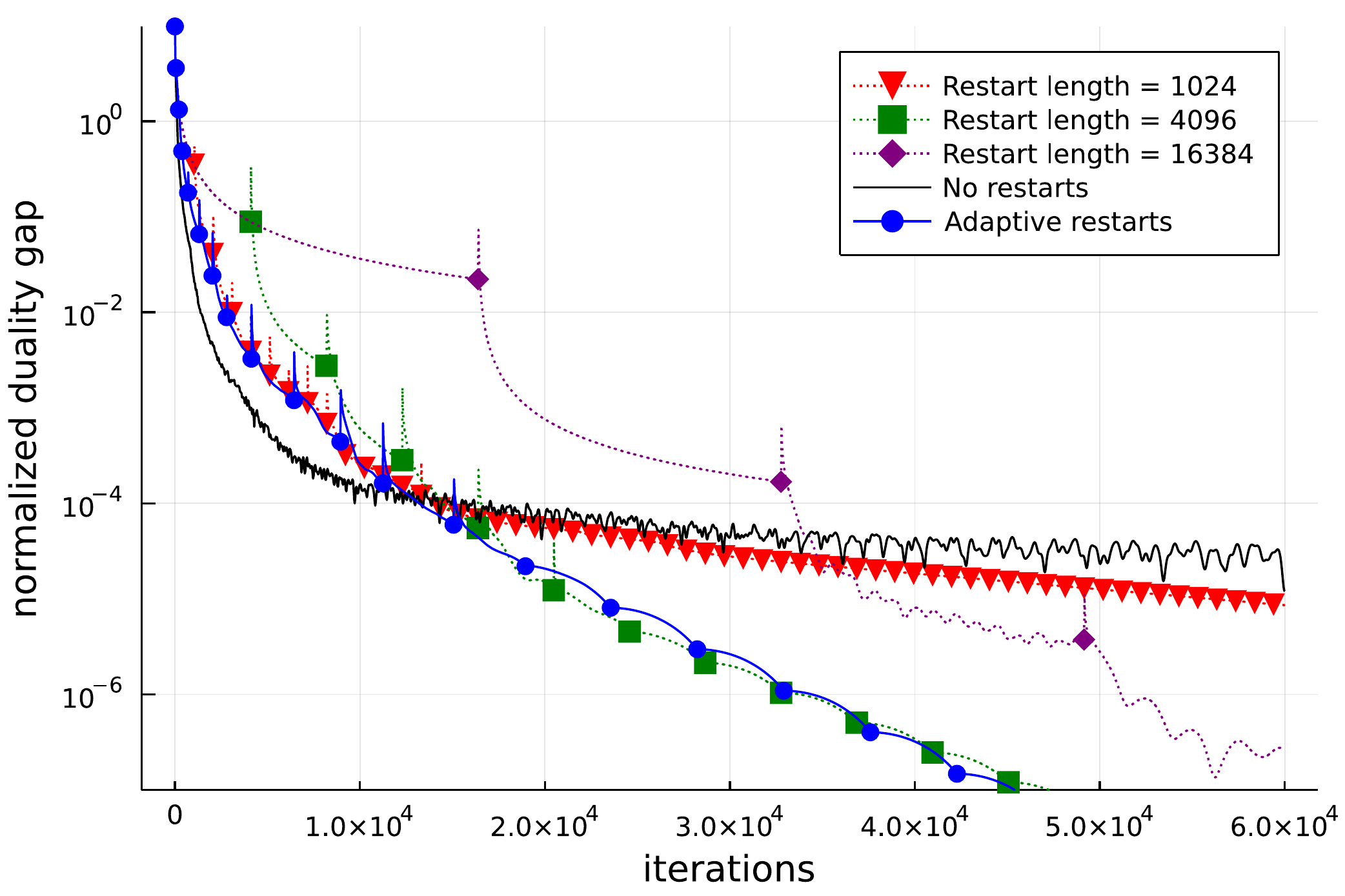}
    \includegraphics[height=150pt]{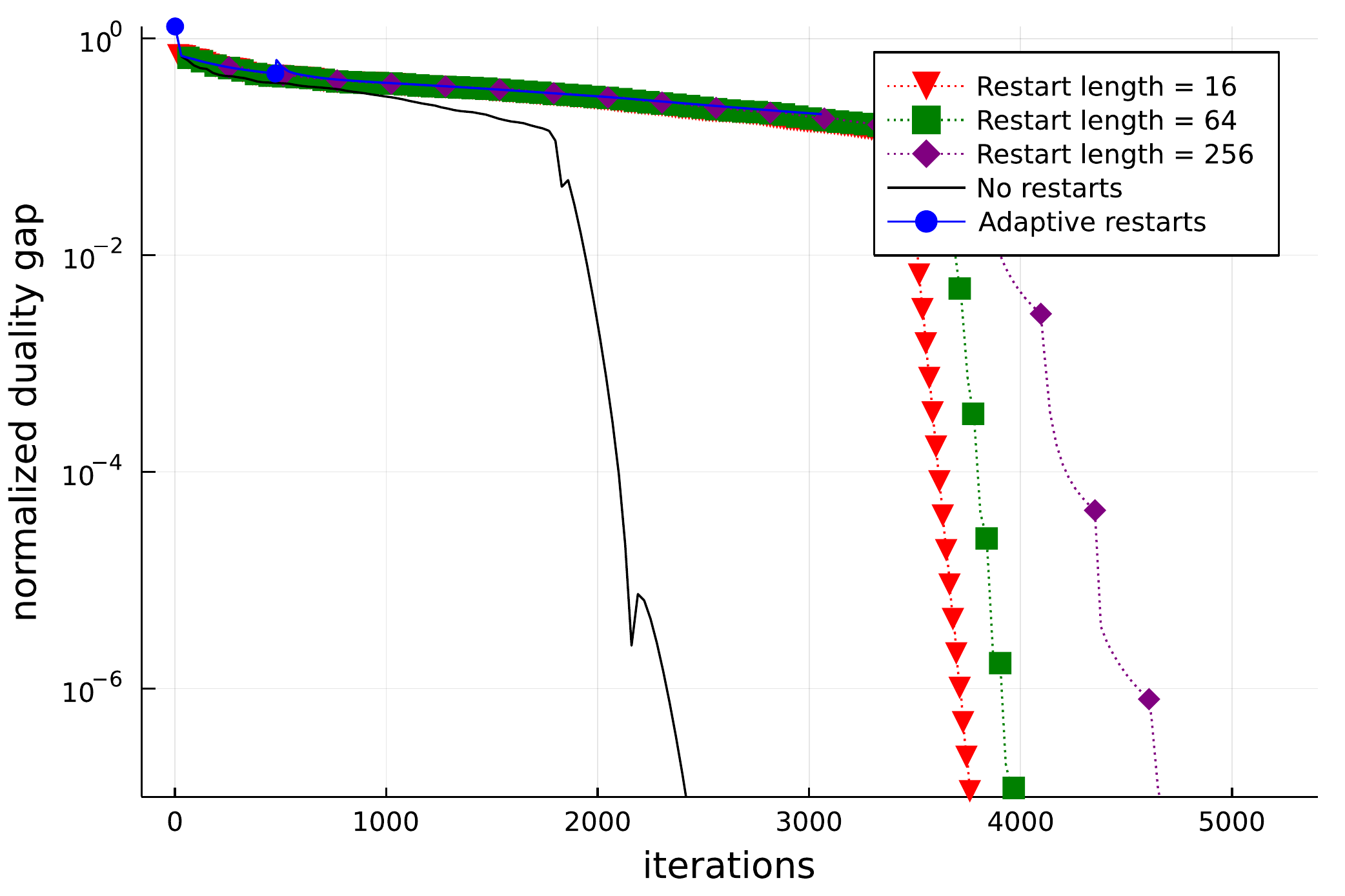}
    \includegraphics[height=150pt]{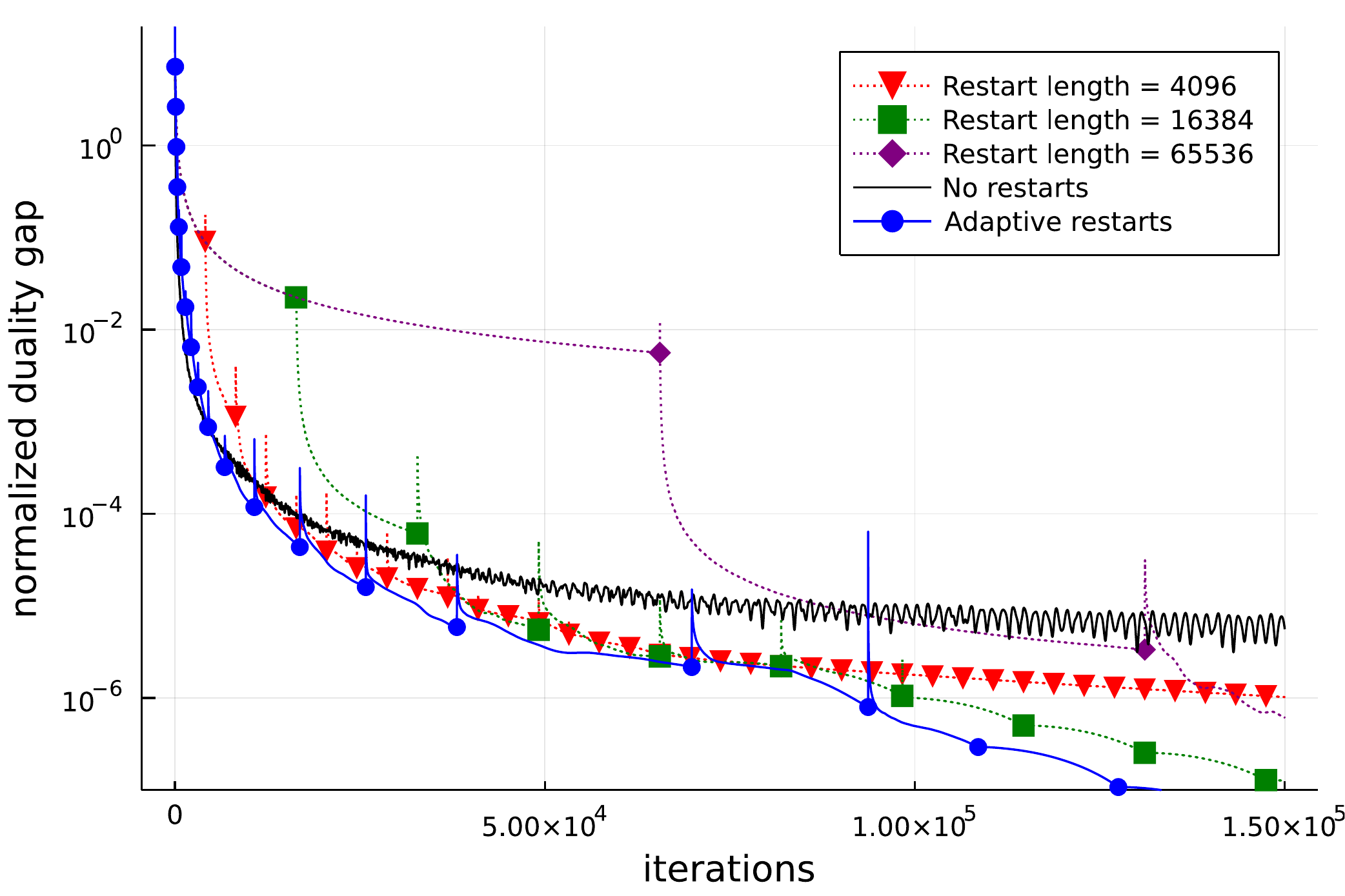} \\
    \caption{Plots show normalized duality gap (in $\log$ scale) versus number of iterations for (restarted) ADMM.  Results from left to right, top to bottom for qap10, qap15, nug08-3rd, and nug20. Each plot compares no restarts, adaptive restarts, and the three best fixed restart lengths.
    For the restart schemes we evaluate the normalized duality gap at the average, for no restarts we evaluate it at the last iterate as the average performs much worse.
    A star indicates the last iteration that the active set changed before the iteration limit was reached.
    Series markers indicate when restarts occur.}
    \label{fig:admm-experiments}
\end{figure}

Figure \ref{fig:admm-experiments} shows the performance of restarted ADMM for the same problems, which verifies again the improved performance of restarted algorithms. We observe that
restarts do appear
to make performance slower for nug08-3rd but this appears to be a very easy problem --- the number of iterations for all methods is very low.
}

\end{document}